%% file: square_tiled_equi_counting_super_final_2.tex
\documentclass{amsart}
\usepackage{amscd,amssymb}
\usepackage{amsmath,amsthm,amsfonts,srcltx}
\usepackage{enumitem}
\numberwithin{equation}{section}

\newcommand{\diag}{\textrm{diag}}
\newcommand{\nonvar}{\textrm{nv}}
\newcommand{\hyp}{\textrm{hyp}}
\newcommand{\EO}{\textrm{EO}}
\newcommand{\gzero}{\textrm{g0}}

\newcommand{\G}{\Gamma}

\newcommand{\e}{{e}}
\newcommand{\nofint}{n} % number of sides of the polygon which forms the boundary of the cylinder

\newlength{\halfbls}
\newcommand{\mydummy}{\phantom{0.4854}}

\newcommand{\noz}{r}     % number of zeroes of Abelian differential
\newcommand{\mult}{\mu}  % \mult_i is the multiplicity of a zero or of a pole
\newcommand{\prop}{p}    % proportion of given number of cylinders e.g. \prop_k(\cH(3,1))

\renewcommand{\epsilon}{\varepsilon}
\renewcommand{\Re}{\operatorname{Re}}
\renewcommand{\Im}{\operatorname{Im}}

\newcommand{\SLZ}{\operatorname{SL}(2,{\mathbb Z})}
\newcommand{\SLR}{\operatorname{SL}(2,{\mathbb R})}
\newcommand{\GLR}{\operatorname{GL}(2,{\mathbb R})}
\newcommand{\Uh}{\operatorname{U}_{hor}}
\newcommand{\Uv}{\operatorname{U}_{vert}}

\newcommand{\CP}{{\mathbb C}\!\operatorname{P}^1}

\newcommand\C{\mathbb C}
\newcommand\N{\mathbb N}
\newcommand\Q{\mathbb Q}
\newcommand\R{\mathbb R}
\newcommand\T{\mathbb T}
\newcommand\ZZ{{\mathbb Z}/2{\mathbb Z}}
\newcommand\Z{\mathbb Z}

\newcommand{\cD}{{\mathcal D}}
\newcommand{\cH}{{\mathcal H}}
\newcommand{\cL}{{\mathcal L}}

\newcommand{\cN}{{\mathcal N}}
\newcommand{\cQ}{{\mathcal Q}}
\newcommand{\cS}{{\mathcal S}}

\newcommand{\Area}{\operatorname{Area}}
\newcommand{\card}{\operatorname{card}}
\newcommand{\Vol}{\operatorname{Vol}}
\newcommand{\vol}{\operatorname{vol}}

\newtheorem{Theorem}{Theorem}[section]
\newtheorem{thm}[Theorem]{Theorem}
\newtheorem{condTheorem}[Theorem]{Conditional Theorem}

\newtheorem*{NoNumberTheorem}{Theorem}

\newtheorem{Proposition}[Theorem]{Proposition}

\newtheorem{Lemma}[Theorem]{Lemma}

\newtheorem*{NNLemma}{Lemma}
\newtheorem{Conjecture}[Theorem]{Conjecture}
\newtheorem{Question}[Theorem]{Question}

\newtheorem{Corollary}[Theorem]{Corollary}
\newtheorem{condCorollary}[Theorem]{Conditional Corollary}

\newtheorem*{Problem}{Problem}

\theoremstyle{definition}
\newtheorem{Definition}[Theorem]{Definition}

\newtheorem{Convention}[Theorem]{Convention}
\newtheorem*{NNConvention}{Convention}

\theoremstyle{remark}
\newtheorem{Remark}[Theorem]{Remark}
\newtheorem{rem}[Theorem]{Remark}
\newtheorem{Example}[Theorem]{Example}

\allowdisplaybreaks

\date{December 26, 2016}

\begin{document}
%\setstcolor{red}

\title[Square-tiled surfaces of fixed combinatorial type]
{Square-tiled surfaces of fixed combinatorial type:
equidistribution, counting, volumes of the ambient strata}

% First author
\author[V.~Delecroix]{Vincent Delecroix}
\address{
LaBRI,
Domaine universitaire,
351 cours de la Lib\'eration, 33405 Talence, FRANCE
}
\email{20100.delecroix@gmail.com}

% Second author
\author[\'E.~Goujard]{\'Elise Goujard}
\address{
Laboratoire de Math\'ematiques d'Orsay,
 Univ. Paris-Sud, CNRS,
  Universit\'e Paris-Saclay,
  91405 Orsay, FRANCE}
\email{elise.goujard@gmail.com}

% Third author
\author[P.~G.~Zograf]{Peter~Zograf}
\thanks{Research of Section~\ref{s:Alternative:counting} is supported by the RScF grant 16-11-10039.}
\address{
St.~Petersburg Department, Steklov Math. Institute, Fontanka 27,
St. Petersburg 191023, and Chebyshev Laboratory,
St. Petersburg State University, 14th
Line V.O. 29B, St.Petersburg 199178 Russia}
\email{zograf@pdmi.ras.ru}

% Fourth author
\author[A.~Zorich]{Anton Zorich}
\thanks{Research of the fourth author is partially supported by IUF}
\address{
Institut Universitaire de France;
Institut de Math\'ematiques de Jussieu --
Paris Rive Gauche,
UMR7586,
B\^atiment Sophie Germain,
Case 7012,
75205 PARIS Cedex 13, France}
\email{anton.zorich@imj-prg.fr}

\begin{abstract}
We  prove  that  square-tiled  surfaces having fixed combinatorics
of horizontal cylinder decomposition and tiled with smaller and
smaller squares become asymptotically equidistributed in any
ambient linear $\GLR$-invariant suborbifold defined over $\Q$ in
the moduli space of Abelian differentials. Moreover,
we prove that the combinatorics of the horizontal and of the
vertical decompositions are asymptotically uncorrelated.
As a consequence, we prove the existence
of an asymptotic distribution for the combinatorics
of a ``random'' interval exchange transformation with integer
lengths.

We compute explicitly the \textit{absolute} contribution of square-tiled
surfaces having  a  single  horizontal cylinder  to  the Masur--Veech
volume of any ambient stratum of Abelian differentials. The resulting
count  is  particularly  simple  and efficient in the large
genus asymptotics. We  conjecture that the corresponding \textit{relative}  contribution is
asymptotically of  the  order  $1/d$,  where  $d$ is the dimension of
the stratum, and prove  that  this  conjecture  is  equivalent  to
the  long-standing conjecture on the large genus asymptotics of the
Masur--Veech volumes. We prove, in particular, that the recent
results of Chen, M\"oller and Zagier imply that the conjecture holds
for the principal stratum of Abelian differentials as the genus tends to infinity.

Our result on random interval exchanges with integer
lengths allows to make empirical computation of
the probability to get a $1$-cylinder pillowcase cover taking a
``random'' one in a given stratum. We use this technique to derive
the approximate values  of  the Masur--Veech volumes of strata of
quadratic differentials of all small dimensions.
\end{abstract}

\maketitle
\tableofcontents

%####################################################################
%####################################################################
%####################################################################
\section*{Introduction}

\noindent\textbf{Siegel--Veech constants and Masur--Veech volumes.}
One  of the most powerful tools in the study of billiards in rational
polygons  (including  ``wind-tree'' billiards with periodic obstacles
in  the  plane),  of interval exchange transformations
and   of  measured  foliations  on  surfaces  is
renormalization.   In  particular,  to  describe  delicate  geometric
and   dynamical   properties   of   the   initial  billiard, interval  exchange
transformation or  measured  foliation,  one  has
to  find  the  $\GLR$-orbit  closure  of  the  associated translation
surface   in   the   moduli   space   of   Abelian   (or   quadratic)
differentials,  and  to  study  its  geometry. This approach,
initiated by H.~Masur and W.~Veech four decades ago became
particularly  powerful recently due  to  the breakthrough rigidity
theorems  of Eskin--Mirzakhani--Mohammadi~\cite{Eskin:Mirzakhani}
and~\cite{Eskin:Mirzakhani:Mohammadi}.

The moduli space of Abelian (or quadratic) differentials is
stratified by the degrees of zeroes of the Abelian (or quadratic)
differential. Each stratum is endowed with a natural measure, the
\textit{Masur--Veech} measure, that is preserved by the
$\SLR$-action. Moreover, in each connected component of a stratum
this $\SLR$-action is ergodic
(after restriction to any
hypersurface of Abelian differentials defining
flat surfaces of constant area).
In many important
situations the $\SLR$-orbit closure of a translation
surface is  an  entire  connected component of a  stratum.  In  order
to count the growth rate  for  the number of closed geodesics on a
translation surface as in~\cite{Eskin:Masur},   or   in  order  to
describe  the  deviation spectrum  of  a  measured  foliation  as
in~\cite{Forni:Deviation} , \cite{Zorich:How:do},  or  to count the
diffusion rate of a wind-tree as in~\cite{Delecroix:Hubert:Lelievre},
\cite{Delecroix:Zorich}, one has to    compute    the
corresponding    \textit{Siegel--Veech constants},
see~\cite{Veech:SV}, and the Lyapunov exponents of the Hodge bundle
over the connected component of stratum. Both   quantities  are
expressed  by  explicit combinatorial   formulas   in   terms   of
the  Masur--Veech volumes     of     the     strata,
see~\cite{Eskin:Masur:Zorich}, \cite{EKZ}, \cite{AEZ:genus:0},
\cite{Goujard:volumes}.
\medskip

%--------------------------------------------------------------------
\noindent\textbf{Equidistribution Theorem.} The Masur--Veech volumes
of strata  of Abelian differentials and of meromorphic quadratic
differentials with at most simple poles were  computed
in~\cite{Eskin:Okounkov:Inventiones},
\cite{Eskin:Okounkov:pillowcase},                                 and
\cite{Eskin:Okounkov:Pandharipande}. The underlying idea (see
also~\cite{Zorich:square:tiled}) was a computation of the asymptotic
number of ``integer points'' (the ones having coordinates in $\Z+i\Z$
in period coordinates) in appropriate bounded domains exhausting the
stratum. Such integer points are represented by ``square-tiled''
surfaces in the strata of Abelian differentials, and by ``pillowcase
covers'' for the strata of quadratic differentials. The square-tiled
surfaces are ramified covers over the standard torus with
all ramification points located over a single point of the torus. The
pillowcase covers are covers over $\CP$ ramified over four points
such that all ramifications  over three out of the four points are of
order $2$.

Similarly, points of an $\epsilon$-grid in period coordinates of a
stratum correspond to square-tiled surfaces (respectively pillowcase
covers for strata of quadratic differentials) tiled with
$\epsilon\times\epsilon$-squares. Each square-tiled surface carries
interesting combinatorial geometry, for example, the decomposition
into maximal flat horizontal cylinders.

We  prove  in Theorem~\ref{th:equidistribution} of
section~\ref{s:Main:results} that  square-tiled
surfaces having fixed combinatorics of horizontal cylinder
decomposition and tiled with smaller and smaller squares become
asymptotically equidistributed in the ambient stratum. Actually, we
prove that this equidistribution theorem
holds not only for strata but for any $\GLR$-invariant suborbifold
that contains a single square-tiled surface\footnote{Or
equivalently, over the $\GLR$-invariant suborbifolds defined over $\Q$; see
section~\ref{s:Main:results} or \cite{Wright:field:of:def}.}.
This means that taking a tiny $\epsilon$-grid in an open domain $U$
of finite volume and taking a random point of this grid, we get a
square-tiled surface having given combinatorics of horizontal
cylinder decomposition with probability which asymptotically (as
$\epsilon\to 0$) does not depend on $U$.

The Equidistribution Theorem gives sense to the notion of
(asymptotic) probability $p_k$ for a ``random'' square-tiled surface
in a given stratum to have a fixed number $k \in \{1, 2, \dots, g+\noz-1\}$ of maximal
cylinders in its horizontal or vertical cylinder decomposition, where
$g$ is the genus of the surface and $r$ is the number of conical
singularities (ramification points). We show that the corresponding
probabilities for horizontal and vertical cylinder decompositions are
uncorrelated.

An interval exchange transformation is called \textit{rational}
if all its intervals under exchange have rational lengths.
We obtain in Theorem~\ref{th:iet} an analogous equidistribution
statement for rational interval exchange transformations.
The probabilities $p_k$ that appear in this
context are the same as the ones
for cylinder decompositions of square-tiled surfaces.
It allows us to give sense to the notion of (asymptotic) probability $p_k$ for a
``random'' rational interval exchange transformation with a given
permutation to have $k$  bands of isomorphic closed trajectories.
\medskip

%--------------------------------------------------------------------
\noindent\textbf{Contribution of $1$-cylinder square-tiled surfaces
and large genus asymptotics of Masur--Veech volumes.}
The only currently known approach to compute
Masur--Veech volumes of strata of Abelian differentials is based on
counting square-tiled surfaces. In
section~\ref{s:contribution:of:1:cylinder} we compute the
\textit{absolute}
contribution $c_1(\cH)$ of $1$-cylinder square-tiled surfaces to the
Masur--Veech volume of a stratum $\cH$,
where $c_1(\cH) := p_1(\cH) \cdot \Vol\cH$.
We define similarly $c_k(\cH)$, and by definition $\Vol\cH = c_1(\cH)
+ c_2(\cH) + \ldots + c_{g+r-1}(\cH)$. We give simple close exact
formulas for the contribution $c_1(\cH)$ to the volumes
$\Vol\cH(2g-2)$ and $\Vol\cH(1,\dots,1)$ of minimal and principal
strata of Abelian differentials. We also provide sharp upper and
lower bounds for contributions of $1$-cylinder square-tiled surfaces
to the Masur--Veech volumes of any stratum. The ratio of the upper
and lower bounds tends to $1$ as $g\to+\infty$ uniformly for all
strata in genus $g$, so the bounds are particularly efficient in
large genus asymptotics.

We  conjecture that the corresponding \textit{relative}  contribution
$p_1(\cH)$ of $1$-cylinder square-tiled surfaces to the Masur--Veech
volume $\Vol\cH$ of any stratum $\cH$ of Abelian differentials is
asymptotically of  the  order  $1/d$  as $g$ (equivalently $d$) tends
to infinity. Here $d$ is the dimension $d=\dim_\C(\cH)$ of the
stratum $\cH$.

We prove  that  this  conjecture  is  equivalent  to  the
long-standing conjecture~\cite{Eskin:Zorich} on the large genus
asymptotics of the Masur--Veech volumes of strata of Abelian
differentials. We prove that the recent
results~\cite{Chen:Moeller:Zagier} of D.~Chen, M.~M\"oller and
D.~Zagier imply that the conjecture holds for the principal stratum
of Abelian differentials as genus tends to infinity.
\medskip

%--------------------------------------------------------------------
\noindent\textbf{Siegel--Veech constants and Masur--Veech volumes of
strata of meromorphic quadratic differentials.}
We indicated that the Masur--Veech  volumes  were  computed
in~\cite{Eskin:Okounkov:Inventiones},
\cite{Eskin:Okounkov:pillowcase},                                 and
\cite{Eskin:Okounkov:Pandharipande}. In   particular,  it is proved
in these papers that  the volume of every connected component of
every  stratum in genus $g$ has  the  form  $r\cdot\pi^{2g}$,  where
$r$ is some rational  number. The generating function
in~\cite{Eskin:Okounkov:Inventiones}   was  translated  by A.~Eskin
into  a  very  efficient  computer  code,  which allowed to evaluate
explicitly volumes of all connected components of all strata of
Abelian differentials in genera up to $g=10$ (that is, to compute
explicitly  the  corresponding  rational  numbers  $r$), and for some
strata  up to $g=60$. Recent results~\cite{Chen:Moeller:Zagier} allow
to compute $r$ for the principal stratum up to genus $g=2000$ and
higher.

The  approach  elaborated by A.~Eskin, A.~Okounkov, R.~Pandharipande
in~\cite{Eskin:Okounkov:pillowcase},   and
in~\cite{Eskin:Okounkov:Pandharipande} to the computation  of the
Masur--Veech volumes of the strata of \textit{quadratic}
differentials had to wait for another decade to  be translated into
tables of numbers. One of the reasons for such a delay  is  a  more
involved  combinatorics  and multitude of various conventions  and
normalizations  required in volume computations (which is a common
source of mistakes in normalization factors like powers
of  $2$). This is why it is necessary to test theoretical predictions on some  table
of volumes obtained by an independent method. In the case of  Abelian
differentials,  the  volumes  of several low-dimensional strata  were
computed by a direct combinatorial method elaborated by A.~Eskin,
M.~Kontsevich,  and  one  of the authors; this approach is described
in~\cite{Zorich:square:tiled}. Another, even more reliable test  was
provided by computer simulations of Lyapunov exponents and their ties
with volumes through Siegel--Veech constants. In  the case of
quadratic differentials, explicit values of volumes of the   strata
in   genus   zero   were   conjectured by M.~Kontsevich about fifteen
years ago. The conjecture was proved in recent
papers~\cite{AEZ:Dedicata}  and~\cite{AEZ:genus:0}.  Further explicit
values  of volumes of all low-dimensional strata up to dimension $11$
were obtained in~\cite{Goujard:volumes}.

Appendices~\ref{s:contibution:of:diag:for:two:lattices}--\ref{s:tables:volumes}
describe   an   approach   to  the evaluation  of \textit{approximate}
values  of  volumes  of  several  dozens of low dimensional   strata.
Our   approach relies on the Equidistribution Theorem. The idea is
to evaluate   experimentally    the approximate value of the
probability $p_1(\cQ)$ to get a $1$-cylinder pillowcase  cover taking
a ``random'' pillowcase cover in a given stratum $\cQ$ of quadratic
differentials. Then we compute rigorously the absolute contribution
$c_1(\cQ)$ of $1$-cylinder pillowcase covers to the Masur--Veech
volume $\Vol\cQ$ of the stratum.
Relation $c_1(\cQ)= p_1(\cQ) \cdot \Vol\cQ$
now provides the
approximate value  of  the Masur--Veech volume $\Vol\cQ$ of the
stratum $\cQ$ of quadratic differentials.

This approach  is  completely  independent  of  the  one  of
A.~Eskin and A.~Okounkov  based on the representation theory. The
approximate data obtained  in this paper were used for ``debugging''
rigorous formulas in~\cite{Goujard:SV} and~\cite{Goujard:volumes}.

The  fact  that  our experimental results match theoretical
ones       in~\cite{AEZ:Dedicata},       \cite{AEZ:genus:0},      and
in~\cite{Goujard:volumes}, and that the induced theoretical values of
Siegel--Veech    constants    obtained   in~\cite{Goujard:SV}   match
independent computer experiments evaluating the Lyapunov exponents of
the Hodge bundle over the Teichm\"uller geodesic flow, as well as the
exact  values  of  the  sums  of  such  Lyapunov  exponents  computed
in~\cite{Chen:Moeller}  for  the  non-varying  strata  provides  some
reliable   evidence  that  the  nightmare  of  various  combinatorial
conventions    leads,   nevertheless,   to   correct   and   coherent
general formulas presented in~\cite{Goujard:SV} and
in~\cite{Goujard:volumes}.
\medskip

%----------------------------------------------------------------------
\noindent\textbf{Structure of the paper.}
Section~\ref{s:Main:results} is devoted to equidistribution and
section~\ref{s:contribution:of:1:cylinder} studies the contribution
of $1$-cylinder square-tiled surfaces (pillowcase covers) to the
Masur--Veech volumes of the strata. The two sections are, basically,
independent with an exception for a short
section~\ref{ss:Application:experimental:evaluation:of:MV:volumes}
describing the experimental approach to the computation of Masur--Veech
volumes based on the combination of results of the two sections.

Section~\ref{s:Alternative:counting} is independent of the first two:
it presents alternative approaches to counting $1$-cylinder
square-tiled surfaces and pillowcase covers based on recursive
relations (section~\ref{ss:recursive:relations}) and on construction of
the Rauzy diagrams (section~\ref{ss:Rauzy:classes}).

To make the paper self-contained we added
Appendix~\ref{s:overview:of:MV:volumes} suggesting necessary basic
information on the Masur--Veech volumes. It provides
information dispersed through several research papers and might be
useful for better understanding of any of the preceding
sections.

The content of Appendix~\ref{s:contibution:of:diag:for:two:lattices}
was isolated to avoid overloading the main body of the paper. It
describes certain subtlety related to normalization of the Masur--Veech
volumes which is not visible in quantitative considerations, but which is
relevant and non-trivial in the context of the current paper.

Appendix~\ref{s:tables:volumes} presents tables of the
Masur--Veech volumes of low-dimensional strata in the moduli spaces
of meromorphic quadratic differentials with at most simple poles
obtained by the method combining the equidistribution and the counting
results from this paper (the method is described in
section~\ref{ss:Application:experimental:evaluation:of:MV:volumes}).
The tables compare approximate values of volumes with the exact ones
indicating in each case the method of the exact computation.

\medskip

%----------------------------------------------------------------------
\noindent\textbf{Acknowledgements.}
We thank C.~Matheus, L.~Monin and P.~Pushkar Jr. for numerous
valuable conversations. We are grateful to A.~Eskin for important
discussions, and for his suggestion of the proof of
Lemma~\ref{lm:Moore}.

%####################################################################
%####################################################################
%####################################################################

\section{Equidistribution}
\label{s:Main:results}

%--------------------------------------------------------------------
\subsection{Definition of an essential invariant lattice subset}

Recall  that  any  stratum $\cH(m_1,\dots,m_\noz)$ in  the moduli
space of Abelian differentials is modelled on the relative
cohomology  $H^1(S,\{P_1,\dots,P_\noz\};\C)$,  where $S$ is the
underlying  topological surface and $\{P_1,\dots,P_r\}$ is a finite
collection   of  zeroes  of  an  Abelian  differential. Square-tiled
surfaces  tiled with unit squares correspond to ``integer points'' in
the stratum: they are represented by  the  points  of  the lattice
$H^1(S,\{P_1,\dots,P_\noz\};\Z\oplus i\Z)$ in period coordinates.

Denote by $\operatorname{P}\subset\SLR$ the subgroup of
upper-triangular matrices. Let $\cL$ be a suborbifold in the ambient
stratum supporting a finite $\operatorname{P}$-invariant ergodic
measure. By the fundamental results~\cite{Eskin:Mirzakhani},
\cite{Eskin:Mirzakhani:Mohammadi} of Eskin, Mirzakhani, Mohammadi,
$\cL$ is represented in period coordinates as the complexification of
a linear subspace $L\subset H^1(S,\{P_1,\dots,P_\noz\};\R)$. We
denote by $d$ the dimension of $L$, which is also the complex
dimension of $\cL$. In the current paper we always assume that the
linear subspace $L$ is  defined by a system of linear equations with
rational coefficients for some basis in
$H^1(S,\{P_1,\dots,P_\noz\};\Z)\subset
H^1(S,\{P_1,\dots,P_\noz\};\R)$, or, equivalently, that the invariant
suborbifold $\cL$ is defined over $\Q$ (see
\cite{Wright:field:of:def} for the notion of the \textit{field of
definition} of a $\GLR$-invariant suborbifold). For example, any
connected component of a stratum or a $\GLR$-orbit of any
square-tiled surface is defined over $\Q$. By a Theorem of
A.~Wright~\cite{Wright:field:of:def}, any connected $\GLR$-invariant
suborbifold containing a single square-tiled surface is defined over
$\Q$.

The rationality assumption implies that $L\cap
H^1(S,\{P_1,\dots,P_\noz\};\Z)$ forms a $d$-dimensional lattice in
$L$ and thus defines a volume element in the vector space $L$ by the
condition that the volume of a fundamental domain in the induced
integer lattice in $L$ is equal to one. We denote by $d\nu$ the
induced volume element in $\cL$.

Consider the following subgroups of $\SLZ$:
$$
\Uh=\left\{\begin{pmatrix} 1&n\\0&1\end{pmatrix}\right\},\
\Uv=\left\{\begin{pmatrix} 1&0\\n&1\end{pmatrix}\right\},\
\text{ where }n\in\Z\,.
$$

\begin{Definition}
\label{def:essential:lattice:subset}
An  \textit{essential  $\Uh$-invariant
(respectively $\Uv$-invariant)
lattice subset} in a closed connected linear $\GLR$-invariant
orbifold $\cL$ defined over $\Q$ is a subset $\cD_\Z\subset\cL$
of  square-tiled surfaces in $\cL$ (tiled with \textit{unit} squares),
satisfying  the following two properties:
\begin{enumerate}[label=\alph*)]
\item The set $\cD_\Z$ is invariant under the action of
      the subgroup $\Uh$ (respectively $\Uv$);
\item
The following limit exists and is strictly positive:
\begin{equation}
\label{eq:limit}
2d\cdot \lim_{a\to+\infty} \frac{1}{a^{d}} \cdot
\card\{S\in \cD_\Z\,\big|\, \Area(S) \le a\} = c(\cD_\Z) >0\,,
\end{equation}
where $d=\dim_\C\cL$.
\end{enumerate}
\end{Definition}

\begin{Remark}
The important part of condition~\eqref{eq:limit} is the
\textit{existence} of the limit. When the limit exists but is zero, the statements
formulated below stay valid, but become trivial.

The normalization factor $2d$ is chosen by esthetic reasons; it is coherent
with the traditional normalization of the Masur--Veech volume and makes
numerous formula less bulky.
\end{Remark}

One can also study \textit{weighted} analogs of essential lattice
subsets associating to each square-tiled surface some $\Uh$-invariant
(respectively $\Uv$-invariant) weight (like the area of one of the
maximal cylinder divided by the total area of the square-tiled
surface).

Recall  that  $\cL_1$ denotes the real hypersurface in $\cL$
of  those  pairs  $(C,\omega)$ in $\cL$
for  which  the  area  defined  by  the Abelian differential $\omega$
equals  one,  $\frac{i}{2}\int_C  \omega\wedge\bar\omega=1$.  The cone
$C_R X\subset\cH(m_1,\dots,m_\noz)$ over a subset
$X\subset\cL$ is defined as
$$
C_R X:=\{(C,r\cdot\omega)\,|\,(C,\omega)\in X,\ 0<r\le R\}\,.
$$
Geometrically the flat surface $r\cdot S=(C,r\cdot\omega)$ is
obtained from the flat surface $S=(C,\omega)$ by applying homothety
with coefficient $r$:
$$
(C,r\cdot\omega)=\begin{pmatrix}r&0\\0&r\end{pmatrix}\cdot S\,.
$$
Note that $\Area(r\cdot S)=r^2\Area(S)$. By
$$
C_\infty X:=\{(C,r\cdot\omega)\,|\,(C,\omega)\in X,\ 0<r\}
$$
we denote the cone over $X$ with no restrictions on the scaling
factor.

By definition, the \textit{Masur--Veech volume} $\nu_1(X_1)$
of a subset $X_1\subset\cL_1$
is defined as the $\nu$-volume of the ``unit cone'' over $X_1$
normalized by a dimensional factor:
\begin{equation}
\label{eq:nu1}
\nu_1(X_1):=2d\cdot \nu(C_1 X_1)\,,
\quad\text{where }d=\dim_\C\cL\,.
\end{equation}
Another fundamental result of Eskin, Mirzakhani,
Mohammadi~\cite{Eskin:Mirzakhani}, \cite{Eskin:Mirzakhani:Mohammadi}
states that rescaling the initial finite ergodic measure on $\cL_1$
by an appropriate constant factor we get the Masur--Veech volume
element $d\nu_1$. In particular, $\nu_1(\cL_1)$ is finite.
(Finiteness of Masur--Veech volumes of the strata was proved much
earlier independently by H.~Masur in~\cite{Masur:82} and by W.~Veech
in~\cite{Veech:Gauss:measures}.)

We denote by $\cL_\Z$ the set of all points in $\cL$ represented by
the points of the lattice $H^1(S,\{P_1,\dots,P_\noz\};\Z\oplus i\Z)$
in period coordinates. Finiteness of the Masur--Veech volume
$\nu_1(\cL_1)$ implies that for any linear invariant orbifold $\cL$
defined over $\Q$ the set $\cL_\Z$ is an essential lattice subset
invariant under both $\Uh$ and $\Uv$. The invariance is obvious while
existence and finiteness of the limit~\eqref{eq:limit} follows from
the definition of the Masur--Veech volume: the limit $c(\cL_\Z)$
coincides with the volume of $\cL_1$:
$$
c(\cL_\Z)={\nu_1(\cL_1)}\,.
$$
Given any essential lattice subset $\cD_\Z$ in $\cL$, we also
define a ``probability'' $\prop(\cD_\Z)$ by
\begin{multline}
\label{eq:prop}
\prop(\cD_\Z) =
\lim_{a\to+\infty}
\frac{\card\{S\in \cD_\Z\,\big|\, \Area(S) \le a\}}
{\card\{S\in \cL_\Z\,\big|\, \Area(S) \le a\}}
=
\\
=\frac{c(\cD_\Z)}{c(\cL_\Z)}
=\frac{c(\cD_\Z)}{\nu_1(\cL_1)}
=\frac{c(\cD_\Z)}{2d\cdot \nu(C_1 \cL_1)}\,.
\end{multline}
The quantity $\prop(\cD_\Z)$ gives the asymptotic proportion of the
square-tiled surfaces which form the subset $\cD_\Z$ among all
square-tiled surfaces $\cL_\Z$ in $\cL$ as the number of squares of
tiling tends to infinity. It represents the probability to get a
square-tiled surface in $\cD_\Z$ taking a random square-tiled surface
in $\cL_\Z$.

%--------------------------------------------------------------------
\subsection{Existence of essential invariant lattice subsets}

Our first result shows that the $\GLR$-invariant suborbifolds of the
strata in the moduli space of Abelian differentials or meromorphic
quadratic differentials with at most simple poles defined over $\Q$
contain many interesting $\Uh$ (respectively $\Uv$) invariant
essential lattice subsets. Note that every square-tiled surface has a
decomposition into maximal horizontal (respectively vertical)
cylinders. One can count separately the numbers of $1$-cylinder,
$2$-cylinder, etc $k$-cylinder square-tiled surfaces tiled with at
most $N$ squares. Clearly, the number of maximal horizontal
(respectively vertical) cylinders is invariant under the action of
$\Uh$ (respectively of $\Uv$). One can go further and fix the
combinatorics of the way these maximal cylinders are attached to each
other along horizontal (respectively vertical) saddle connections;
this combinatorics is described by the associated \textit{critical
graph} called also \textit{separatrix diagram}. To make our paper
self-contained we recall the formal definition of a separatrix
diagram in section~\ref{ss:separatrix:diagrams}.

A separatrix diagram $\cD$ is called ``realizable'' in $\cL$ if there
exists a surface in $\cL$ {with periodic
horizontal foliation having a decomposition into maximal horizontal
cylinders given by the diagram $\cD$.

\begin{Theorem}
\label{th:sep:diagram:is:essential}
Let $\cL$ be a $\GLR$-invariant suborbifold of Abelian differentials
(or quadratic differentials with at most simple poles) defined over
$\Q$. Let $\cD$ be a separatrix diagram realizable in $\cL$. Then the
set $\cD_\Z(\cL)$   of  all  square-tiled  surfaces in $\cL$
(respectively pillowcase  covers in $\cL$ for the strata of quadratic
differentials) sharing  the diagram  $\cD$  is  an  essential lattice
subset in $\cL$.
\end{Theorem}

Theorem~\ref{th:sep:diagram:is:essential} is proved in
section~\ref{ss:Invariance:along:Re:and:Im:foliations}.

By a result of J.~Smillie square-tiled surfaces in a stratum
$\cH(m_1,\dots,m_\noz)$ can have from $1$ to $g+\noz-1$ maximal
cylinders, and any value in this range is attained by some
square-tiled surface in the stratum. Since a  finite  union  of
essential  lattice subsets is also an essential lattice  subset and
since  the  collection  of  all separatrix diagrams corresponding
to      a     given     stratum     is     finite,
Theorem~\ref{th:sep:diagram:is:essential}  implies  that the set of
all square-tiled surfaces (respectively pillowcase covers) having
exactly $k$  cylinders  is  an  essential  lattice  subset in any
stratum for any $k$ in $\{1,\dots,g+\noz-1\}$. The corresponding
relative contributions $\prop_k(\cH^{comp}(m_1,\dots,m_\noz))$ of
$k$-cylinder square-tiled surfaces to the Masur--Veech volume of
connected components $\cH^{comp}(m_1,\dots,m_\noz)$ of the strata are
of particular interest to us.

%--------------------------------------------------------------------
\subsection{Equidistribution in the unit hyperboloid}

Our first result concerns equidistribution of integer points from any
essential lattice subset in the ``unit hyperboloid'' $\cL_1$ for any
linear invariant orbifold $\cL$ defined over $\Q$.

\begin{Theorem}
\label{th:equidistribution}
Let $\cD_\Z$ be a $\Uh$- or $\Uv$-invariant essential lattice subset
in some linear $\GLR$-invariant suborbifold $\cL$ defined over $\Q$
of some stratum of Abelian differentials. Let $d$ be  the  complex
dimension of $\cL$. Let $X_1$ be an open set in the ``unit
hyperboloid'' $\cL_1$.

Then the number
$$
\cN_\cD(X_1,a)=
\card\{S\in \cD_\Z\cap C_{\sqrt{a}} X_1\}
$$
of     square-tiled    surfaces $S$ in
$\cD_\Z$ tiled with $N\le a$ unit squares which project to $X_1$
under the natural projection $\cL\to\cL_1$
asymptotically   depends   only  on
$\cD_\Z$ and on the Masur--Veech measure $\nu_1(X_1)$ of $X_1$:
\begin{equation}
\label{eq:equidistribution1:hyperboloid}
2d\cdot \lim_{a\to+\infty} \frac{\cN_\cD(X_1,a)}{a^d}  =
c(\cD_\Z)\cdot
\frac{\nu_1(X_1)}{\nu_1(\cL_1)}\,,
\end{equation}
where  the  constant
$c(\cD_\Z)$
defined in~\eqref{eq:limit}
depends only on $\cD_\Z$.

Equivalently, for any bounded continuous function $f: \cL_1 \to \R$
one has
\begin{equation}
\label{eq:equidistribution2:hyperboloid}
2d\cdot\lim_{a \to + \infty} \frac{1}{a^d}
\sum_{S \in \cD_{\Z}\cap C_{\sqrt a} X_1}
f\left(\frac{S}{\sqrt{\Area(S)}}\right) =
\frac{c(\cD_\Z)}{\nu_1(\cL_1)} \int_{X_1} f d\nu_1.
\end{equation}
\end{Theorem}

We start the proof of Theorem~\ref{th:equidistribution} with the
following preparatory Lemma.

\begin{Lemma}
\label{lm:Moore}
Any finite $\operatorname{P}$-invariant ergodic measure $\nu_1$ on
any  stratum  of  Abelian differentials  is ergodic with respect to
the action of the discrete parabolic  subgroup
$\Uh\subset\operatorname{P}$ of  matrices  of  the form
$\begin{pmatrix} 1&n\\0&1\end{pmatrix}$ with $n\in\Z{}$.
\end{Lemma}

\begin{proof}
By   the  fundamental  Theorem  of A.~Eskin and
M.~Mirzakhani~\cite{Eskin:Mirzakhani} any finite
$\operatorname{P}$-invariant ergodic measure is ergodic with respect
to the $\SLR$-action.

Let $\operatorname{G}$  be  a simple  Lie  group,  $\operatorname{H}$
be a closed non-compact subgroup of $\operatorname{G}$ and let
$\operatorname{G}$-action be ergodic with respect to a finite
invariant measure. By a particular case of Moore's  Ergodicity
theorem (Theorem~2.2.15 in~\cite{Zimmer}) the
$\operatorname{H}$-action is also ergodic.

In our case the simple Lie group is $\SLR$ and the closed non-compact
subgroup $\operatorname{H}$ is $\Uh$.
\end{proof}

\begin{Remark}
Note that in the general statement of Moore's Ergodic Theorem, the
group $\operatorname{G}$ is a finite product of simple Lie groups
with finite center, and the ergodic $\operatorname{G}$-action is
supposed to be \textit{irreducible} (see Theorem~2.2.15 and
Definition~2.2.11 in~\cite{Zimmer}). However, for a \textit{simple}
Lie group $\operatorname{G}$ the requirement of irreducibility of the
action is satisfied automatically; see the remark after
Definition~2.2.11 in~\cite{Zimmer}.
\end{Remark}

\begin{proof}[Proof of Theorem~\ref{th:equidistribution}]
The     proof     mimics     the     proof     of     Theorem     6.4
in~\cite{Mirzakhani:growth:of:simple:geodesics}.

Fix    a   connected $\GLR$-invariant orbifold $\cL$ defined over
$\Q$. Any \textit{essential  lattice  subset}  $\cD_\Z$  defines  a
sequence of measures  $\mu^{(N,\cD)}$  on  $\cL_1$,  where $N\in\N$.
Namely, we put Dirac masses to all points represented by square-tiled
surfaces tiled with  at  most  $N$  unit  squares which  belong  to
$\cD_\Z$. Then we project  these  points from $\cL$ to $\cL_1$ by the
natural  projection,  and  normalize  the resulting measure    by
$2d\cdot N^{-d}$,   where   $d=\dim_{\C}\cL$.

Taking  \textit{all} square-tiled surfaces in $\cL$, and not only
those which belong to the subset  $\cD_\Z$, we  get  a  sequence of
measures which we denote by $\mu^{(N)}$  and  which  weakly  converges
to  our  canonical  invariant Masur--Veech  measure  $\nu_1$ on
$\cL_1$, see~\eqref{eq:nu1}.

By  definition,  for  any  essential  lattice subset $\cD_\Z$ we have
$\mu^{(N,\cD)}\le \mu^{(N)}$ for we take \textit{only part} of square-tiled
surfaces  of area at most $N$ to  define $\mu^{(N,\cD)}$ while we take
\textit{all} square-tiled  surface  of  area  at most $N$ to  define
$\mu^{(N)}$.  Since  the normalization factor $2d\cdot N^{-d}$ is the
same in both cases, we get the desired   inequality.   This   implies
that   for   any  open  ball $X\subset\cL_1$ we have
$$
\limsup_{N\to+\infty} \mu^{(N,\cD)}(X)\le
\limsup_{N\to+\infty} \mu^{(N)}(X) = \nu_1(X) < +\infty\,.
$$
and any   subsequence   of   $\{\mu^{(N,\cD)}\}$  contains
a weakly converging subsequence.

The  inequality  $\mu^{(N,\cD)}\le  \mu^{(N)}$
implies that any weak limit
$\mu_{J}$  of  $\{\mu^{(N,\cD)}\}_{N\in\N}$ is in the same Lebesgue class
as  the  Masur--Veech  measure  $\nu_1$,  that  is  for  any measurable
$V\subset\cL_1$ with $\nu_1(V)=0$, we have $\mu_{J}(V)=0$.

The  $\Uh$-invariance  of the essential lattice subset implies that
all measures  $\mu^{(N,\cD)}$ are $\Uh$-invariant. Hence, the measure
$\mu_{J}$ is also $\Uh$-invariant.

By Lemma~\ref{lm:Moore} the Masur--Veech measure $\nu_1$
is ergodic with respect to the action of $\Uh$.
Ergodicity  of $\nu_1$ with respect to the action of
$\Uh$,
invariance of $\mu_{J}$  under the action of
$\Uh$, and the fact that
$\mu_{J}$ is in the Lebesgue  measure  class  of  $\nu_1$  all
together imply that the two measures are proportional:
$$
\mu_{J} = k_J\cdot \nu_1\,.
$$

Finally,  equation~\eqref{eq:limit}  implies  that the coefficient of
proportionality  $k_J$  does  not  depend  on the subsequence $J$: it
equals  the  ratio $\cfrac{c(\cD_\Z)}{\nu_1(\cL_1)}$, where
$c(\cD_\Z)$ is the limit in~\eqref{eq:limit}.
Theorem~\ref{th:equidistribution} is proved.
\end{proof}

%--------------------------------------------------------------------
\subsection{Equidistribution in a linear invariant suborbifold defined over $\Q$}

In certain situations it is convenient to work with subsets $X$ of
the entire linear $\GLR$-invariant suborbifold $\cL$ defined over
$\Q$ instead of the subsets of the unit hyperboloid $\cL_1$. The
statement below provides a version of our equidistribution result
applied to spatial versus hypersurface domains.

Define $\cD_{\epsilon\Z}$ as the image of $\cD_\Z$ under the action of
the uniform contraction with the scaling factor $\epsilon$, namely
$$
\cD_{\epsilon\Z}:=
\begin{pmatrix}\epsilon&0\\0&\epsilon\end{pmatrix}\cD_\Z
$$
Given $X \subset \cL$ we can either apply a homothety
$\begin{pmatrix}\epsilon^{-1}&0\\0&\epsilon^{-1}\end{pmatrix}$
to $X$ and consider the intersection with an essential lattice $\cD_Z$.
Or, equivalently, intersect the original set $X$ with
the rescaled essential lattice subset $\cD_{\epsilon\Z}$.

\begin{Theorem}
\label{th:equidistribution:spatial}
Let $\cD_\Z$ be a $\Uh$- or $\Uv$-invariant essential lattice subset
in some linear $\GLR$-invariant suborbifold $\cL$ defined over $\Q$
of complex dimension $d$. Let $X$ be an open subset in $\cL$.
Then the number
$$
\widetilde\cN_\cD(X,\epsilon)=
\card\{S\in\cD_{\epsilon\Z}\cap X\}
$$
of     square-tiled    surfaces $S$ in $X$
tiled with tiny
$\epsilon\times\epsilon$-squares
that belong to $\cD_\Z$ after rescaling by $\frac{1}{\epsilon}$,
asymptotically   depends   only  on
$\cD_\Z$ and on the Masur--Veech volume $\nu(X)$ of $X$:
\begin{equation}
\label{eq:equidistribution1:spatial}
\lim_{\epsilon\to+0} \epsilon^{2d}\cdot \widetilde\cN_\cD(X,\epsilon)  =
\frac{c(\cD_\Z)}{\nu_1(\cL_1)}\cdot{\nu(X)}\,,
\end{equation}
where  the  constant
$c(\cD_\Z)$
defined in~\eqref{eq:limit}
depends only on $\cD_\Z$.

Equivalently, for any compactly supported or positive continuous function
$f: \cL \to \R$ one has:
\begin{equation}
\label{eq:equidistribution2:spatial}
\lim_{\epsilon \to +0} \epsilon^{2d}
\sum_{S \in \cD_{\epsilon\Z}\cap X}
f(S) = \frac{c(\cD_\Z)}{\nu_1(\cL_1)} \int_{X} f d\nu.
\end{equation}
\end{Theorem}
The statements of Theorem~\ref{th:equidistribution} and
Theorem~\ref{th:equidistribution:spatial} are rather different for
the simple reason that $\nu_1$ has finite mass while $\nu$ has not.
In particular, in the version above the quantities in the
equations~\eqref{eq:equidistribution1:spatial}
and~\eqref{eq:equidistribution2:spatial} might can be infinite (on both
sides of the equality).

\begin{proof}
The proof is a consequence of the previous Equidistribution
Theorem (Theorem~\ref{th:equidistribution}). Note that the number $\cN_\cD(X_1,a)$
    introduced in the
statement of Theorem~\ref{th:equidistribution}
and the number $\widetilde\cN_\cD(X,\epsilon)$ for $X=C_1 X_1$ and
for $\epsilon=\frac{1}{\sqrt a}$
used in the statement of Theorem~\ref{th:equidistribution:spatial}
coincide:
$$
\cN_\cD(X_1,a)=
\card\{S\in \cD_\Z\cap C_{\sqrt{a}} X_1\}=
\card\{S\in \cD_{\frac{1}{\sqrt a}\Z}\cap C_1 X_1\}=
\widetilde\cN_\cD(C_1 X_1,\frac{1}{\sqrt a})\,.
$$

\begin{figure}[htb]

\includegraphics{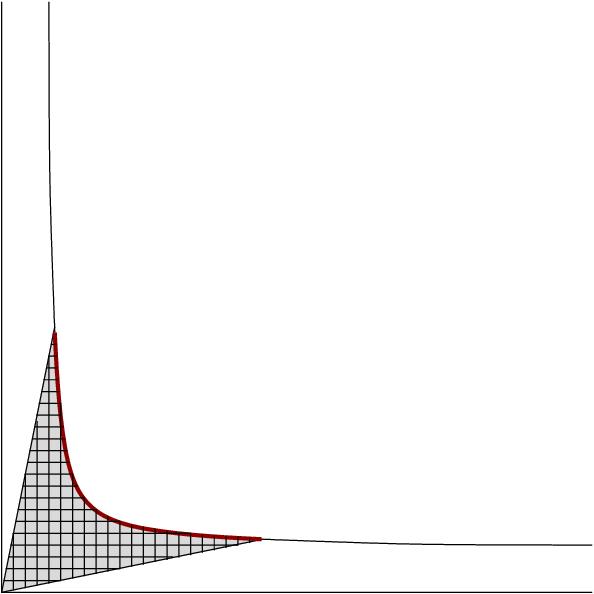}

\includegraphics{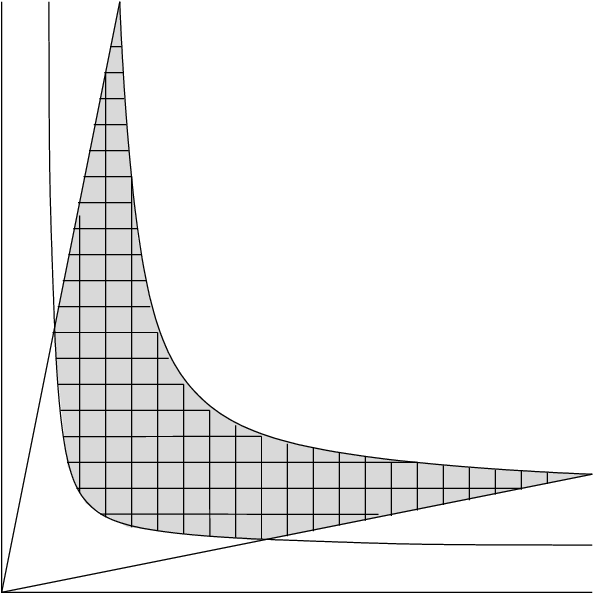}

\vspace{70pt}

\caption{
\label{fig:trapezoid}
Cone based on the ``unit hyperboloid'' on the left
and part of the cone confined between two ``hyperboloids''
on the right.
}
\end{figure}

Hence the particular case of
Theorem~\ref{th:equidistribution:spatial} when $X$ has the form of a
cone $X= C_1 X_1$ based on an open subset of the ``unit hyperboloid''
(see the left picture in Figure~\ref{fig:trapezoid}) is already
proved. By homogeneity, the latter statement is also valid for any
cone $X= C_r X_1$ with any $r>0$. Hence it is also valid for any
complement $X= C_R X_1- C_r X_1$ for any $R>r>0$ (see the right
picture in Figure~\ref{fig:trapezoid}). Following the lines of
Riemann integration, for any set $X$ as in the statement of
Theorem~\ref{th:equidistribution:spatial} we can find a pair of
finite collections of such ``hyperbolic trapezoids'' such that the
trapezoids would be pairwise disjoint in each collection; $X$ would
be a subset of the union of the trapezoids in the first collection;
the union of trapezoids in the second collection would be a subset of
$X$; the difference between two unions of trapezoids would have
arbitrary small measure. This proves
Theorem~\ref{th:equidistribution:spatial}.
\end{proof}

%--------------------------------------------------------------------
\subsection{Invariance along $\Re$ and $\Im$ foliations}
\label{ss:Invariance:along:Re:and:Im:foliations}

Note that any stratum of Abelian differentials
of complex dimension $d$
is endowed with a pair of transverse foliations
of real dimension $d$ induced from the canonical
direct sum decomposition in period coordinates
$$
H^1(S,\{P_1,\dots,P_\noz\};\C)=
H^1(S,\{P_1,\dots,P_\noz\};\R)\oplus
H^1(S,\{P_1,\dots,P_\noz\};i\R)\,.
$$
In particular, in a neighborhood of any point $(C,\omega)$ of the
stratum one has canonical direct product structure in period
coordinates. Locally, leafs of the $\Im$-foliation are preimages of
points under projection to the first summand, and of the
$\Re$-foliation --- to the second. In other words, pairs $(C,\omega)$
in a leaf of the $\Im$-foliation (respectively of the
$\Re$-foliation) share the cohomology class $[\Re \omega]\in
H^1(S,\{P_1,\dots,P_\noz\};\R)$ (respectively $[\Im \omega]\in
H^1(S,\{P_1,\dots,P_\noz\};\R)$).

By a fundamental theorem of Eskin, Mirzakhani and Mohammadi,
any linear $\GLR$-invariant suborbifold $\cL$ is represented in
period coordinates by a complexification of a real linear subspace
$L$ in the real relative cohomology. Thus, any such suborbifold also has
a direct product structure as above and is endowed with analogous
$\Re$ and $\Im$-foliations. They have real dimension $d$ where $d=\dim_{\R} L=
\dim_{\C}\cL$, and their leaves are intersections of the leaves of
$\Re$ and $\Im$-foliations in the ambient stratum with $\cL$.
If $\cL$ is locally represented in period
coordinates as a finite union of several linear subspaces, every such
subspace is foliated by $\Re$ and $\Im$-foliations.

\begin{Theorem}
\label{th:Re:Im:invariance}
Assume that the vertical foliation of some flat surface $S_0$ is
periodic. Let $\cD$ be its separatrix diagram. Then the vertical
foliation of any flat surface $S$ in the leaf of the $\Im$-foliation
passing through $S_0$ is also periodic and has separatrix diagram
$\cD$.

Similarly, if some flat surface $S_0$ has periodic horizontal
foliation, then all other flat surfaces in the same leaf of the
$\Re$-foliation also have periodic horizontal foliation with the same
separatrix diagram.
\end{Theorem}

Let us emphasize that by definition a leaf is always
connected. Note that if one considers
a $\GLR$-invariant suborbifold
$\cL \subset \cH$ then the intersection of $\cL$ with a leaf of
the $\Re$-foliation of the stratum $\cH$ might be non-connected.
The leaves of the $\Re$-foliations on $\cL$ are precisely the
connected components of these intersections.

The two parts of the statement are completely symmetric, so
it is sufficient to prove the second one, where the horizontal
foliation of $S_0$ is periodic. Its proof will use the following
simple Lemma.

\begin{Lemma}
\label{lem:span:hom}
Let $S$ be a translation surface with periodic
horizontal foliation. Consider the collection of all horizontal
saddle connections completed for each cylinder by a choice of a
non-horizontal segment inside the cylinder joining
some pair of singularities on the two boundary components of the
cylinder. Viewed as a collection of relative homology cycles,
such collection of saddle connections
spans the entire relative homology group
$H_1(S,\{P_1,\dots,P_\noz\};\Z)$.
\end{Lemma}
\begin{proof}
The complement to the union of our segments is a disjoint union of
topological discs. Thus, the collection of segments as above
defines a $1$-skeleton of a $CW$ decomposition of $S$, where all $0$-cells
belong to the finite set $\{P_1,\dots,P_\noz\}$.
\end{proof}

\begin{proof}[Proof of Theorem~\ref{th:Re:Im:invariance}] Denote the
lengths of the horizontal saddle connections of $S_0$ by $\ell_j$,
the lengths of the waist curves of the cylinders by $w_k$, the
heights of the cylinders by $h_k$ and the ``twists'' (horizontal
projections of the oriented non-horizontal segments crossing the
cylinders) by $\phi_k$. Denote by $\omega_0$ the Abelian differential
representing the initial flat surface $S_0$. Note that the numbers $\ell_j$
and $\phi_k$ are the relative periods of the relative cohomology
class $[\Re\omega_0]\in H^1(S,\{P_1,\dots,P_\noz\};\R)$ and the
numbers $h_k$ are the relative periods of $[\Im\omega_0]\in
H^1(S,\{P_1,\dots,P_\noz\};\R)$. The relative periods of
$[\Im\omega_0]$ corresponding to the horizontal saddle connections
are equal to zero.

Now, change the cohomology class $[\Re\omega_0]\in
H^1(S,\{P_1,\dots,P_\noz\};\R)$ by a small deformation and keep the
cohomology class $[\Im\omega_0]$ fixed. It would deform periods
$\ell_j, \phi_k$. The deformed parameters $\ell^{\mathit{def}}_j,
\phi^{\mathit{def}}_k, h_k$ represent the cylinder decomposition
of a well-defined translation surface $S^{\mathit{def}}$ with
periodic horizontal foliation. Indeed, one can glue a translation
surface living in the original stratum from the collection of deformed horizontal cylinders
corresponding to parameters $\ell^{\mathit{def}}_j,
\phi^{\mathit{def}}_k, h_k$
if and only if the following two conditions are satisfied:
all $\ell^{\mathit{def}}_j$ are strictly positive
and the lengths of the waist curves $w_k^{\mathit{def}}$ of the
deformed cylinders satisfy linear relations imposed by the
combinatorics of the diagram $\cD$. Since these relations are reduced
to relations in integer cohomology cocycles, and since we performed
the deformation of relative periods $\ell_j$ inside the relative
cohomology $H^1(S,\{P_1,\dots,P_\noz\};\R)$, all such relations are
automatically satisfied.

By construction any deformed translation surface $S^{\mathit{def}}$
belongs to the leaf of the $\Re$-foliation passing through $S_0$ and
represents the separatrix diagram $\cD$. Considering deformations in a
small open neighborhood of $[\Re\omega_0]$ in
$H^1(S,\{P_1,\dots,P_\noz\};\R)$ we get an entire small open
neighborhood of $S_0$ in the $\Re$-leaf passing through $S_0$. This
implies the statement of the Theorem.
\end{proof}

Now we are ready to prove Theorem~\ref{th:sep:diagram:is:essential}.

\begin{proof}[Proof of Theorem~\ref{th:sep:diagram:is:essential}]
By assumption, the suborbifold $\cL$ is $\GLR$-invariant,
in particular, $\Uh$-invariant. The set of all square-tiled
surfaces (pillowcase covers) in the ambient stratum represented by
any given separatrix diagram $\cD$ is $\Uh$-invariant. Thus, the
condition  a) of Definition~\ref{def:essential:lattice:subset} is
respected for any $\GLR$-invariant suborbifold $\cL$ and for any
separatrix diagram $\cD$.

We have to show now that the limit in ``b'' exists, that it is
finite, and that it is strictly positive. (Basically, what we have to
prove is that the contribution  of  any  realizable separatrix
diagram to the Masur--Veech volume of any the ambient
$\GLR$-invariant suborbifold $\cL$ defined over $\Q$ is
well-defined and strictly positive.)

Applying if necessary the canonical orienting
double cover construction $\hat S\to S$ to every flat surface $S$ in
$\cL$, we can restrict the proof to the case of a suborbifold in a
stratum of Abelian differentials: the two counts might differ only by
some finite strictly positive normalization factor.

First note that
\begin{equation}
\label{eq:sup}
2d\cdot \limsup_{a\to+\infty} \frac{1}{a^{d}} \cdot
\card\{S\in \cD_\Z(\cL)\,\big|\, \Area(S) \le a\} <+\infty\,.
\end{equation}
This property is a trivial corollary of the nontrivial theorems of
A.~Eskin, M.~Mirzakhani and A.~Mohammadi~\cite{Eskin:Mirzakhani},
\cite{Eskin:Mirzakhani:Mohammadi} telling that the Masur--Veech
volume of $\cL$ is finite. By definition the Masur--Veech volume
$\Vol(\cL)$ is defined as the analogous limit where instead of
$\cD_\Z(\cL)$ we take the set of all integer points in $\cL$. Since
$\cD_\Z(\cL)$ forms a subset of the set of all integer points, the
relation above follows.

Let us prove the relation
\begin{equation}
\label{eq:inf}
2d\cdot \liminf_{a\to+\infty} \frac{1}{a^{d}} \cdot
\card\{S\in \cD_\Z(\cL)\,\big|\, \Area(S) \le a\} >0\,.
\end{equation}
We use notation $\ell_j, \phi_k, w_k, h_k$ for parameters of a
cylinder decomposition as in the proof of
Theorem~\ref{th:Re:Im:invariance}. Let $S_0\in\cL$ be a translation
surface with periodic horizontal foliation realizing the separatrix
diagram $\cD$. By assumption, the $\GLR$-invariant suborbifold $\cL$
is defined over $\Q$ which implies that rescaling the flat surface
$S_0$ in vertical direction by an appropriate linear transformation
$\begin{pmatrix}1& 0\\ 0&R\end{pmatrix}$ we can make all vertical
parameters $h_j$ integer. We will still denote the resulting surface
in $\cL$ by $S_0$, and the Abelian differential representing it by
$\omega_0$.

Consider a small open neighborhood $U(S_0)$ of $S_0$ in the leaf of
the $\Re$-foliation in $\cL$ passing through $S_0$. Denote by $A$ the
supremum of the flat area of translation surfaces in $U(S_0)$. Let
$L$ be the linear subspace in period coordinates
$H^1(S,\{P_1,\dots,P_\noz\};\R)$ representing $\cL$. We can consider
$U(S_0)$  as an open domain $U([\Re(\omega_0)])\subset L$ under the
natural local identification of the $\Re$-leaf through $S_0$ in $\cL$
with $L\subset H^1(S,\{P_1,\dots,P_\noz\};\R)$. Since $L$ is a
rational subspace of the relative cohomology
$H^1(S,\{P_1,\dots,P_\noz\};\R)$ with respect to the natural integer
basis in period coordinates, we conclude that the rational grid $L\cap
H^1(S,\{P_1,\dots,P_\noz\};\frac{1}{N}\Z)$ forms a $d$-dimensional
lattice in $L$, where $d=\dim L=\dim_{\C}\cL$. Consider a translation
surface $S^{\mathit{def}}\in U$ corresponding to a point of this
grid, and apply the transformation
$\begin{pmatrix}N&0\\0&1\end{pmatrix}$ to it.
Theorem~\ref{th:Re:Im:invariance} implies that we get a surface in
$\cD_\Z(\cL)$ of area at most $A\cdot N$. The number of the points of
our rational grid which get inside $U$ is of the order $\Vol(U)\cdot
N^d$. The relation~\eqref{eq:inf} is proved.

It remains to prove that the upper and lower limits~\eqref{eq:sup}
and~\eqref{eq:inf} coincide. This follows from the fact that we can
vary the integer horizontal parameters $\ell_j, \phi_k$ independently
from the integer vertical parameters $h_k$ in such way that the
translation surface stays in $\cD_\Z(\cL)$, and from the fact that we
have flexibility with order of summation when manipulating series of
\textit{positive} terms.

Fix any collection of integer parameters $h_k$ realizable as heights
of the cylinders of some translation surface $S_0$ in $\cD_\Z(\cL)$.
The $\Re$-leaf $\cL_{\Re}(S_0)=\cL_{\Re}(h_1,\dots)$ in $\cL$ passing
through $S_0$ has the following global structure. The admissible
parameters $\ell_j$ representing the horizontal cylinder
decomposition of an actual translation surface in $\cL$ with periodic
horizontal foliation corresponding to the diagram $\cD$ might be
defined as the set of solutions of a system of linear equations and
inequalities in $\ell_j$. This defines a finite union of polyhedral
cones. For any fixed collection of admissible parameters $\ell_j$,
parameters $\phi_k$ belong to a torus of some fixed dimension
determined by $\cD_\Z(\cL)$, and $\cL_{\Re}(h_1,\dots)$ is the total
space of the corresponding torus bundle. Note that it is endowed with
the natural volume element coming from the volume element in the
period coordinates $L\subset H^1(S,\{P_1,\dots,P_\noz\};\R)$
normalized by means of the integer lattice $L\cap
H^1(S,\{P_1,\dots,P_\noz\};\Z)$. (By assumption $L$ is rational, so
this lattice has maximal dimension in $L$.)

The condition $\Area(S)\le a$ is expressed in our coordinates as
$\sum w_k h_k \le a$, where $w_k$ are appropriate sums of
subcollections of $\ell_j$, and where we consider $h_k$ as fixed
parameters. Denote by
$\cD_\Z(h_1,\dots)=\cL_{\Re}(h_1,\dots)\cap\cD_\Z(\cL)$ the subset of
$\cD_\Z(\cL)$ corresponding to the fixed realizable collection of
integer parameters $h_k$. The limit
$$
c(h_1,\dots)=2d\cdot \lim_{a\to+\infty} \frac{1}{a^{d}} \cdot
\card\{S\in \cD_\Z(h_1,\dots)\,\big|\, \Area(S) \le a\} \,,
$$
exists since it computes the well-defined volume of
$\cL_{\Re}(h_1,\dots)$.
We have expressed the limit $c(\cD_\Z(\cL))$
in~\eqref{eq:limit} as an infinite sum of positive terms
\begin{equation}
\label{eq:sum:c:h1:dots}
c(\cD_\Z(\cL))=
\sum_{\substack{\mathit{Realizable}\\ \mathit{integer}\ h_1,\dots}} c(h_1,\dots)\,.
\end{equation}
By~\eqref{eq:sup} the sum is bounded, and hence converging, which
proves that the limit exists.
\end{proof}

\begin{Remark}
The diagram-by-diagram computation of the Masur--Veech volume
performed for some low-dimensional strata of Abelian and quadratic
differentials follows~\eqref{eq:sum:c:h1:dots}.
The reader can find such calculation
of $\Vol\cH(2)$ in~\eqref{eq:D1} and~\eqref{eq:D2} in
Appendix~\ref{ss:separatrix:diagrams}.
\end{Remark}

\begin{Definition}
\label{def:Im:Re:invariant}
We say that a lattice subset $\cD_\Z(\cL)$
in a linear $\GLR$-invariant
orbifold $\cL$ defined over $\Q$ is $\Re$-\textit{invariant}
(respectively $\Im$-\textit{invariant}) if for any integer point
$S_0$ of $\cD_\Z(\cL)$ all integer points
located in the
leaf of the $\Re$-foliation (respectively $\Im$-foliation)
in $\cL$ passing through $S_0$
also belong to $\cD_\Z(\cL)$.
\end{Definition}

Note that the unipotent subgroups
$$
\operatorname{U}_h=\left\{\begin{pmatrix} 1&t\\0&1\end{pmatrix}\right\},\
\operatorname{U}_v=\left\{\begin{pmatrix} 1&0\\t&1\end{pmatrix}\right\},\
\text{ where }t\in\R\,.
$$
act along the leaves of the $\Re$-foliation and $\Im$-foliation
respectively. Thus any $\Re$-invariant (respectively $\Im$-invariant)
lattice subset $\cD_\Z$ is automatically $\Uh$-invariant
(respectively $\Uv$-invariant). If, in addition, condition b) in
Definition~\ref{def:essential:lattice:subset} is satisfied, we will
say that $\cD_\Z$ is a $\Re$-\textit{invariant} (respectively
$\Im$-\textit{invariant}) \textit{essential lattice subset}. Clearly,
the $\Re$-invariance (correspondingly $\Im$-invariance) is a stronger
property than the $\Uh$-invariance (correspondingly $\Uv$-invariance)
of an essential lattice subset.

\begin{Corollary}
\label{cor:DZ:is:Im:invariant}
Fix a separatrix diagram $\cD$ of the horizontal (respectively
vertical) foliation realizable for some
$\GLR$-invariant orbifold $\cL$ defined over $\Q$.
Consider the set $\cD_\Z(\cL)$ of square-tiled surfaces in
$\cL$ tiled by unit squares and having the separatrix diagram $\cD$. The
set $\cD_\Z(\cL)$ is a $\Re$-\textit{invariant} (respectively
$\Im$-\textit{invariant}) essential lattice subset.
\end{Corollary}
\begin{proof}
The corollary is a simple combination of
Theorem~\ref{th:sep:diagram:is:essential} and
Theorem~\ref{th:Re:Im:invariance}.
\end{proof}

The above Corollary immediately implies that the sets of square-tiled
surfaces in any connected component of a stratum having a fixed
number of horizontal (respectively vertical) cylinders is also
$\Re$-invariant (respectively $\Im$-invariant).

We have seen in the proof of Theorem~\ref{th:Re:Im:invariance} that
assigning any fixed realizable integer values to the heights
$h_1,\dots$ of the maximal cylinders we still get a
$\Re$-\textit{invariant} essential lattice subset
$\cD(h_1,\dots)\subset \cD_\Z(\cL)$ in the ambient $\GLR$-invariant
suborbifold $\cL$. However, imposing any nontrivial realizable
additional linear inequalities on the lengths of the horizontal saddle
connections $\ell_j$ we get a $\Uh$-invariant essential lattice
subset which is already \textit{not} $\Re$-invariant.

%--------------------------------------------------------------------
\subsection{Equidistribution for interval exchange transformations}

For any interval exchange transformation $T=(\pi,\lambda)$ one can
construct a continuous family $\cS(\pi,\lambda)$
of flat surfaces carrying the structure
of a suspension over $T$.

\begin{Lemma}
\label{lm:suspensions:Im:leaf}
The set $\cS(\pi,\lambda)$ of all suspensions over any fixed interval
exchange transformation $T=(\pi, \lambda)$ forms an open connected
subset of the corresponding leaf of the $\Im$-foliation.
\end{Lemma}
\begin{proof}
The set $\cS(\pi,\lambda)$ of suspensions is described
in~\cite{Veech:Gauss:measures}; in particular, it is proved there
that it forms an open connected $d$-dimensional domain, where $d$ is
the complex dimension of the stratum. (Actually,
the article~\cite{Veech:Gauss:measures} describes much finer structure of
this set.) The fact that it is a subset of a leaf of the
$\Im$-foliation is a triviality.

The reader might consult the paper~\cite{Masur:82} for a
very explicit realization of a certain concrete suspension over any
given interval exchange transformation
\end{proof}

Theorem~\ref{th:Re:Im:invariance} and
Lemma~\ref{lm:suspensions:Im:leaf} imply that certain properties of
any suspension over an interval exchange transformation are
completely determined by the underlying interval exchange
transformation. In particular, the vertical foliation is
periodic for a suspension if and only if the original interval
exchange transformation is periodic. The separatrix
diagram $\cD$ of the corresponding periodic vertical
foliation is one and the same for all suspensions and, thus, is
encoded by the underlying interval exchange transformation $T$.
Note that if the lengths of all subintervals under exchange are
commensurable, the interval exchange transformation is necessarily
periodic.

Combining the  equidistribution
Theorem~\ref{th:equidistribution:spatial} with
Lemma~\ref{lm:suspensions:Im:leaf} and with
Corollary~\ref{cor:DZ:is:Im:invariant} we obtain a
similar equidistribution   result   for   interval  exchange
transformations.

We say that a permutation $\pi$ on $\{1,2,\ldots,d\}$
is {\em irreducible} if the only
$\pi$-invariant subsets of the form $\{1, 2, \dots, k\}$ are the empty
set and $\{1,2, \dots, d\}$.

\begin{Theorem}
\label{th:iet}
Let $\pi$ be an irreducible
permutation, let $\cH^{comp}$ be the associated connected component
of the stratum of Abelian differentials ambient for suspensions over
interval exchange transformations with permutation $\pi$ and let $d$ be
its complex dimension. Furthermore, let $\cD_\Z$ be an $\Im$-invariant essential
lattice subset of square-tiled surfaces in $\cH^{comp}$.

Consider any open and relatively compact set $I$ in $\R^d_+$.
The number
$$
n_\cD(I,\epsilon):=
\card\{\lambda\in(\epsilon\N)^d\cap I\ \big|\
\cS(\pi,\lambda)\cap\cD_{\varepsilon\Z}\neq\emptyset\}
$$
of interval exchange transformations $(\pi,\lambda)$ in
an $\epsilon$-grid in $I$
such that $(\pi,\lambda)$
has suspensions in $\cD_{\varepsilon\Z}$,
asymptotically   depends   only  on
$\cD_\Z$ and on the Lebesgue measure $\vol(I)$ of $I$:
\begin{equation}
\label{eq:equidistribution1:iet}
\lim_{\epsilon\to+0} \epsilon^{d}\cdot n_\cD(I,\epsilon)  =
\frac{c(\cD_\Z)}{\nu_1(\cH^{comp})}\cdot\vol(I)\,,
\end{equation}
where  the  constant $c(\cD_\Z)$ defined in~\eqref{eq:limit} depends
only on the essential lattice subset.

Equivalently, for any continuous function
$f: \R^d_+ \to \R$ with compact support one has
\begin{equation}
\label{eq:equidistribution2:iet}
\lim_{\epsilon \to +0} {\epsilon^{d}}
\sum_{\substack{\lambda\in(\epsilon\N)^d\cap I\ \text{such}\\
\text{that }\ \cS(\pi,\lambda)\cap\cD_{\varepsilon\Z}\neq\emptyset}}
f(\lambda) =
\frac{c(\cD_\Z)}{\nu_1(\cH^{comp})}
\int_{I} f(\lambda) d\lambda.
\end{equation}
\end{Theorem}

\begin{proof}
We first prove the result for some small enough open neighborhood $I$ of a
point $\lambda_0 \in \R^d_+$. Fix $\lambda_0 \in \R^d_+$ and a
suspension $S_0 \in \cS(\pi, \lambda_0)$. Because of the product structure,
there exists a neighborhood of $S_0$ in the corresponding stratum of the
form $X = I \times J$, where $I$ corresponding to the real part is identified
with a subset of $\R^d_+$ corresponding to the length data of the interval
exchange transformation, and $J$ corresponds to the imaginary part.

Because $\cD_\Z$ is $\Im$-invariant, the property of being in $\cD_\Z$ only depends on
the coordinate in $I$. As a consequence we have
$$
\widetilde{N}_{\cD}(X, \epsilon) = n_\cD(I, \epsilon)
\times \card( J \cap (\epsilon \Z^d))\,,
$$
where $\widetilde{N}_{\cD}(X, \epsilon)$ is the quantity introduced
in Theorem~\ref{th:equidistribution:spatial}. Applying
Theorem~\ref{th:equidistribution:spatial} we get as $\epsilon$
tends to $0$ that
$$
\epsilon^d n_\cD(I, \epsilon) =
\frac{\epsilon^{2d} \widetilde{\cN}_\cD(X, \epsilon)}{\epsilon^d \card( J \cap (\epsilon \Z^d))}
\to \frac{c(\cD_\Z)}{\nu_1(\cL_1)} \frac{\nu(X)}{\vol(J)}.
$$
Because $\nu$ is the product measure of the Lebesgue measures along
$\Re$ and $\Im$ foliations we have $\nu(X) = \vol(I) \times \vol(J)$ and
the result follows.

We just proved the result for an open set $I$ small enough
so that there exists an open set $J$ with a product $X = I \times J$ that
embedds in the corresponding stratum $\cH$. Now the result extends to any
relatively compact $I$ by considering finite partition so that the embedding
can be realized on each element of the partition.
\end{proof}

We are especially interested by the following particular case
of the above Theorem.

\begin{Corollary}
\label{cor:k:cyl:iet}
Let $\pi$ be a permutation in the Rauzy class representing some
connected component $\cH^{comp}(m_1,\dots,m_\noz)$ of a stratum of
Abelian differentials, and let $d$ be the number of elements in $\pi$.

Consider any open bounded set $I$ in $\R^d_+$. The numbers
of interval exchange transformations $(\pi,\lambda)$ in
an $\epsilon$-grid in $I$ satisfying
$$
n_k(I,\epsilon):=
\card\{\lambda\in(\epsilon\N)^d\cap I\ \big|\
(\pi,\lambda)\text{ has $k$ bands of periodic orbits}\}\quad
k=1,2,\dots
$$
have the same asymptotic proportions as the
proportions $p_k(\cH^{comp}(m_1,\dots,m_\noz))$ of
$k$-cylinder square-tiled surfaces
in the ambient component $\cH^{comp}(m_1,\dots,m_\noz)$
of the stratum of Abelian differentials:
\begin{equation}
\label{eq:equidistribution:iet:ratio:pi:pj}
\lim_{\epsilon\to+0} \frac{n_i(I,\epsilon)}{n_j(I,\epsilon)}  =
\frac{\prop_i(\cH^{comp}(m_1,\dots,m_\noz))}{\prop_j(\cH^{comp}(m_1,\dots,m_\noz))}\quad\text{for any }i,j\in\N\,.
\end{equation}
\end{Corollary}

Theorem~\ref{th:iet} allows a refinement for any linear
$\GLR$-invariant orbifold $\cL$ defined over $\Q$ and for any
$\Im$-invariant essential lattice subset $\cD_\Z$ in it. This time
the number of elements in $\pi$ would be different from the dimension
$d$ of $\cL$. One has to consider an initial interval exchange
transformation $(\pi,\lambda_0)$ in such a way that some suspension
over it belongs to $\cL$. One has to take an $\epsilon$-grid in an
appropriate open neighborhood $I$ of $(\pi,\lambda_0)$ in $L$, where
$L$ is the real linear subspace in $H^1(S,\{P_1,\dots,P_\noz\};\R)$
defined over $\Q$ representing $\cL$ in period coordinates. The power
$d$ in the normalization $\epsilon^d$ should be understood now as the
complex dimension of $\cL$, or, equivalently, as the real dimension
of $L$. Under these adjustments the statement and the proof become
completely analogous to the ones of Theorem~\ref{th:iet}.

\begin{Remark}
\label{rem:lin:invol}
Statements completely analogous to Theorem~\ref{th:iet}
and to Corollary~\ref{cor:k:cyl:iet} hold for rational linear
involutions and $\Im$-invariant essential lattice subsets of
pillowcases covers in the strata of quadratic differentials.
We specify the meaning of \textit{``integer''} (or \textit{``rational''})
linear involutions in Lemma~\ref{lm:integer:linear:involutions} in
the next section.
\end{Remark}

%--------------------------------------------------------------------
\subsection{Dependence on the definition of ``integer points''}\label{sect:dep}

Clearly, everything which was stated for linear $\GLR$-invariant
suborbifolds defined over $\Q$ in strata of Abelian differentials can
be generalized to analogous linear invariant suborbifolds $\cL$
defined over $\Q$ in strata of quadratic differentials. Applying the
canonical orientation double cover construction $\hat S\to S$ to
every flat surface $S$ in $\cL$ we transform $\cL$ into a linear
$\GLR$-invariant suborbifold $\hat\cL$ defined over $\Q$ located
already in the stratum of Abelian differentials ambient for $\hat S$.

Note that the choice of the ``integer lattice'' in the period
coordinates under this construction is a matter of convention. Recall
that a stratum of quadratic differentials is modeled on the subspace
$H^1_-(\hat S, \{\hat{P}_1,\dots, \hat{P}_\noz\};\C)$ antiinvariant
under the canonical involution of $\hat S$.

\begin{Convention}
\label{conv:lattice}
We chose as a distinguished lattice    in    $H^1_-(\hat
S,\{\hat{P}_1,\ldots,\hat{P}_\noz\};\C{})$ the  subset  of those
linear  forms  which  take  values  in  $\Z\oplus i\Z$ on $H^-_1(\hat
S,\{\hat   P_1,\dots,\hat   P_\noz\};\Z)$.
\end{Convention}

Alternatively, we can define the lattice as
$$
H^1_-(\hat S, \{\hat{P}_1,\dots, \hat{P}_\noz\};\C)\cap H^1(\hat S,
\{\hat{P}_1,\dots, \hat{P}_\noz\};\Z\oplus i\Z)\,.
$$
The former lattice is a sublattice of index $4^{2g+\noz-1}$ in the latter
one, where $g$ is the genus of $S$ and $\noz$ is the number of true
zeroes on $\hat S$, forgetting the ``false ones'', i.e. forgetting
the marked points corresponding to preimages of the poles (see
section~2.2 in~\cite{AEZ:Dedicata} and section~5.8
in~\cite{Goujard:volumes} for details on comparison of these two
conventions).

Note also, that the choice of such convention might considerably
affect the value of the constant $c(\cD_\Z)$ in the
definition~\eqref{eq:limit} of an essential lattice subset. For
example, the contribution of a fixed separatrix diagram to the volume
of the ambient stratum of quadratic differentials might change by a
factor different from $4^{2g+\noz-1}$ when passing from the second
lattice to the first one, though the total sum of contributions of all
separatrix diagrams differs for two choices of the lattice by the
factor $4^{2g+\noz-1}$, see an example considered in detail in
section~\ref{s:contibution:of:diag:for:two:lattices}.

Throughout this paper we follow Convention~\ref{conv:lattice}.
Lemma~\ref{lm:integer:linear:involutions} below shows that this
choice is coherent with the naive choice of integer (rational) linear
involutions.

Consider some linear involution $\pi$ of
$d+1$ elements (see~\cite{Danthony:Nogueira};
in alternative terminology --- a
\textit{generalized permutation}, see~\cite{Boissy:Lanneau}).
The lengths of subintervals of the corresponding generalized interval
exchange transformation satisfy a linear relation of the form
\begin{equation}
\label{eq:relation:generalized:iet}
\lambda_{j_1}+\dots+\lambda_{j_m}=\lambda_{k_1}+\dots+\lambda_{k_n}\,.
\end{equation}
Choose any parameter in the right or left hand side of
relation~\eqref{eq:relation:generalized:iet}. The remaining $d$
lengths of intervals under exchange in our generalized interval
exchange transformation provide local coordinates in the space of
generalized interval exchange transformations.

Consider a suspension over the generalized interval exchange
transformation $(\lambda,\pi)$, where $\lambda$ satisfies the
relation~\eqref{eq:relation:generalized:iet}. Note that each
interval under exchange is the projection of the corresponding saddle
connection to the real axes. In particular, under appropriate
orientation of the relative cycles in $H_1^-(\hat S,
\{\hat{P}_1,\dots, \hat{P}_\noz\};\Z)$ induced by such saddle
connections, the lengths of the subintervals under exchange represent
the real parts of half-periods of $\hat\omega$ on $\hat S$ over the
corresponding cycles, and thus, our coordinates in the space of
generalized interval exchange transformations can be viewed as period
coordinates $H^1_-(\hat S, \{\hat{P}_1,\dots, \hat{P}_\noz\};\R)$.

Choosing values in $(\epsilon\Z)^d$ for the chosen parameters we get
a subset $S_{\epsilon\Z}$ in the intersection of a neighborhood
$U(\lambda_0,\pi)\subset H^1_-(\hat S, \{\hat{P}_1,\dots,
\hat{P}_\noz\};\R)$ with the lattice $H^1_-(\hat S,
\{\hat{P}_1,\dots, \hat{P}_\noz\};\epsilon\Z)$. Note that when all
lengths chosen as coordinates are integer multiples of $\epsilon$,
the single remaining length is also an integer multiple of $\epsilon$
by relation~\eqref{eq:relation:generalized:iet}. This observation
implies the following simple Lemma.

\begin{Lemma}
\label{lm:integer:linear:involutions}
The resulting subset $S_{\epsilon\Z}$ in a small open neighborhood
$U(\lambda_0,\pi)\subset H^1_-(\hat S, \{\hat{P}_1,\dots,
\hat{P}_\noz\};\R)$ coincides with the intersection of $U$ with the
lattice defined by Convention~\ref{conv:lattice} scaled by the factor
$2\epsilon$.
\end{Lemma}

This lemma implies that stating
direct analogs of Theorem~\ref{th:iet} and of Corollary~\ref{cor:k:cyl:iet}
for linear involutions with lengths of the subintervals in the mesh
$(\epsilon\N)^d$ we get absolute contributions $c(\cD_\Z)$
in the analog of Theorem~\ref{th:iet}
and the relative contributions $p_i(\cQ^{comp}(d_1,\dots,d_k))$
in the analog of Corollary~\ref{cor:k:cyl:iet}
which correspond to Convention~\ref{conv:lattice} on the choice of the
integer lattice in
$H^1_-(\hat S, \{\hat{P}_1,\dots,
\hat{P}_\noz\};\Z\oplus i\Z)$.

Returning to the discussion of the choice of ``integer points'',
one can redefine ``integer points'' more radically than choosing a
finite index sublattice in cohomology. The Equidistribution
Theorem~\ref{th:equidistribution} shows that \textit{any}
$\Uh$-invariant (or $\Uv$-invariant) essential lattice subset $\cD_\Z$
might serve as a ``set of integer points''. We have to adjust the
definition of the Masur--Veech measure in terms of such a set of
integer points defining it as
\begin{equation}
\label{eq:measure:through:integer:points}
\nu'_1(X):=\lim_{R\to+\infty} 2d\cdot\frac{1}{R^d}\cdot
\card\{\cD_\Z\cap C_R X_1\}
\end{equation}
(normalization factor $2d$ is kept here by purely esthetic
reasons). By Theorem~\ref{th:equidistribution} the resulting measure
$d\nu'_1$ is proportional to the original Masur--Veech measure
with the constant factor $p(\cD_\Z)$.

%--------------------------------------------------------------------
\subsection{Horizontal and vertical cylinder decompositions
are uncorrelated}
\label{ss:Uncorrelated:properties}

Essential lattice subsets $\cD_\Z$ and $\cD'_\Z$ are, in general,
``correlated'':
$p(\cD_\Z\cap\cD'_\Z)$ is not necessarily equal to $p(\cD_\Z)\cdot
p(\cD'_\Z)$.
The results below show, however,
that any horizontally-invariant and any vertically-invariant
essential properties detectable by interval exchange transformations
or by the real part of relative cohomology are ``uncorrelated'', that is the
equality:
\begin{equation}
\label{eq:uncorrelated}
p(\cD_\Z\cap\cD'_\Z)=p(\cD_\Z)\cdot p(\cD'_\Z)
\end{equation}
is valid.

Fix some lattice subset $\cD_\Z$ of square-tiled surfaces tiled with
unit squares in some linear $\GLR$-invariant suborbifold defined over
$\cQ$. Recall that $\cD_{\epsilon\Z}$ denotes the image of $\cD_\Z$
under the action of the uniform contraction with the scaling factor
$\epsilon$, namely
$$
\cD_{\epsilon\Z}:=
\begin{pmatrix}\epsilon&0\\0&\epsilon\end{pmatrix}\cD_\Z\,.
$$

\begin{Theorem}
\label{th:uncorrelated}
Let $\cD^h_\Z,\cD^v_\Z$ be a pair of essential lattice subsets
in some linear $\GLR$-invariant suborbifold $\cL$ defined over $\Q$
that are respectively $\Re$-invariant and $\Im$-invariant. Denote by $d$
the complex dimension of $\cL$. Let $X$ be an open subset in $\cL$. The number
$$
\widetilde\cN_{\cD^h,\cD^v}(X,\epsilon)=
\card\{S\in\cD^h_{\epsilon\Z}\cap \cD^v_{\epsilon\Z}\cap X\}
$$
asymptotically   depends   only  on
$\cD^h_\Z, \cD^v_\Z$ and on the Masur--Veech measure $\nu(X)$:
\begin{equation}
\label{eq:uncorrelated:count}
\lim_{\epsilon\to+0} \epsilon^{2d} \widetilde\cN_\cD(X,N)  =
\frac{c(\cD^h_\Z)}{\nu_1(\cL_1)}\cdot \frac{c(\cD^v_\Z)}{\nu_1(\cL_1)}\cdot
\nu(X)\,,
\end{equation}
where  the  constants
$c(\cD^h_\Z)$ and $c(\cD^v_\Z)$ are
defined in~\eqref{eq:limit}.

Equivalently, for any continuous function $f: \cL \to \R$, which is compactly
supported or positive (or satisfies both of these
properties) one has:
\begin{equation}
\label{eq:uncorrelated:f}
\lim_{\epsilon \to +0}
\epsilon^{2d}
\sum_{S \in \cD^h_{\epsilon\Z}\cap\cD^v_{\epsilon\Z}\cap X}
f(S) = \frac{c(\cD^h_\Z)}{\nu_1(\cL_1)}\cdot \frac{c(\cD^v_\Z)}{\nu_1(\cL_1)}\cdot
 \int_{X} f d\nu.
\end{equation}
\end{Theorem}
This result which can be thought of as a refinement of
Theorem~\ref{th:equidistribution:spatial} in the case of $\Re$ and
$\Im$-invariant essential lattices also has an equivalent statement
in terms of the ``unit hyperboloid'' $\cL_1$ as in
Theorem~\ref{th:equidistribution}.

In particular, the properties of a square-tiled surface in a given
connected component of a stratum of Abelian differentials to have
certain fixed separatrix diagram of the horizontal foliation and
certain fixed separatrix diagram of the vertical foliation, or to
have $i$ horizontal and $j$ vertical maximal cylinders in the
cylinder decomposition are uncorrelated: in these situations
equation~\eqref{eq:uncorrelated} is valid.

\begin{proof}
The proof is similar to the one of Theorem~\ref{th:iet}.

It is enough to consider sets $X$ of the form $I \times J$ where $I$
and $J$ are respectively open sets in the natural product structure
of $\cL$ provided by $\Re(\omega)$ and $\Im(\omega)$. For an open set
$X = I \times J$ the quantity
$\widetilde\cN_{\cD^h,\cD^v}(X,\epsilon)$ is exactly a product since
being in $\cD^h_\Z$ (respectively in
$\cD^v_\Z$) depends only on the coordinates in $I$
(respectively in $J$)):
$$
\widetilde\cN_{\cD^h,\cD^v}(X,\epsilon) =
\frac{\widetilde{\cN}_{\cD^h}(X, \epsilon)}{\card(J \cap \cD_{\epsilon \Z})} \times
\frac{\widetilde{\cN}_{\cD^v}(X,\epsilon)}{\card(I \cap \cD_{\epsilon \Z})}
$$
where $\widetilde{\cN}_{\cD^h}(X, \epsilon)$ and $\widetilde{\cN}_{\cD^v}(X,\epsilon)$ are the same
as in Theorem~\ref{th:equidistribution:spatial} and $\cD_\Z$ is the essential lattice of all square
tiled surfaces. Now by Theorem~\ref{th:equidistribution:spatial} we have
$$
\epsilon^d \frac{\widetilde{\cN}_{\cD^h}(X, \epsilon)}{\card(J \cap \cD_\Z)}
\to
\frac{c(\cD_\Z^h) \nu(X)}{\nu_1(\cL_1) \vol(J)}
$$
	as $\epsilon$ tends to $0$
A similar formula holds for the other term in the product, and the result follows from
the fact that $\nu(X) = \vol(I) \times \vol(J)$.

\end{proof}

%####################################################################
%####################################################################
%####################################################################
\section{Contribution of $1$-cylinder square-tiled surfaces to Masur--Veech volumes}
\label{s:contribution:of:1:cylinder}

In   this   section  we  consider square-tiled surfaces and
pillowcase covers represented  by  a  \emph{single} maximal flat
cylinder $C$ filled by closed horizontal leaves and their
contributions to the Masur--Veech volumes of strata of Abelian
differentials and of meromorphic quadratic differentials with at most
simple poles.

In section~\ref{ss:Contribution:of:1:cylinder:diagrams} we state the
main results. Their proofs are postponed to
sections~\ref{ss:contribution:of:one:1:cylinder:diagram:computation}
and~\ref{ss:1:cylinder:diagrams:Abelian}. In
section~\ref{ss:Asymptotics:in:large:genera} we apply our results to
strata of Abelian differentials in large genus and discuss how they
compare with the conjectures on the large genus asymptotic behavior of
Masur--Veech volumes and with recent
results of D.~Chen, M.~M\"oller and
D.~Zagier. In
section~\ref{ss:Application:experimental:evaluation:of:MV:volumes} we
describe the experimental approach to the computation of Masur--Veech
volumes unifying our equidistribution and counting results.

We proceed in section~\ref{ss:contribution:of:one:1:cylinder:diagram:computation}
with a detailed discussion of relevant combinatorial aspects
and with a computation of the contribution of a single $1$-cylinder
separatrix diagram to the Masur--Veech volume of the ambient stratum
proving Propositions~\ref{pr:contribution:Abelian}
and~\ref{pr:contribution:quadratic}.

In section~\ref{ss:1:cylinder:diagrams:Abelian} we count the number
of $1$-cylinder diagrams for strata of Abelian differentials.
Combining our count with the result of
section~\ref{ss:contribution:of:one:1:cylinder:diagram:computation}
we derive very sharp
bounds~\eqref{eq:contribution:all:1:cyl:estimate} for the absolute
contribution of $1$-cylinder square-tiled surfaces to the
Masur--Veech volume claimed in
Theorem~\ref{th:contribution:all:1:cyl:estimate}. We also obtain
exact closed formulas for the absolute contributions
of $1$-cylinder square-tiled surfaces to the Masur--Veech volumes of
the minimal and principal strata stated in
Corollary~\ref{cor:total:contribution:min:and:principal}.

%--------------------------------------------------------------------
\subsection{Contribution of $1$-cylinder diagrams}
\label{ss:Contribution:of:1:cylinder:diagrams}

The volume of a stratum consists of contributions of $1$-cylinder
surfaces, 2-cylinder surfaces, etc.
\begin{eqnarray*}
\Vol \cH_1(m_1,\dots, m_\noz) & = & c_1+c_2+\dots\\
&= &  \sum_{\textrm{realizable }\; \cD} c(\cD)\,,
\end{eqnarray*}
where $c_i=c_i(m_1, \dots, m_\noz)$ is the contribution of the
surfaces with $i$ cylinders to the volume, and $c(\cD)=c(\cD_\Z)$ is
the contribution of a realizable separatrix diagram $\cD$ for the
stratum $\cH(m_1, \dots, m_\noz)$. The same decomposition holds for
the strata of quadratic differentials.

\begin{Proposition}
\label{pr:contribution:Abelian}
The   contribution   of   any   $1$-cylinder   orientable  separatrix
diagram    $\cD$    to   the   volume   $\Vol\cH_1(m_1,\dots,m_\noz)$
of   the   corresponding  stratum  of  Abelian  differentials  equals
\begin{equation}
\label{eq:contribution:numbered}
c(\cD)=\cfrac{2}{|\Gamma(\cD)|}\cdot
\cfrac{\mult_1!\cdot \mult_2! \cdots}{(d-2)!}\,\cdot \zeta(d)\,.
\end{equation}
Here  $|\Gamma(\cD)|$ is the order of the symmetry group of the
separatrix  diagram  $\cD$; $\mult_i$ is the number of zeroes of order
$i$,  i.e.  the  multiplicity  of  the  entry  $i$  in  the  set
$\{m_1,\dots,m_\noz\}$;      and      $d=\dim_\C
\cH(m_1,\dots,m_\noz)=2g+\noz-1$.  Speaking  of  the  volume  of  the
stratum  we  assume that the zeroes $P_1,\dots,P_\noz$ of the Abelian
differentials are numbered (labeled).
\end{Proposition}

For the case of quadratic differentials, consider a non-orientable measured foliation on a closed surface such
that  all  its  regular  leaves  are  closed  and  fill a single flat
cylinder.  Cut  the surface along all saddle connections to unwrap it
into  a  cylinder.  Every  saddle connection is presented exactly two
times  on the boundary of the resulting cylinder. Call any of the two
boundary   components  of  the  cylinder  the  ``top''  one  and  the
complementary  component  ---  the  ``bottom'' one. Denote by $l$ the
number of saddle connections which are presented once on top and once
on  the  bottom;  by  $m$  the number of saddle connections which are
presented  twice  on  the  top,  and  by  $n$  the  number  of saddle
connections  which are presented twice on the bottom. It is immediate
to  see  that  if  the  original flat surface belongs to some stratum
$\cQ(d_1,\dots,d_k)$  of  meromorphic quadratic differentials with at
most    simple    poles,    then    $l+m+n=d$,    where    $d=\dim_\C
\cQ(d_1,\dots,d_k)=2g+k-2$.  Since  the  measured  foliation  is  non
orientable, both $m$ and $n$ are strictly positive.

In the rest of the paper, for quadratic differentials, we normalize
the volume using the first convention in \S\ref{sect:dep}, that is,
as a lattice in $H^1_-(\hat S, \{\hat P_1, \dots, \hat P_k\}; \C)$,
we choose the subset of those linear forms which, restricted on
$H_1^-(\hat S, \{\hat P_1, \dots, \hat P_k\};\Z)$, take values in
$\Z\oplus i\Z$.

\begin{Proposition}
\label{pr:contribution:quadratic}
The  contribution  of  any  $1$-cylinder  non  orientable  separatrix
diagram    $\cD$    to    the    volume    $\Vol\cQ_1(d_1,\dots,d_k)$
of    the    corresponding    stratum    of   meromorphic   quadratic
differentials with at most simple poles equals:
\begin{equation}
\label{eq:general:contribution}
c(\cD)=
\cfrac{2^{l+2}}{|\Gamma(\cD)|}\cdot
\frac{(m+n-2)!}{(m-1)!(n-1)!}\,
\cdot\cfrac{\mult_{-1}!\cdot\mult_1!\cdot \mult_2! \cdots}{(d-2)!}\,
\cdot\zeta(d)\,.
\end{equation}
Here  $|\Gamma(\cD)|$ is the order of the symmetry group of the
separatrix  diagram  (ribbon graph) $\cD$; $\mult_{-1}$ is the number
of  simple  poles;  $\mult_i$  is  the number of zeroes of order $i$;
$d=\dim_\C \cQ(d_1,\dots,d_k)=2g+k-2$; $m$ and $n$ are the numbers of
saddle  connections  which  are  presented only on top (respectively, on
bottom) boundary components of the cylinder.

Defining  the symmetry group $\Gamma(\cD)$ we assume that none of the
vertices,  edges, or boundary components of the ribbon graph $\cD$ is
labeled;  however,  we  assume  that  the orientation of the ribbons is
fixed.  Defining the volume $\Vol\cQ_1(d_1,\dots,d_k)$ we assume that
the zeroes and poles are numbered (labeled).
\end{Proposition}

Propositions~\ref{pr:contribution:Abelian}
and~\ref{pr:contribution:quadratic} are proved in
section~\ref{ss:contribution:of:one:1:cylinder:diagram:computation}.

Following    notation   of   \S A.2   in~\cite{Zagier}, denote by
$\mathbf{St}_n=\C^n/\C$  the  standard  irreducible representation of
dimension $n-1$ of the symmetric group $\mathfrak{S}_n$ and define
$$
\chi_j(g):=\operatorname{tr}(g,\pi_j)\qquad
\pi_j:=\wedge^j(\mathbf{St}_n)\qquad
(0\le j\le n-1)\,.
$$

\begin{Theorem}
\label{th:contribution:all:1:cyl}
The absolute contribution $c_1(m_1,\dots,m_\noz)$  of   all
$1$-cylinder orientable  separatrix diagrams   $\cD_\alpha$    to
the   volume $\Vol\cH_1(m_1,\dots,m_\noz)$ of   the   corresponding
stratum  of Abelian  differentials  equals
\begin{equation}
\label{eq:contribution:all:1:cyl}
 c_1(m_1,\dots,m_\noz)=
\frac{2}{n!}\cdot
\prod_k \frac{1}{(k+1)^{\mult_k}}\cdot
\sum_{j=0}^{n-1} j!\,(n-1-j)!\,\chi_j(\nu)
\,\cdot \zeta(n+1)\,.
\end{equation}
Here     $n=(m_1+1)+\dots+(m_\noz+1)=\dim_\C\cH(m_1,\dots,m_\noz)-1$;
$\nu\in \mathfrak{S}_n$ is any permutation which decomposes into cycles
of  lengths  $(m_1+1),\dots,(m_\noz+1)$;  $\mult_i$  is the number of
zeroes of order $i$, i.e. the multiplicity of the entry $i$ in the set
$\{m_1,\dots,m_\noz\}$. Speaking of the volume of
the  stratum  we  assume  that  the  zeroes $P_1,\dots,P_\noz$ of the
Abelian differentials are numbered (labeled).
\end{Theorem}

Applying  Theorem~\ref{th:contribution:all:1:cyl}  to  two particular
strata,  namely  to  the principal stratum and to the minimal one, we
get         a         close         expression        given        by
Corollary~\ref{cor:total:contribution:min:and:principal}.  For  other
strata     Theorem~\ref{th:contribution:all:1:cyl:estimate}     below
provides                           very                          good
estimate~\eqref{eq:contribution:all:1:cyl:estimate}               for
$c_1(m_1,\dots,m_\noz)$.

\begin{Corollary}
\label{cor:total:contribution:min:and:principal}
The absolute contribution  of all $1$-cylinder orientable separatrix
diagrams to  the  volume $\Vol\cH_1(1^{2g-2})$ of the principal
stratum and to the  volume  $\Vol\cH_1(2g-2)$  of  the  minimal
stratum  of Abelian differentials equals
\begin{align}
\label{eq:contribution:principal}
c_1(\underbrace{1,\dots,1}_{2g-2})=
\frac{\zeta(4g-3)}{4g-2}\cdot\frac{4}{2^{2g-2}}
\\
\label{eq:contribution:minimal}
c_1(2g-2)=\frac{\zeta(2g)}{2g}\cdot\frac{4}{2g-1}
\end{align}
\end{Corollary}

Theorem~\ref{th:contribution:all:1:cyl}                           and
Corollary~\ref{cor:total:contribution:min:and:principal}  are  proved
in section~\ref{ss:1:cylinder:diagrams:Abelian}.

\begin{Example}
We recall in section~\ref{ss:separatrix:diagrams} that a square-tiled
surface  in  the  stratum $\cH(2)$ may have one of the two separatrix
diagrams   $\cD_1,   \cD_2$   shown   in   Figure~\ref{fig:diag} (Appendix~\ref{s:overview:of:MV:volumes}).  By
Theorem~\ref{th:sep:diagram:is:essential}   each   separatrix   diagram
defines   an   essential   lattice   subset.   Square-tiled  surfaces
corresponding  to  separatrix  diagrams $\cD_1,\cD_2$ have one or two
maximal  cylinders  filled  with  closed regular horizontal geodesics
respectively.   Direct  computations~\eqref{eq:D1}  and~\eqref{eq:D2}
reproduced  from~\cite{Zorich:square:tiled}  imply that the constants
$c(\cD_i)$  corresponding  to  the  induced essential lattice subsets
have values
\begin{equation}
\label{eq:contributions:1:2:for:H2}
 c(\cD_1)=\frac{2}{3!}\cdot \zeta(4)
\qquad\qquad
c(\cD_2)=\cfrac{2}{3!} \cdot \cfrac{5}{4} \cdot \zeta(4)\,.
\end{equation}
Note that the  separatrix  diagram  $\cD_1$  has  symmetry  of  order
$3$,  so $|\Gamma(\cD_1)|=3!$,       and       the       value
$c(\cD_1)$ matches~\eqref{eq:contribution:numbered}. We have :

\[\Vol\cH_1(2)=c(\cD_1)+c(\cD_2)=\frac{3}{4}\zeta(4)=\frac{\pi^4}{120}.\]

Morally,    Theorem~\ref{th:equidistribution}    implies    that    a
``random''  Abelian  differential  with  rational periods in any open
subset   in   $\cH(2)$   would   have  single  maximal  horizontal
cylinder   filling   the   entire   surface  with  probability  $4/9$
and   two   horizontal  cylinders  of  different  perimeters  filling
together the   entire  surface with probability $5/9$.
\end{Example}

The  next  example  shows  that the values of similar proportions for
more complicated strata become much more elaborate.

\begin{Example}
\label{ex:H31}
A  square-tiled surface in the stratum $\cH(3,1)$ might have from $1$
to  $4$  cylinders.  Taking  the  sums  of  $c(\cD_\alpha)$  for  $4$
one-cylinder   diagrams   in   the   stratum   $\cH(3,1)$,  $30$
two-cylinder  diagrams,  $44$  three-cylinder  diagrams, and $10$
four-cylinder  diagrams  (here  the  numbers  of  oriented separatrix
diagrams are given without any weights), and computing the
proportions, or probabilities
\[p_i(\cH(3,1))=\frac{c_i(3,1)}{\Vol\cH_1(3,1)}\] we get
(see~\cite{Zorich:square:tiled}):
\begin{equation*}
\begin{split}
 \prop_1(\cH(3,1))& =
\cfrac{3\,\zeta(7)}{16\,\zeta(6)}\approx 0.19 \,,  \\  & \\ % -----
 \prop_2(\cH(3,1))& =
\cfrac{55\,\zeta(1,6) + 29\,\zeta(2,5) + 15\,\zeta(3,4) + 8\,\zeta(4,3) + 4\,\zeta(5,2)}
    {16\,\zeta(6)}\approx 0.47 \,, \\  & \\ % ---------
\prop_3(\cH(3,1))&  =
\cfrac{1}{32\,\zeta(6)}
\bigg( 12\,\zeta(6) - 12\,\zeta(7) + 48\,\zeta(4)\,\zeta(1,2) + 48\,\zeta(3)\,\zeta(1,3) \\ &
  + 24\,\zeta(2)\,\zeta(1,4) + 6\,\zeta(1,5) - 250\,\zeta(1,6) - 6\,\zeta(3)\,\zeta(2,2)
\\ &
 -  5\,\zeta(2)\,\zeta(2,3) +
6\,\zeta(2,4) - 52\,\zeta(2,5) + 6\,\zeta(3,3) - 82\,\zeta(3,4)
\\ &
 + 6\,\zeta(4,2) - 54\,\zeta(4,3) + 6\,\zeta(5,2) + 120\,\zeta(1,1,5) - 30\,\zeta(1,2,4)
\\ &
  - 120\,\zeta(1,3,3) - 120\,\zeta(1,4,2) - 54\,\zeta(2,1,4) - 34\,\zeta(2,2,3)
\\ &
  - 29\,\zeta(2,3,2) - 88\,\zeta(3,1,3) - 34\,\zeta(3,2,2) - 48\,\zeta(4,1,2)
\bigg)\approx 0.30  \,, \\  & \\ % ---------
 \prop_4(\cH(3,1))& = \cfrac{\zeta(2)}{8\,\zeta(6)}
\,\bigg( \zeta(4) - \zeta(5) + \zeta(1,3) + \zeta(2,2) - \zeta(2,3) - \zeta(3,2) \bigg)\approx 0.04 \,.
\end{split}
\end{equation*}
\end{Example}

Note  that for separatrix diagrams $\cD_\alpha$ with $k>1$ cylinders,
the  contribution  $c(\cD_\alpha)$ of the diagram varies from diagram
to  diagram, and even in the example above the contribution of an
individual diagram is not necessarily reduced
to  a  polynomial in multiple zeta values with rational coefficients.

\begin{Question}
Is it true that the total contribution of all $k$-cylinder separatrix
diagrams  to  the volume of any stratum of Abelian differentials is a
polynomial  in  multiple  zeta values with rational (or even integer)
coefficients?
\end{Question}

%------------------------------------------------------------------
\subsection{Asymptotics in large genera}
\label{ss:Asymptotics:in:large:genera}

Theorem~\ref{th:contribution:all:1:cyl}   combined   with   Theorem~2
in~\cite{Zagier:bounds} provides the following result which is proved
in section~\ref{ss:1:cylinder:diagrams:Abelian}.

\begin{Theorem}
\label{th:contribution:all:1:cyl:estimate}
The absolute  contribution   $c_1(m_1,\dots,m_\noz)$   of   all
$1$-cylinder orientable      separatrix      diagrams      to
the      volume $\Vol\cH_1(m_1,\dots,m_\noz)$    of    any    stratum
of   Abelian differentials satisfies the following bounds
\begin{equation}
\label{eq:contribution:all:1:cyl:estimate}
\frac{\zeta(d)}{d+1}\cdot
%\prod_k \frac{1}{(k+1)^{\mult_k}}
\frac{4}{(m_1+1)\dots(m_\noz+1)}
  %\le\\
\le c_1(m_1,\dots,m_\noz)
\le
\frac{\zeta(d)}{d-\frac{10}{29}}
\cdot
%\prod_k \frac{1}{(k+1)^{\mult_k}}
\frac{4}{(m_1+1)\dots(m_\noz+1)}
,
\end{equation}
where $d=\dim_\C\cH(m_1,\dots,m_\noz)$.
\end{Theorem}

To   discuss   the   asymptotic  behavior  of  the  \textit{relative}
contribution $\prop_1(\cH(m_1,\dots,m_\noz))$ for the strata of large
genera  we  need to recall the conjecture on asymptotics of volumes
of the  connected components of the strata of Abelian differentials,
see \cite{Eskin:Zorich} for details. Let $m=(m_1,\dots,m_\noz)$  be
an  unordered  partition  of  an even number $2g-2$,  i.e.,  let
$|m|=m_1+\dots+m_\noz=2g-2$. Denote by $\Pi_{2g-2}$ the
set              of             all             such partitions.

\begin{Conjecture}[{\cite[Main Conjecture 1]{Eskin:Zorich}}]
\label{conj:vol}
For any $m\in\Pi_{2g-2}$ one has
\begin{equation}
\label{eq:asymptotic:formula:for:the:volume}
\Vol\cH_1(m_1,\dots,m_\noz)=\cfrac{4}{(m_1+1)\cdot\dots\cdot(m_\noz+1)}
\cdot (1+\varepsilon_1(m)),
\end{equation}
where
$$
\lim_{g\to\infty} \max_{m\in\Pi_{2g-2}} \varepsilon_1(m) = 0.
$$
   %
% For any $k\in\Pi_{g-1}$ one has
%    %
% \begin{equation}
% \label{eq:even:over:odd}
%    %
% \cfrac{\Vol\cH^{even}_1(2k_1,\dots,2k_n)}
% {\Vol\cH^{odd}_1(2k_1,\dots,2k_n)}
% = (1+\varepsilon_2(k)),
% \end{equation}
%    %
% where
%    %
% $$
% \lim_{g\to\infty} \max_{k\in\Pi_{g-1}} \varepsilon_2(k) = 0.
% $$
%    %
\end{Conjecture}

% The  result  in~\cite{AEZ:genus:0}  providing  the  exact  value  for
% the  hyperelliptic  connected  components  show  that their volume is
% negligible      in      comparison      with      the     conjectural
% volume~\eqref{eq:asymptotic:formula:for:the:volume}  of  the  stratum
% when genus is sufficiently large.

\begin{condTheorem}
\label{th:contribution:1:cyl:g:to:infty}
Denote  by  $d=\dim_\C \cH(m_1,\dots,m_\noz)=2g+\noz-1$ the dimension
of  the  stratum  $\cH(m_1,\dots,m_\noz)$  of  Abelian differentials in
genus  $g$.  Let  $\prop_1(\cH(m_1,\dots,m_\noz))$  be  the  relative
contribution of $1$-cylinder separatrix diagrams to the volume of the
stratum $\cH(m_1,\dots,m_\noz)$.

Conjecture~\ref{conj:vol} is equivalent to the following statement:
\begin{equation}
\label{eq:asymptotics:for:1:cyl}
d\cdot\prop_1(\cH(m_1,\dots,m_\noz)) \to 1 \text{ as }g\to+\infty\,,
\end{equation}
where the convergence is uniform
for all strata in genus $g$.
\end{condTheorem}

This Theorem is a straightforward corollary of Theorem~\ref{th:contribution:all:1:cyl:estimate}.

Note    that    for    the    principal    stratum   $\cH(1,\dots,1)$
Conjecture~\ref{conj:vol}    was    recently   proved   by   D.~Chen,
M.~M\"oller,   and  D.~Zagier~\cite{Chen:Moeller:Zagier}.  Thus,  the
Conditional Theorem \ref{th:contribution:1:cyl:g:to:infty} becomes
unconditional   for   the   principal strata.
Formula~\eqref{eq:contribution:principal} combined with the
asymptotic expansion evaluated in Theorem 19.2
in~\cite{Chen:Moeller:Zagier} provide the following statement.

\begin{Theorem}
The  relative
contribution $\prop_1(\cH(1^{2g-2}))$ of $1$-cylinder separatrix
diagrams to the volume of the principal stratum
of Abelian differentials
$\cH(1^{2g-2}))$
in genus $g$ satisfies the following asymptotic formula:
\begin{equation}
\label{eq:asymptotics:for:1:cyl:princ}
(4g-3)\cdot\prop_1(\cH(\underbrace{1,\dots,1}_{2g-2}))=
1+\frac{\pi^2-6}{24g} +o\left(\frac{1}{g}\right) \text{ as }g\to+\infty\,.
\end{equation}
\end{Theorem}

  It  would  be  very interesting to find an argument proving
asymptotics~\eqref{eq:asymptotics:for:1:cyl}    for    $\prop_1(\cH)$
directly,    and    thus    prove    the    conjectural    asymptotic
formula~\eqref{eq:asymptotic:formula:for:the:volume}  for the volumes
of all strata.

Recall  that  some  strata  are not connected. However, all the above
results  can  be easily generalized to connected components. We start
with  the  hyperelliptic  connected  components $\cH^{hyp}(2g-2)$ and
$\cH^{hyp}(g-1,g-1)$,  which  are  always very special and do not fit
the  general picture. The situation is particularly simple with them.
The results  in~\cite{AEZ:genus:0} provide a simple closed formula for the
volume  of  these components. These volumes are completely negligible
with               respect               to               conjectural
volume~\eqref{eq:asymptotic:formula:for:the:volume}   of  the  entire
strata.  On the other hand, each hyperelliptic component has a unique
$1$-cylinder  separatrix diagram $\cD$, which has the cyclic symmetry
group    $\Gamma(\cD)$    of    order    $d-1$   (see~Proposition   5
in~\cite{Zorich:representatives}).  Thus,  the  contribution $c_1$ of
all     $1$-cylinder     diagrams  is basically    given    by
Proposition~\ref{pr:contribution:Abelian}.

\begin{Proposition}
\label{pr:proportion:hyp}
The  relative contribution $p_1$ of $1$-cylinder separatrix diagrams to
the volumes of the hyperelliptic components is given by the following
expressions:
\begin{align*}
%\label{eq:prop:hyp}
\prop_1(\cH^{hyp}_1(2g-2))&=\cfrac{\zeta(2g)}{\pi^{2g}}\cdot
2g(2g+1)\cdot
\cfrac{(2g-2)!!}{(2g-3)!!} \sim
4\cdot\frac{g^{5/2}}{\pi^{2g-1/2}}\,.
\\
\prop_1(\cH^{hyp}_1(g-1,g-1))&=\cfrac{\zeta(2g+1)}{2\pi^{2g}}\cdot
(2g+1)(2g+2)\cdot
\cfrac{(2g-1)!!}{(2g-2)!!} \sim
4\cdot\frac{g^{5/2}}{\pi^{2g+1/2}}\,.
\end{align*}
\end{Proposition}

Proposition~\ref{pr:proportion:hyp} shows that the resulting relative
contribution  $p_1$  of $1$-cylinder separatrix diagrams to the volumes
of the hyperelliptic components is completely negligible with respect
to~\eqref{eq:asymptotics:for:1:cyl}.  It  is  proved  in  the  end of
section~\ref{ss:1:cylinder:diagrams:Abelian}.

It     remains     to     consider     nonhyperelliptic    components
$\cH^{even}(2m_1,\dots,2m_\noz)$ and $\cH^{odd}(2m_1,\dots,2m_\noz)$.
Recall another conjecture from \cite{Eskin:Zorich}:

% \begin{Theorem}
% \label{th:even:odd:1:cyl}
%    %
% The ratio of the numbers of $1$-cylinder separatrix diagrams $\cD_\alpha$
% counted with weight $1/|\Gamma(\cD_\alpha)|$
% corresponding to even and odd components
% of a stratum
% $\cH(2m_1,\dots,2m_\noz)$
% tends to one uniformly for all partitions $2m_1+\dots+2m_\noz=2g-2$
% as genus $g$ tends to infinity.
% \end{Theorem}

\begin{Conjecture}[{\cite[Conjecture 2]{Eskin:Zorich}}]
\label{conj:even:odd}
The ratio of volumes
of  even and odd components of strata $\cH(2m_1,\dots,2m_\noz)$ tends
to  $1$  uniformly for all partitions $m_1+\dots+m_\noz=g-1$ as genus
$g$ tends to infinity, i.~e.
$$
\lim_{g\to+\infty}\frac
{\Vol\cH^{even}(2m_1,\dots,2m_\noz)}
{\Vol\cH^{odd}(2m_1,\dots,2m_\noz)}
=1
$$
	uniformly in $m_1,\dots , m_\noz$.
\end{Conjecture}

By the result in~\cite[Theorem 4.19]{Delecroix}, the ratio of the weighted numbers
of $1$-cylinder separatrix diagrams in the connected components
$\cH^{even}_1(2m_1,\dots,2m_\noz)$ and $\cH^{odd}_1(2m_1,\dots,2m_\noz)$ also tends to $1$
uniformly for all partitions $m_1+\dots+m_\noz=g-1$ as genus
$g$ tends to infinity. Thus,  we  obtain the following statement.

\begin{condCorollary}
\label{cor:even:odd:1:cyl}
Conjecture~\ref{conj:even:odd} and Conjecture~\ref{conj:vol}
restricted  to  the  strata with zeroes of even  degrees  are
together  equivalent  to  the  following  statement: for any
partition  $(m_1,\dots,m_\noz)$  of $g-1$ into a sum of strictly
positive integers $m_1+\dots+m_\noz=g-1$ one has
\begin{align*}
% \label{eq:asymptotics:even}
d\cdot\prop_1(\cH^{even}(2m_1,\dots,2m_\noz)) &\to 1 \text{ as }g\to+\infty\\
% \label{eq:asymptotics:odd}
d\cdot\prop_1(\cH^{odd}(2m_1,\dots,2m_\noz)) &\to 1 \text{ as }g\to+\infty\,,
\end{align*}
where  $d=2m_1+\dots+2m_\noz+\noz+1$  and  convergence is uniform for
all strata in genus $g$.
\end{condCorollary}

% Conditional Corollary~\ref{cor:even:odd:1:cyl} is proved in
% section~\ref{s:1:cylinder:diagrams:Abelian}.

Since we do not want to overload the current paper, the questions
concerning the asymptotic proportions
$\prop_k(\cH(m_1,\dots,m_\noz))$ of   $k$-cylinder   diagrams   for
$k=2,3,\dots$  for  strata  of  high  genera will be address in a
separate paper. In this forthcoming paper we will treat, in particular,
the question of the dependence of $p_k$ on the genus and the
dimension of the stratum, and the question of the limit distribution
of $p_k$ with respect   to   all   possible   $k$   for  strata  of
large  genera.

%--------------------------------------------------------------------
\subsection{Application: experimental evaluation of the Masur--Veech volumes}
\label{ss:Application:experimental:evaluation:of:MV:volumes}
Let $\cH^{comp}$ (respectively, $\cQ^{comp}$) be a component of a stratum of Abelian differentials
(respectively, a component of a stratum of meromorphic
quadratic differentials with at most simple poles). We first present
a Monte-Carlo method to approximate $\prop_1(\cH^{comp})$
(respectively $\prop_1(\cQ^{comp})$). Pick  a random trajectory of
the generalized Rauzy induction in the  space of interval exchange
transformations (respectively, in the space  of linear involutions).
Stop at a random time, and take a small box  $\Pi$  around  the
endpoint  of the trajectory. Then collect the statistics  of
frequency of those interval exchange transformations whose
suspensions are filled with a single  vertical cylinder. By Theorem
\ref{th:iet} and Corollary \ref{cor:k:cyl:iet} (respectively Remark
\ref{rem:lin:invol}), this frequency gives an approximation of the
relative  contribution $\prop_1(\cH^{comp})$ (respectively
$\prop_1(\cQ^{comp})$)  of $1$-cylinder diagrams to the volume of the
chosen component of the stratum.

Now, one can perform an  exact  count  of  the weighted number of
$1$-cylinder separatrix   diagrams  (where  the  weight  is
reciprocal to the order of  the  symmetry  group  of  the
diagram). Applying Proposition~\ref{pr:contribution:Abelian}
(respectively, Proposition~\ref{pr:contribution:quadratic})   we
obtain  the  exact value $c_1(\cH^{comp})$ (respectively $c_1(\cQ^{comp})$) of  the  contribution  of
$1$-cylinder diagrams to the volume. Since,  we  already  know
approximately, what part of the total value makes  the  resulting
volume, we obtain an approximate value of the volume of the ambient
stratum. The experimental and theoretical values of the volumes of
low dimensional strata of quadratic differentials are compared in Appendix
\ref{s:tables:volumes}.

%----------------------------------------------------------------------
\subsection{Contribution of a single $1$-cylinder separatrix diagram: computation}
\label{ss:contribution:of:one:1:cylinder:diagram:computation}

Consider  Jenkins--Strebel  differentials
represented  by  a  \emph{single}  flat cylinder $C$ filled by closed
horizontal leaves. Note that all zeroes and poles (critical points of
the  horizontal  foliation)  of  such differential are located on the
boundary of this cylinder.

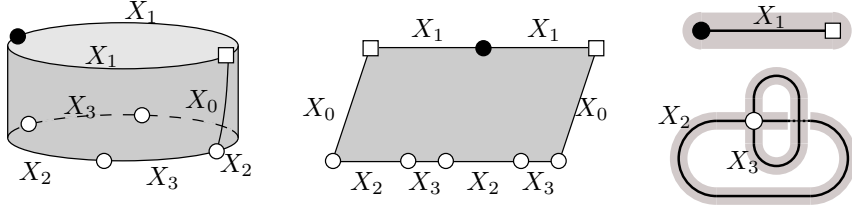
\begin{figure}[htb]
\includegraphics{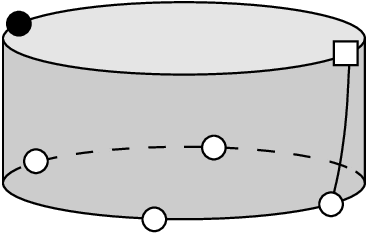}
\includegraphics{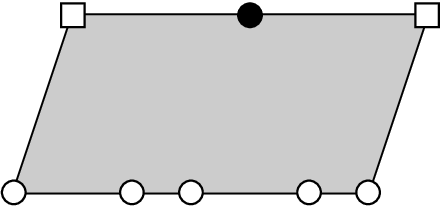}
\includegraphics{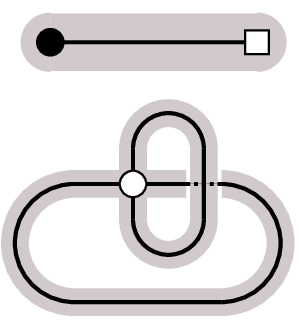}
\begin{picture}(0,0)(160,5)
\put(40,7){$X_1$}
\put(25,-10){$X_1$}
\put(17,-29){$X_3$}
\put(63,-27){$X_0$}
\put(0,-54){$X_2$}
\put(48,-56){$X_3$}
\put(76,-51){$X_2$}
\end{picture}
\begin{picture}(0,0)(40,-105)
\put(25,-110){$X_1$}
\put(69,-110){$X_1$}
\put(2,-168){$X_2$}
\put(24,-168){$X_3$}
\put(46,-168){$X_2$}
\put(67,-168){$X_3$}
\put(-16,-140){$X_0$}
\put(87,-140){$X_0$}
\end{picture}
\begin{picture}(0,0)(-73,-100)
\put(38,-100){$X_1$}
\put(2,-138){$X_2$}
\put(28,-155){$X_3$}
\end{picture}
\vspace{60bp}
\caption{
\label{fig:Jenkins:Strebel}
A  Jenkins--Strebel  differential  with a single cylinder,
one of its parallelogram  patterns, and its ribbon graph representation.
We have $l=0, m=1, n=2$. The stratum is $\cQ(2,-1^2)$.
}
\end{figure}

Each of the two boundary components $\partial C^+$ and $\partial C^-$
of  the cylinder is subdivided into a collection of horizontal saddle
connections $\partial C^+=X_{\alpha_1}\sqcup\dots\sqcup X_{\alpha_r}$
and  $\partial  C^-=X_{\alpha_{r+1}}\sqcup\dots\sqcup  X_{\alpha_s}$.
The  subintervals are naturally organized in pairs of
equal  length; subintervals in every pair are identified by a natural
isometry  which  preserves  the  orientation of the surface. Denoting
both subintervals in the pair representing the same saddle connection
by  the same symbol, we encode the combinatorics of identification of
the boundaries of the cylinder by two lines of symbols,

\begin{equation}
\label{eq:js:prepermutation}
\begin{picture}(0,0)(-2,0)
\put(-3,10){\vector(1,0){0}}
\put(-5,15){\oval(10,10)[bl]}
\put(30,15){\oval(80,10)[t]}
\put(65,15){\oval(10,10)[br]}
\put(-3,-4){\vector(1,0){0}}
\put(-5,-9){\oval(10,10)[tl]}
\put(40,-9){\oval(100,10)[b]}
\put(85,-9){\oval(10,10)[tr]}
\end{picture}
\begin{matrix}\alpha_1&\dots&\alpha_r&\\ \alpha_{r+1}\!\!&\dots\!&\!\dots\!&\!\alpha_s\end{matrix}
\vspace{8pt}
\end{equation}

\noindent
where the symbols in each line are organized in a cyclic order.
\medskip

%--------------------------------------------------------------------
\noindent\textbf{Choice of cyclic ordering.}
There  are  two  alternative conventions on the choice of this cyclic
order.  Note  that our surface is oriented (and not only orientable).
Hence,  this  orientation  induces  a  natural orientation of each of
$\partial  C^+$ and of $\partial C^-$ which defines a cyclic order on
the symbols labeling the segments.

Note  also  that  if  we have an Abelian differential, its horizontal
foliation  is  oriented.  The  corresponding  orientation  of  leaves
defines  the  same  cyclic  order as the previous one on one boundary
component  of the cylinder and the opposite cyclic order on the other
boundary component of the cylinder.

For  quadratic differentials the foliation is nonorientable. However,
for  a  Jenkins--Strebel  differential  we  can coherently choose the
orientation  of  all  regular  leaves in the interior of each maximal
cylinder,  and  it  induces  the  cyclic order of symbols labeling the
segments on $\partial C^+$ and $\partial C^-$. Similarly to the case of
Abelian  differentials, this  cyclic  ordering coincides with the one
induced  by  the  orientation  of  the  surface  on  one  of  the two
components  $\partial C^+, \partial C^-$ and provides the opposite cyclic
ordering on the other component.

In~\eqref{eq:js:prepermutation}
we  use  the  cyclic  ordering  coming  from  the orientation of the
foliation and not from the orientation of the surface.
\medskip

%--------------------------------------------------------------------
\noindent\textbf{Abelian versus quadratic differentials.}
By   construction,   every   symbol  appears  exactly  twice  in  two
lines~\eqref{eq:js:prepermutation}.  If  all the symbols in each line
are  distinct, the resulting flat surface has trivial linear holonomy
and  corresponds  to  an  Abelian  differential.  In  this case every
interval  on one side of the cylinder is identified with and interval
on  the  other  side  and vice versa, so there are no relations
between  the lengths of the intervals. In other words, any orientable
separatrix diagram having only two boundary components is realizable.

Otherwise, a flat metric of the resulting closed surface has holonomy
group  $\ZZ$;  in  the  latter  case  it corresponds to a meromorphic
quadratic  differential with at most simple poles. In this case there
is a linear relation between the lengths of the intervals: the sum of
lengths  of all intervals on one side of the cylinder is equal to the
sum  of  lengths of all intervals on the other side. This implies the
following  combinatorial  restriction: the set of symbols in one line
cannot  be a proper subset of the set of symbols on the complimentary
line.  This  condition  is  a  necessary  and sufficient condition of
realizability  for  a non-orientable separatrix diagram. For example,
the following combinatorial data

\begin{equation*}
\begin{picture}(0,0)(-2,0)
\put(7,10){\vector(1,0){0}}
\put(5,15){\oval(10,10)[bl]}
\put(25,15){\oval(50,10)[t]}
\put(45,15){\oval(10,10)[br]}
\put(7,-4){\vector(1,0){0}}
\put(5,-9){\oval(10,10)[tl]}
\put(40,-9){\oval(80,10)[b]}
\put(75,-9){\oval(10,10)[tr]}
\end{picture}
\begin{matrix}&1&2&3&\\ &3&4&1&2&4\end{matrix}
\vspace{8pt}
\end{equation*}
do  not  admit  any  strictly  positive  solution  for the lengths of
subintervals, while

\begin{equation*}
\begin{picture}(0,0)(-2,0)
\put(-3,10){\vector(1,0){0}}
\put(-5,15){\oval(10,10)[bl]}
\put(30,15){\oval(80,10)[t]}
\put(65,15){\oval(10,10)[br]}
\put(-3,-4){\vector(1,0){0}}
\put(-5,-9){\oval(10,10)[tl]}
\put(30,-9){\oval(80,10)[b]}
\put(65,-9){\oval(10,10)[tr]}
\end{picture}
\begin{matrix}5&1&2&3&\!5\\ 3&4&1&2&\!4\end{matrix}
\vspace{8pt}
\end{equation*}
admits   strictly   positive   solutions   satisfying   the  relation
$\ell_4=\ell_5$.
\bigskip

%--------------------------------------------------------------------
\noindent\textbf{Contribution of each individual $1$-cylinder separatrix diagram.}
Now everthing is ready for the proofs of
Propositions~\ref{pr:contribution:Abelian}
and~\ref{pr:contribution:quadratic}.

\begin{proof}[Proof of Proposition~\ref{pr:contribution:Abelian}]
An  orientable  $1$-cylinder  separatrix diagram $\cD$ representing a
stratum  of  Abelian differentials of complex dimension $d$ has $d-1$
separatrices (horizontal saddle connections). Denote the length of
the $i$-th separatrix by $\ell_i$. The  perimeter  $w$  of  the
cylinder  is  equal  to  the sum of the lengths          of
all         separatrices,         namely
$w=\ell_1+\ell_2+\dots+\ell_{d-1}$.  Denote  by $h$ the height of the
cylinder.   Finally,   denote   by   $\phi$   the   ``twist'',  where
$0\le\phi<w$.  The  number  of  square-tiled  surfaces  tiled with at
most  $N$  unit  squares  and  having $\cD$ as the separatrix diagram
equals
\begin{multline*}
\cfrac{1}{|\Gamma(\cD)|}\ \sum_{\substack{\ell_1,\dots,\ell_{d-1},h\in\N\\
w=\ell_1+\dots+\ell_{d-1}\\
w\cdot h\le N}}
w \ \approx \
\cfrac{1}{|\Gamma(\cD)|}\
\sum_{\substack{w,h\in\N\\w\cdot h\le N}}
w\cdot\cfrac{w^{d-2}}{(d-2)!}
\ =\
\cfrac{1}{|\Gamma(\cD)|}\
\cfrac{1}{(d-2)!}\ \sum_{\substack{w,h\in\N\\w\le \frac{N}{h} }} w^{d-1}
\ \approx \\ \approx \,
\cfrac{1}{|\Gamma(\cD)|}\
\cfrac{1}{(d-2)!}\, \sum_{h\in\N} \cfrac{1}{d}\,\cdot
\left(\cfrac{N}{h}\right)^d =\,
\cfrac{1}{|\Gamma(\cD)|}\
\cfrac{N^d}{(d-2)!}\,
\cfrac{1}{d}\,
\cdot\sum_{h\in\N} \cfrac{1}{h^d}
\, =\\= \,
\cfrac{1}{|\Gamma(\cD)|}\
\cfrac{1}{d}\cdot
\cfrac{N^d}{(d-2)!}\,\cdot
\zeta(d)\,,
\end{multline*}
(compare                     to~\eqref{eq:D1}).                    By
equation~\eqref{eq:volume:from:square:tiled}  the contribution of any
such  term  to  the  volume $\Vol\cH_1(m_1,\dots,m_\noz)$ of the stratum
with   \textit{unnumbered}  zeroes  is  computed  evaluating the derivative
$\left.2\cfrac{d}{dN}\right|_{N=1}$.   Thus,   the   contribution   of   the
$1$-cylinder  separatrix  diagram  $\cD$ to the volume of the ambiant
stratum is
\begin{equation*}
\label{eq:contribution:not:numbered}
\cfrac{1}{|\Gamma(\cD)|}\cdot
\cfrac{2}{(d-2)!}\,\cdot \zeta(d)\,.
\end{equation*}

Representing  the  set  $\{m_1,\dots,m_\noz\}$  as
$\{1^{\mult_1},2^{\mult_2},\dots\}$  we get the following formula for
the  contribution  of  an  individual  rooted  diagram  to the volume
$\Vol\cH^{numbered}_1(m_1,\dots,m_\noz)$    of    the    stratum    with
\textit{numbered} zeroes:
$$
\cfrac{2}{|\Gamma(\cD)|}\cdot
\cfrac{\mult_1!\cdot \mult_2! \cdots}{(d-2)!}\,\cdot \zeta(d)\,.
$$
which completes the proof of Proposition~\ref{pr:contribution:Abelian}.
\end{proof}

\begin{proof}[Proof of Proposition~\ref{pr:contribution:quadratic}]
The  evaluation of the contribution of an $1$-cylinder diagram to the
volume of a stratum of quadratic differentials is analogous. The only
difference  is  that  it  gets  an  extra  weight  depending  on  the
additional discrete parameters $l,m,n$ of the diagram.

Consider a nonorientable $1$-cylinder separatrix diagram. Each
separatrix (i.e. each horizontal saddle  connection) is represented
by two intervals on  the boundary of the cylinder. One may have one
interval on each of the two boundary  components,  both intervals on
the ``top'' boundary component of the cylinder, or both on the
``bottom'' boundary component. Recall that we denote the number   of
corresponding   saddle connections  by  $l,m,n$  correspondingly.

We start with a more general situation when $l>0$. Introduce
the following notation:
\begin{align*}
w_1&:=\ell_{i_1}+\dots+\ell_{i_l}\\
w_2&:=2(\ell_{j_1}+\dots+\ell_{j_m})=2(\ell_{k_1}+\dots+\ell_{k_n})\,,
\end{align*}
where  by  $\ell_{i_s}$,  $s=1,\dots,l$  we denote the lengths of the
segments  which  are  present on the both  sides  of the cylinder, by
$\ell_{j_s}$,  $s=1,\dots,m$  we  denote  the lengths of the segments
which are present only on top of the cylinder, and by $\ell_{k_s}$,
$k=1,\dots,n$  we  denote  the  lengths  of  the  segments  which are
present  only  on  the  bottom  of  the  cylinder.  For example, on
Figure~\ref{fig:Jenkins:Strebel} the segment  $X_1$ is present only on the
top,  the segments  $X_2,  X_3$  --- only on the bottom, and there are no
other  segments,  so  we have $l=0, m=1, n=2$.

In this notation the
length $w$ of the waist curve (perimeter) of the cylinder is equal to
$w=w_1+w_2$.  When  $l>0$  (that  is  when the boundary components of the
cylinder share at least one common interval) the waist curve $\gamma$
of  the  cylinder is not homologous to zero. Under our assumptions on
the normalization (see Convention~\ref{conv:lattice} for details)
the   lengths   $\ell_s$  of  all  subintervals  are
half-integers,  $w_1$ is a half-integer, $w_2$ is automatically an integer,
and $w$ is a half-integer.

The  leading term in the number of ways to represent $w_1$ as a sum
of $l$ half-integers
$$
w_1=\ell_{i_1}+\dots+\ell_{i_l}
$$
is
$$
2^{l-1}\cfrac{w_1^{l-1}}{(l-1)!}\,.
$$
The  leading term in the number of ways to represent $w_2$ as a sum
of $m$ (respectively $n$) integers
$$
w_2=2\ell_{j_1}+\dots+2\ell_{j_m}=2\ell_{k_1}+\dots+2\ell_{k_n}
$$
is
$$
\cfrac{w_2^{m-1}}{(m-1)!}
\qquad\left(\text{respectively }\
\cfrac{w_2^{n-1}}{(n-1)!}
\right)
\,.
$$

Denote by $h$ the half-integer height of our single cylinder and
introduce the integer parameter $H=2h$. The condition $w\cdot h\le
N/2$ on the area of the surface translates as $w\cdot H\le N$ in
terms of the parameter $H$. Thus, introducing the notation $W:=2w$,
we can represent the leading term in the corresponding sum as
\begin{multline*}
\sum_{\substack{w\in\frac{1}{2}\N\\H\in\N\\
w\cdot H\le N}}\
\sum_{\substack{w_2\in\N\\ w_2< w}}
2w\cdot 2^{l-1}\cfrac{(w-w_2)^{l-1}}{(l-1)!}
\cdot\cfrac{w_2^{m-1}}{(m-1)!}\cdot\cfrac{w_2^{n-1}}{(n-1)!}
=\\=
\frac{2^{l-1}}{(l-1)!(m-1)!(n-1)!}
\sum_{\substack{W, H\in\N\\W\cdot H\le 2N}} W
\sum_{w_2=1}^{\lfloor W/2\rfloor}(W/2-w_2)^{l-1}w_2^{m+n-2}
\sim \\ \sim
\frac{2^{l-1}}{(l-1)!(m-1)!(n-1)!}
%\cdot \\
 \cdot
 \sum_{H\in \N}\sum_{W=1}^{\lfloor{2N/H}\rfloor}
W\cdot
\left(\frac{W}{2}\right)^{l+m+n-2}
\cdot \int_0^1(1-u)^{l-1}u^{m+n-2}\,du
%\sim
 \\ \sim
\frac{2^{l-1}}{(l-1)!(m-1)!(n-1)!}
\cdot \frac{(l-1)!(m+n-2)!}{(l+m+n-2)!}
\cdot\\
 \cdot \frac{1}{2^{l+m+n-2}} \cdot
\sum_{H\in\N} \cfrac{1}{l+m+n}\cdot\left(\cfrac{2N}{H}\right)^{l+m+n}
%=
\\ =
\frac{2^{l+1}(m+n-2)!}{(m-1)!(n-1)!(l+m+n-2)!}\cdot
\cfrac{N^{l+m+n}}{l+m+n}\cdot\zeta(l+m+n)\,.
\end{multline*}
where we used the relation
$$
\int_0^1 u^a (1-u)^b\, du = \cfrac{a!\, b!}{(a+b+1)!}\,.
$$
Taking the derivative $\left.2\cdot\cfrac{d}{dN}\right|_{N=1}$ we get
the following contribution to the volume of the corresponding stratum
with \textit{anonymous} (\textit{non-numbered}) zeroes and poles:
\begin{equation*}
%\label{eq:general:contribution}
\frac{2^{l+2}(m+n-2)!}{(m-1)!(n-1)!(l+m+n-2)!}
\cdot\zeta(l+m+n)
\end{equation*}

Multiplying  the  result by the product of factorials responsible for
numbering    the    zeroes    and   poles,   we   get   the   desired
formula~\eqref{eq:general:contribution}.

In  the  remaining  particular  case  when  $l=0$  (that  is, when the
boundary  components  of  the  cylinder  do not share a single common
saddle  connection)  the  waist  curve  $\gamma$  of  the cylinder is
homologous  to zero, while $\hat\gamma$ is not. Under our assumptions
on the normalization,  the  lengths  $\ell_s$  of  all  subintervals are
half-integers,  and  $w=w_2$  is  automatically an integer, as it should
be. Performing a completely analogous computation we get a particular
case of formula~\eqref{eq:general:contribution} where $l=0$.
\end{proof}

%--------------------------------------------------------------------
\subsection{Counting $1$-cylinder diagrams for strata of Abelian
differentials based on Frobenius formula and Zagier bounds}
\label{ss:1:cylinder:diagrams:Abelian}

Enumeration  of  orientable  $1$-cylinder  separatrix
diagrams     through     Frobenius     formula     was     elaborated
in~\cite{Delecroix}.  Consider  some stratum of Abelian differentials
$\cH(m_1,\dots,m_\noz)$. Let
\begin{equation}
\label{eq:n}
n=\sum_{i=1}^r  (m_i+1)=2g-2+r=\dim_\C\cH(m_1,\dots,m_\noz)-1\,.
\end{equation}
Denote  by  $C(\psi)$  the conjugacy class of a permutation $\psi$ in
the  symmetric  group  $\mathfrak{S}_n$;  denote  by  $C(\sigma)$ the
conjugacy  class  of the cyclic permutation $\sigma=(1,2,\dots,n)$ in
$\mathfrak{S}_n$. Finally, denote by $C(\nu)$  the  conjugacy  class
of the product of $r$ cycles of lengths $(m_1+1,\dots,m_\noz+1)$.

Following~\cite{Zagier}                   denote                   by
$\cN(\mathfrak{S}_n;C(\sigma),C(\sigma),C(\nu))$    the    number    of
solutions  of the equation $c_1 c_2 c_3 =1$, where the permutations $c_1$
and   $c_2$  belong  to  the  conjugacy  class  $C(\sigma)$  and  the
permutation $c_3$ belongs to the conjugacy class $C(\nu)$:
\begin{multline}
\label{eq:N:C:C:C}
\cN(\mathfrak{S}_n;C(\sigma),C(\sigma),C(\nu))
=\\=
\#\{(c_1,c_2,c_3)\in C(\sigma)\times C(\sigma)\times C(\nu)\,|\,
c_1 c_2 c_3 =1\}\,.
\end{multline}

Every  such  solution  defines  a    $1$-cylinder  separatrix diagram
corresponding to the stratum $\cH(m_1,\dots,m_\noz)$. Indeed, consider a
horizontal  cylinder $S^1\times [0;1]$ such that each of its boundary
components  is  subdivided into $n$ segments. Choose the orientation of
the  boundary  components  induced by the orientation of the circle
$S^1$  (on  one of the two components it differs from the orientation
induced  from the orientation on the cylinder) and assign labels from
$1$  to $n$ to the subintervals of one boundary component in such a way
that  they  appear  in  the cyclic order $c_1$, and assign labels to the
remaining boundary component in such a way that they appear in the cyclic
order  $c^{-1}_2$.  Cut the cylinder along the horizontal waist curve
and  identify  pairs  of  subintervals  on  the  boundary  components
carrying  the  same  labels  respecting  the orientation induced from
$S^1$. Consider the $1$-cylinder separatrix diagram $\cD$ represented
by   the  resulting  ribbon  graph.  The  relation  $c_1\cdot  c_2  =
c^{-1}_3$,  where  $c_3\in C(\nu)$, guarantees that $\cD$ corresponds
to the stratum $\cH(m_1,\dots,m_\noz)$.

\begin{Example}(See~\cite{Zorich:representatives} for details.)
Consider the pair of cyclic permutations
$c_1=(1,2,3,4,5,6,7,8)$ and
$c_2=(4,3,2,5,8,7,6,1)$
in $\mathfrak{S}_8$. The two boundary components of the corresponding
horizontal cylinder get the following labeling:

\begin{equation}
\label{eq:H1111:cyclic}
\begin{picture}(-4,6)
\put(-3,10){\vector(1,0){0}}
\put(-5,15){\oval(10,10)[bl]}
\put(70,15){\oval(160,10)[t]}
\put(145,15){\oval(10,10)[br]}
\put(-3,-4){\vector(1,0){0}}
\put(-5,-9){\oval(10,10)[tl]}
\put(70,-9){\oval(160,10)[b]}
\put(145,-9){\oval(10,10)[tr]}
\end{picture}
\begin{matrix}
1\to 2\to 3\to 4\to 5\to 6\to 7\to 8\\
4\to 3\to 2\to 5\to 8\to 7\to 6\to 1
\end{matrix}
\vspace{8pt}
\end{equation}

The    corresponding    translation   surface   is   represented   in
Figure~\ref{fig:merging:zeroes}  in two different ways: as a cylinder
(rather a parallelogram) with pairs of corresponding sides identified
by  parallel translations and as a ribbon graph (separatrix diagram).
The  core  of  the  corresponding  ribbon  graph has four vertices of
valence  four  representing  four  conical  singularities  of  angles
$4\pi$, or, equivalently, four simple zeroes of the resulting Abelian
differential.  Each  edge of the ribbon graph represents a horizontal
saddle   connection   (separatrix).   Turning   around  zeroes  in  a
counterclockwise  direction,  see Figure~\ref{fig:merging:zeroes}, we
see  the  incoming  horizontal  separatrix  rays appear in the cyclic
orders  given  by  the  cyclic  decomposition of $c_1\cdot c_2^{-1}$,
namely
$$
c_1\cdot c_2^{-1}=(1,3)(2,4)(5,7)(6,8)
$$

\begin{figure}[htb]
\includegraphics{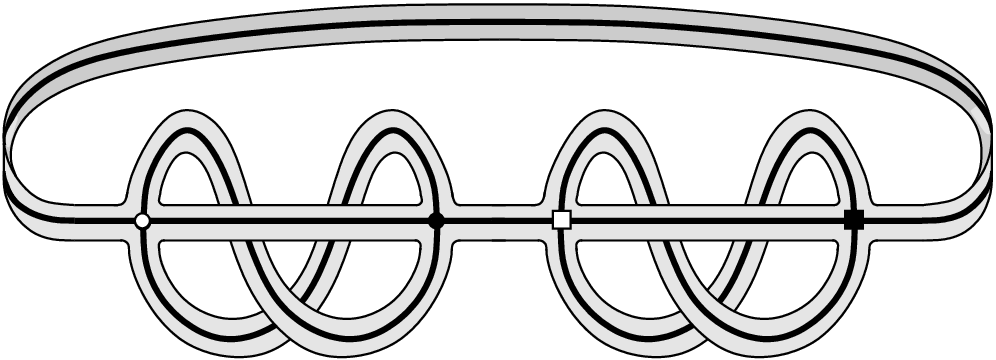}
\includegraphics{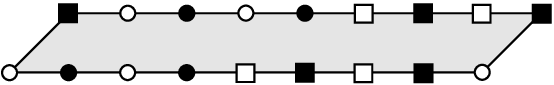}
\begin{picture}(0,0)(35,-89)
\put(-88,-138){$X_1$}
\put(-59,-112){$X_2$}
\put(-32,-138){$X_3$}
\put(-1,-112){$X_4$}
\put(30,-138){$X_5$}
\put(61,-112){$X_6$}
\put(90,-138){$X_7$}
\put(120,-112){$X_8$}
\put(146,-138){$X_1$}
\end{picture}

\begin{picture}(0,0)(3,-13)
\begin{picture}(0,0)(0,5)
\put(-88,-112){1}
\put(-60,-112){2}
\put(-32,-112){3}
\put(-4,-112){4}
\put(24,-112){5}
\put(53,-112){6}
\put(82,-112){7}
\put(110,-112){8}
\end{picture}
\begin{picture}(0,0)(28,5)
\put(-89,-155){4}
\put(-60,-155){3}
\put(-32,-155){2}
\put(-4,-155){5}
\put(24,-155){8}
\put(53,-155){7}
\put(82,-155){6}
\put(110,-155){1}
\end{picture}
\begin{picture}(0,0)(3,5)
\put(-122,-131){0}
\put(115,-136){0}
\end{picture}
\end{picture}

% \begin{picture}(0,0)(-3,-26)
% \begin{picture}(0,0)(-27,75)
% \put(-88,-112){1}
% \put(-60,-112){2}
% \put(-32,-112){4}
% \put(-4,-112){6}
% \put(24,-112){7}
% \put(53,-112){8}
% \end{picture}
% \begin{picture}(0,0)(3,75)
% \put(-89,-155){4}
% \put(-61,-155){2}
% \put(-32,-155){8}
% \put(-4,-155){7}
% \put(24,-155){6}
% \put(53,-155){1}
% \end{picture}
% \begin{picture}(0,0)(-27,75)
% \put(-126,-130){0}
% \put(51,-137){0}
% \end{picture}
% \end{picture}
\vspace{150bp}
\caption{
\label{fig:merging:zeroes}
The  ribbon  graph  representation of a Jenkins--Strebel differential
with   a   single   cylinder   (top   picture)  versus  the  cylinder
representation  (bottom  picture).  All vertices marked with the same
symbols are identified to a single conical singularity. }
\end{figure}
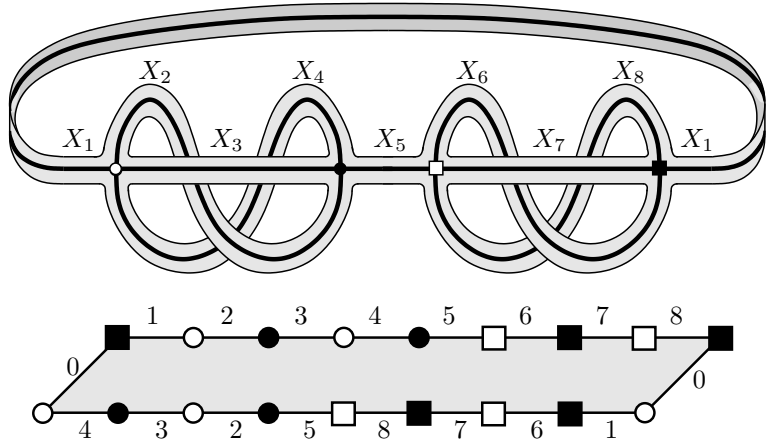
\end{Example}

It  is  clear  that  a simultaneous conjugation of permutations $c_1,
c_2,  c_3$  by  the same permutation does not change the $1$-cylinder
diagram.  In  particular, we can choose $c_1=\sigma$. Note also, that
our  diagrams  do  not  have any distinguished (marked) intervals. We
have $|C(\sigma)|=(n-1)!$ for cardinality of $C(\sigma)$, and we have
$n$  ways  to  attribute  index  $1$  to  one of the intervals at the
bottom.    Thus,    we    have    proved    the    following    Lemma
from~\cite{Delecroix}:

\begin{Lemma}
The  weighted  number $\cN_1(m_1,\dots,m_\noz)$ of $1$-cylinder diagrams
$\cD$  for  a given stratum $\cH(m_1,\dots,m_\noz)$, where the weight is
the  inverse of the order of the group of symmetries, is expressed as
\begin{equation}
\label{eq:N:m1:mr}
\cN_1(m_1,\dots,m_\noz)
=\sum_{\substack{\text{One-cylinder}\\ \text{diagrams }\cD\\
\text{in the stratum}\\ \cH(m_1,\dots,m_\noz)}}
\frac{1}{|\Gamma(\cD)|}
=
\cfrac{1}{n!}\cdot
\cN(\mathfrak{S}_n;C(\sigma),C(\sigma),C(\nu))
\end{equation}
\end{Lemma}

Now we are ready to prove Theorem~\ref{th:contribution:all:1:cyl:estimate}.

\begin{proof}[Proof of Theorem~\ref{th:contribution:all:1:cyl:estimate}]
Following~\cite{Zagier:bounds} denote by $R(\psi)$ the number of
ways to represent an even permutation $\psi$ in $\mathfrak{S}_n$
as a product of two $n$-cycles. Clearly,
\begin{equation}
\label{eq:N:through:R}
\cN(\mathfrak{S}_n;C(\sigma),C(\sigma),C(\psi)=R(\pi)\cdot|C(\psi)|\,.
\end{equation}
From now on choose any $\psi\in C(\nu)$, where $C(\nu)$ is
the conjugacy class of the
product of $r$ cycles of lengths $m_1+1,\dots,m_\noz+1$
respectively. The cardinality of $C(\psi)$ is given by
\begin{equation}
\label{eq:card:C:m}
|C(\psi)|=|C(\nu)|=n!\cdot\prod_k \frac{1}{\mult_k! (k+1)^{\mu_k}}\,,
\end{equation}
where  $\mult_k$  is  the  multiplicity of the entry $k=1,2,\dots$ in
$(m_1,\dots,m_\noz)$.

Denote  by  $c(m_1,\dots,m_\noz)$  the  absolute
contribution of all $1$-cylinder
diagrams    to    the   volume   $\Vol\cH_1(m_1,\dots,m_\noz)$   as   in
equation~\eqref{eq:contribution:all:1:cyl:estimate}              from
Theorem~\ref{th:contribution:all:1:cyl:estimate}.     Recall     that
$d=\dim\cH(m_1,\dots,m_\noz)=n+1$.

Nesting~\eqref{eq:card:C:m}                 in~\eqref{eq:N:through:R}
in~\eqref{eq:N:m1:mr}      and      combining     it     with     the
formula~\eqref{eq:contribution:numbered}                         from
Proposition~\ref{pr:contribution:Abelian}  for the contribution of an
individual    $1$-cylinder    diagram    to   the   volume   we   get
\begin{multline*}
c(m_1,\dots,m_\noz)=
\frac{1}{n!}
\cdot
\left(n!\cdot\prod_k \frac{1}{\mult_k! (k+1)^{\mu_k}}\right)
\cdot R(\psi)
\cdot
\cfrac{\mult_1!\cdot \mult_2! \cdots}{(n-1)!}\,\cdot 2\zeta(n+1)
=\\=
\frac{R(\psi)}{(n-1)!}\cdot
\frac{2\zeta(n+1)}{(m_1+1)\cdot\dots\cdot(m_\noz+1)}\,.
\end{multline*}
By Theorem~2 in~\cite{Zagier:bounds} the following universal
bounds are valid:
$$
\frac{2(n-1)!}{n+2}\le R(\psi) \le
\frac{2(n-1)!}{n+\frac{19}{29}}\,.
$$
Plugging these bounds in the latter expression for $c(m_1,\dots,m_\noz)$
in  terms  of  $R(\psi)$  and returning to notation $d=n+1$ we obtain the
bounds~\eqref{eq:contribution:all:1:cyl:estimate}                from
Theorem~\ref{th:contribution:all:1:cyl:estimate}.
\end{proof}
\medskip

%--------------------------------------------------------------------
\noindent\textbf{Frobenius formula.}
We      now      apply      Frobenius      formula      to      prove
Theorem~\ref{th:contribution:all:1:cyl}    and   then   we   evaluate
explicitly  the  contribution  of  all  $1$-cylinder  diagrams to the
volume of the ambient stratum for the minimal stratum $\cH(2g-2)$ and
for   the   principal   stratum   $\cH(1,\dots,1)$,  and  thus  prove
Corollary~\ref{cor:total:contribution:min:and:principal}.  Note  that
for   $g>3$   the   stratum   $\cH(2g-2)$  contains  three  connected
components.  Contribution  of all $1$-cylinder diagrams to individual
components is described in Proposition~\ref{pr:proportion:hyp} and in
the Conditional Corollary~\ref{cor:even:odd:1:cyl}.

\begin{proof}[Proof of Theorem~\ref{th:contribution:all:1:cyl}]
Applying Frobenius formula in the notation of~(A.8) in~\cite{Zagier}, we
express the quantity~\eqref{eq:N:C:C:C}  as a sum over characters $\chi$
of the symmetric group $\mathfrak{S}_n$:
\begin{multline}
\label{eq:Frobenius:Formula:1}
\cN(\mathfrak{S}_n;C(\sigma),C(\sigma),C(\nu))
=\\=
\cfrac{|C(\sigma)|\cdot|C(\sigma)|\cdot|C(\nu)|}{|\mathfrak{S}_n|}
\,
\sum_\chi \frac{\chi(C(\sigma))\chi(C(\sigma))\chi(C(\nu))}{\chi(1)^{3-2}}\,.
\end{multline}
In  our particular case the cardinality of the conjugacy class of the
long cycle $\sigma$ is $|C(\sigma)|=(n-1)!$ and $|\mathfrak{S}_n|=n!$.

Following  the  notation  of   \S A.2   in~\cite{Zagier},   denote   by
$\mathbf{St}_n=\C^n/\C$  the  standard  irreducible representation of
dimension $n-1$ of the group $\mathfrak{S}_n$ and put
$$
\chi_j(g):=\operatorname{tr}(g,\pi_j)\qquad
\pi_j:=\wedge^j(\mathbf{St}_n)\qquad
(0\le j\le n-1)\;,
$$
where $g\in \mathfrak{S}_n$ is any permutation.
It  is  known  that  the  representations $\pi_j$ are irreducible and
pairwise distinct for $0\le j\le n-1$ (Lemma~A.2.1 in~\cite{Zagier}).
Moreover, by Lemma~A.2.2 in~\cite{Zagier} for any irreducible
representation $\pi$ one has
$$
\chi_\pi(\sigma)=
\begin{cases}
(-1)^j,&\text{if }\pi\simeq\pi_r\text{ for some }j,\ 0\le j\le n-1\\
0&\text{otherwise\,,}
\end{cases}
$$
where  $\sigma=(1,2,\dots,n)$  is the maximal cycle in $\mathfrak{S}_n$.

Finally, $\chi_j(1)=\dim\pi_j=\binom{n-1}{j}$.

Substituting   all   these   values  in  the  Frobenius  formula  we  can
rewrite~\eqref{eq:Frobenius:Formula:1} as
\begin{multline}
\label{eq:Frobenius:Formula:2}
\cN(\mathfrak{S}_n;C(\sigma),C(\sigma),C(\nu))
=
\cfrac{(n-1)!\cdot(n-1)!\cdot|C(\nu)|}{n!}
\,\cdot\\ \cdot
\sum_{j=0}^{n-1}
(-1)^j\cdot(-1)^j\cdot\chi_j(C(\nu))\cdot\frac{j!(n-1-j)!}{(n-1)!}
=\\=
\cfrac{|C(\nu)|}{n}
\,\cdot
\sum_{j=0}^{n-1} j!\,(n-1-j)! \cdot \chi_j(C(\nu))
\end{multline}
Plugging the expression~\eqref{eq:Frobenius:Formula:2}
into~\eqref{eq:N:m1:mr} and applying~\eqref{eq:contribution:numbered}
we complete the proof of Theorem~\ref{th:contribution:all:1:cyl}.
\end{proof}

The latter formula becomes particularly simple in the case of the
minimal stratum  $\cH(2g-2)$  when  $C(\nu)=C(\sigma)$ and in the
case of the principal  stratum  $\cH(1,\dots,1)$  when the cyclic
decomposition of $\nu$ is composed of $2g-2$ cycles of length $2$.
\medskip

\begin{proof}[Proof of
Corollary~\ref{cor:total:contribution:min:and:principal}
for the minimal stratum $\cH(2g-2)$.]

In the case of the minimal stratum we get
\begin{equation}
\label{eq:Frobenius:Formula:long:cycle}
\cN(\mathfrak{S}_n;C(\sigma),C(\sigma),C(\sigma))
=
\cfrac{(n-1)!}{n}
\,\cdot
\sum_{j=0}^{n-1}(-1)^j j!\,(n-1-j)!\,.
\end{equation}
Using the combinatorial identity
$$
\sum_{k=0}^m \frac{(-1)^k}{\binom{x}{k}}=
\frac{x+1}{x+2}\left(1+\frac{(-1)^m}{\binom{x+1}{m+1}}\right)
$$
(see~(2.1) in~\cite{Gould}) we can simplify~\eqref{eq:Frobenius:Formula:long:cycle}
as
\begin{equation}
\label{eq:Frobenius:Formula:long:cycle:answer}
\cN(\mathfrak{S}_n;C(\sigma),C(\sigma),C(\sigma))
=\begin{cases}
2\cdot\cfrac{\big((n-1)!\big)^2}{n+1}&\text{for odd }n\\
0&\text{for even }n
\end{cases}
\end{equation}
Plugging the expression~\eqref{eq:Frobenius:Formula:long:cycle:answer}
into~\eqref{eq:N:m1:mr} and applying~\eqref{eq:contribution:numbered}
we complete the proof of formula~\eqref{eq:contribution:minimal}.
\end{proof}

The Lemma below will be used in the proof of
Corollary~\ref{cor:total:contribution:min:and:principal}.

\begin{Lemma}
The following identity is valid
\begin{equation}
\label{eq:combinatorial:identity}
\sum_{k=0}^m (-1)^k \left(\frac{\binom{m}{k}}{\binom{2m+1}{2k}}-\frac{\binom{m}{k}}{\binom{2m+1}{2k+1}}\right)
=
\begin{cases}
0\,,&\text{ when $m$ is even}\\
2\cdot\frac{m+1}{m+2}\,,&\text{ when $m$ is odd\,.}
\end{cases}
\end{equation}
\end{Lemma}
\begin{proof}
We use the  following combinatorial identities (see~(4.22) and~(4.23):
in~\cite{Gould})
\begin{align*}
S(m)&:=\ \sum_{k=0}^m (-1)^k \frac{\binom{m}{k}}{\binom{2m}{2k}}\ \quad =\
\frac{1+(-1)^m}{2}\cdot\frac{2m+1}{m+1}\\
T(m)&:=\ \sum_{k=0}^m (-1)^k \frac{\binom{m}{k}}{\binom{2m+1}{2k+1}}\ =\
\frac{1-(-1)^m}{2}\cdot\frac{1}{m+2} +\ (-1)^m\,.
\end{align*}
The second term in the sum~\eqref{eq:combinatorial:identity} is exactly $T(m)$,
while the first one
can be expressed in terms of $S(m)$ and $T(m)$ as follows:
\begin{multline*}
\sum_{k=0}^m (-1)^k \frac{\binom{m}{k}}{\binom{2m+1}{2k}}
=
\sum_{k=0}^m (-1)^k \frac{\binom{m}{k}}{\binom{2m}{2k}}\cdot\frac{2m+1-2k}{2m+1}
=\\=
\sum_{k=0}^m (-1)^k \frac{\binom{m}{k}}{\binom{2m}{2k}}\cdot
\left(\frac{2m+2}{2m+1}-\frac{2k+1}{2m+1}\right)
=\\=
\frac{2m+2}{2m+1}\cdot \sum_{k=0}^m (-1)^k \frac{\binom{m}{k}}{\binom{2m}{2k}}
\ -\
\sum_{k=0}^m (-1)^k \frac{\binom{m}{k}}{\binom{2m+1}{2k+1}}
=\\=
\frac{2m+2}{2m+1}\cdot S(m) - T(m)\,.
\end{multline*}
Plugging the values of $S(m)$ and of $T(m)$ into
the above expression we complete
the proof of the combinatorial identity~\eqref{eq:combinatorial:identity}.
\end{proof}

\begin{proof}[Proof of
Corollary~\ref{cor:total:contribution:min:and:principal}
for the principal stratum $\cH(1,\dots,1)$]
In the case of the principal stratum
we have $C(\nu)=C(\tau)$, where
$$
\tau=(1,2)(3,4)\dots (n-1,n)\qquad\text{and}\qquad
n=4g-4\,
$$
(see equation~\eqref{eq:n} for the formula for $n$). One has
$$
\chi_j(\tau)=(-1)^{[(j+1)/2]} \binom{n/2-1}{[j/2]}
$$
(see  the formula below~(A.26) in~\cite{Zagier}). Finally, it is easy
to    see   directly   that   $|C(\tau)|=(n-1)!!$.   Thus,   we   can
rewrite~\eqref{eq:Frobenius:Formula:2} in this particular case as
\begin{multline*}
%\label{eq:Frobenius:Formula:transpositions}
\cN(\mathfrak{S}_n;C(\sigma),C(\sigma),C(\tau))
=\\=
\cfrac{(n-1)!!}{n}
\,\cdot
\sum_{j=0}^{n-1} j!\,(n-1-j)! \cdot
(-1)^{\left[\frac{j+1}{2}\right]}
\begin{pmatrix}\frac{n}{2}-1\\\left[\frac{j}{2}\right]\end{pmatrix}
=\\=
\cfrac{(n-1)!!}{n}
\,\cdot
(n-1)! \sum_{j=0}^{n-1}
(-1)^{\left[\frac{j+1}{2}\right]}\,\cdot\
\frac{\begin{pmatrix}\frac{n}{2}-1\\\left[\frac{j}{2}\right]\end{pmatrix}}{\binom{n-1}{j}}\,.
\end{multline*}
Denoting $m=\frac{n}{2}-1$, we rewrite the above sum as
$$
\sum_{j=0}^{n-1}
(-1)^{\left[\frac{j+1}{2}\right]}\,\cdot\
\frac{\begin{pmatrix}\frac{n}{2}-1\\\left[\frac{j}{2}\right]\end{pmatrix}}{\binom{n-1}{j}}
=
\sum_{k=0}^m (-1)^k \left(\frac{\binom{m}{k}}{\binom{2m+1}{2k}}-\frac{\binom{m}{k}}{\binom{2m+1}{2k+1}}\right)
$$
Recall that $n=4g-4$, so $m=2g-3$ is odd.
Applying formula~\eqref{eq:combinatorial:identity} we obtain
$$
\cN(\mathfrak{S}_n;C(\sigma),C(\sigma),C(\tau))
=
\cfrac{(n-1)!!}{n}\cdot(n-1)!
\cdot\left(2\cdot\frac{m+1}{m+2}\right)
\,.
$$
Thus,  the weighted number $\cN(1,\dots,1)$ of $1$-cylinder diagrams
(see~\eqref{eq:N:m1:mr})
for    the    principal   stratum
$\cH(1,\dots,1)$ in genus $g$, when $n=4g-4$ equals
\begin{multline*}
\cN(1,\dots,1)=
\frac{1}{n!}\cdot\cN(\mathfrak{S}_n;C(\sigma),C(\sigma),C(\tau))
=\\=
\frac{1}{(4g-4)!}\cdot
\frac{(4g-5)!!}{(4g-4)}\cdot(4g-5)!\left(2\cdot\frac{2g-2}{2g-1}\right)
=\\=
\frac{(4g-5)!!}{(4g-4)(2g-1)}=
\frac{(4g-5)!}{(2g-1)!}\cdot 2^{-(2g-2)}\,.
\end{multline*}
Applying~\eqref{eq:contribution:numbered}
we complete the proof of formula~\eqref{eq:contribution:principal}.
\end{proof}

We complete this section with the proof of
Proposition~\ref{pr:proportion:hyp}.

\begin{proof}[Proof of Proposition~\ref{pr:proportion:hyp}]
The  results  in~\cite{AEZ:genus:0}  provide the  exact values for the
hyperelliptic  connected  components  (and,  more  generally, for all
hyperelliptic loci), namely:

\begin{align}
\label{eq:vol:hyp}
&\Vol\cH^{hyp}_1(2g-2)&=\cfrac{2\pi^{2g}}{(2g+1)!}\cdot
\cfrac{(2g-3)!!}{(2g-2)!!} \sim
\cfrac{1}{\pi^2 g}\left(\frac{\pi e}{2g+1}\right)^{2g+1}\,.
\\
&\Vol\cH^{hyp}_1(g-1,g-1)&=\cfrac{4\pi^{2g}}{(2g+2)!}\cdot
\cfrac{(2g-2)!!}{(2g-1)!!} \sim
\cfrac{1}{\pi^2 g}\left(\frac{\pi e}{2g+2}\right)^{2g+2}\,.
\end{align}

There is a single $1$-cylinder separatrix diagram for any
hyperelliptic connected component $\cH^{hyp}(2g-2)$ or
$\cH^{hyp}(g-1,g-1)$. Proposition~\ref{pr:contribution:Abelian}
provides the contribution of this diagram to the volume. Taking the
ratio of the resulting expressions~\eqref{eq:contribution:numbered}
and~\eqref{eq:vol:hyp} we obtain the expressions claimed in
Proposition~\ref{pr:proportion:hyp}.
\end{proof}

%####################################################################
%####################################################################
%####################################################################
\section{Alternative counting of $1$-cylinder separatrix diagrams}
\label{s:Alternative:counting}

In this section we suggest two alternative methods of
counting $1$-cylinder separatrix diagrams. The first one, elaborated in
section~\ref{ss:recursive:relations}, is based
on recursive relations  for the numbers of such
diagrams. The second method, presented in
section~\ref{ss:Rauzy:classes}, uses Rauzy diagrams and admits simple
computer realization for low-dimensional strata.

%--------------------------------------------------------------------
\subsection{Approach based on recursive relations}
\label{ss:recursive:relations}

Here we explicitly enumerate $1$-cylinder separatrix diagrams that give
rise  to  Abelian  differentials  (orientable  case)  or to quadratic
differentials  (nonorientable case) with 0, 1 or 2 saddle connections
shared between the two boundary components of the cylinder.
\medskip

\noindent\textbf{Strata of Abelian differentials.}
We start with the case of orientable separatrix diagrams;
they represent strata of Abelian differentials.
Take a cylinder whose boundary
components are two identical copies of an $\nofint$-gon with a marked side.
Choose an orientation of the cylinder and consider the induced orientation
on its boundary components.
Consider a gluing that identifies the sides of one boundary polygon
with the sides of the other reversing their orientation and
respecting the marked sides. We get a closed orientable surface with
a connected graph $\G$ (the image of the cylinder boundary
components) embedded into it. All vertices of $\G$ have even degree,
and we denote by $v_i$ the number of vertices of $\G$ of degree $2i$.
Clearly, $\nofint=\sum_{i\geq 1} iv_i$, and we call
$[1^{v_1}2^{v_2}\dots]$ the \textit{type} of the cylinder gluing.
The associated $1$-cylinder separatrix diagram corresponds
to the stratum $\cH(0^{v_1},1^{v_2},2^{v_3},\dots)$, and the
complex dimension of this stratum is $n+1$.

Let us now fix a partition $\nu=[1^{v_1}2^{v_2}\dots]$ of $\nofint$ and
denote by $N_\nofint(\nu)$ the number of cylinder gluings of type $\nu$
described above. Consider the generating functions
\begin{align}
&F_\nofint(t_1,t_2,\dots)=\sum_{\nu\,\vdash\, \nofint} N_{\nofint}(\nu)\,t_1^{v_1}\,t_2^{v_2}\dots\;,\nonumber\\
&F(s;t_1,t_2,\dots)=\sum_{\nofint\geq 1}s^{\nofint-1}\,F_\nofint(t_1,t_2,\dots)\;.\nonumber
\end{align}

\begin{thm}\label{abel}
Put
\begin{align}
M_1=\sum_{i=2}^\infty \sum_{j=1}^{i-1} (i-1)t_j t_{i-j}\,\frac{\partial}{\partial t_{i-1}} + j(i-j) t_{i+1}\,\frac{\partial^2}{\partial t_j \partial t_{i-j}}\;.
\end{align}
Then the generating function $F=F(s;t_1,t_2,\dots)$ satisfies the linear PDE
\begin{align}\label{pde}
\frac{\partial F}{\partial s}=M_1F
\end{align}
and is uniquely determined by the initial condition $F|_{s=0}=t_1$.
Equivalently, the generating function $F$ is explicitly given by the formula
\begin{align}
F(s;t_1,t_2,\dots)=e^{sM_1}t_1\;.
\end{align}
\end{thm}

\begin{proof}
First, rewrite~\eqref{pde} as a recursion for the numbers
$N_{\nofint}(\nu)$. Denote by $\e_i$ the sequence with 1 at the $i$-th
place and 0 elsewhere. Then~\eqref{pde} is equivalent to
\begin{align}
(\nofint-1)&N_{\nofint}(\nu)=\nonumber\\
&=\sum_{i=2}^\infty \sum_{j=1}^{i-1} (i-1)(v_{i-1}+1-\delta_{j,1}-\delta_{i-j,1})\,N_{\nofint}(\nu-\e_j-\e_{i-j}+\e_{i-1})+\nonumber\\
&+\sum_{i=2}^\infty \sum_{j=1}^{i-1} j(i-j)(v_j+1)(v_{i-j}+1+\delta_{j,i-j})\,N_{\nofint}(\nu+\e_j+\e_{i-j}-\e_{i-1})\label{rec}\;.
\end{align}
We prove it by establishing a direct bijection between cylinder
gluings counted in the left and right hand sides of (\ref{rec}).
Consider the ribbon graph $\G^*$ dual to $\G$. It has 2 vertices
(each of degree $\nofint$) and $\nofint$ edges connecting these two vertices (one
of these edges is marked). Let us pick a non-marked edge in $\G^*$,
this can be done in $(\nofint-1)$ ways giving the l.h.s. in (\ref{rec}).
Deletion of this edge results in one of the following two
possibilities:
\begin{enumerate}[label=\roman*)] %[(i)]
\item The edge belongs to two different boundary cycles of $\G^*$
    of lengths $2j$ and $2(i-j)$. The edge deletion gives rise to
    one boundary cycle of length $2(i-1)$ and the graph type
    changes to $\nu-\e_j-\e_{i-j}+\e_{i-1}$.
\item One boundary cycle of length $2(i+1)$ traverses the edge
    twice (once in each direction). After the edge deletion the
    boundary cycle splits into two ones of lengths $2i$ and
    $2(i-j)$ and the graph type changes to
    $\nu+\e_j+\e_{i-j}-\e_{i+1}$.
\end{enumerate}
Counting the number of ways that each case can occur we get the first
and the second sums in~\eqref{rec} respectively.

To show that the generating function $F$ is uniquely determined by
the initial condition $F|_{s=0}=t_1$, we first notice that $F_1=t_1$
(for $\nofint=1$ there is only one $1$-cylinder configuration). The equation
(\ref{pde}) recursively expresses $F_\nofint$ in terms of $F_{\nofint-1}$ as
follows:
\begin{align}
(\nofint-1)\,F_\nofint=M_1F_{\nofint-1}\;.\label{ind}
\end{align}
Explicit formula $F=e^{sM_1}t_1$ is just another way of writing the same thing.
\end{proof}

\begin{rem}
The numbers $N_\nofint(\nu)$ giving the {\em rooted} count of $1$-cylinder
configurations and the numbers
$\cN(0^{v_1},1^{v_2},2^{v_3},\dots)$,
see~\eqref{eq:N:m1:mr}, giving the {\em
weighted} count of $1$-cylinder diagrams
in $\cH(0^{v_1},1^{v_2},2^{v_3},\dots)$ with weights $1/|{\rm Aut}(\G)|$ are related by the
simple formula
\begin{equation}
\label{eq:nonrooted:from:rooted}
\cN(0^{v_1},1^{v_2},2^{v_3},\dots)=
\frac{1}{\nofint}\cdot N_\nofint(\nu)\,.
\end{equation}

\end{rem}

\begin{Example}

Consider the generating functions for small values of $\nofint$:
\begin{align*}
F_1&=t_1\\
F_2&=t_1^2\\
F_3&=t_1^3+t_3\\
F_4&=t_1^4+4 t_1 t_3 + t_2^2\\
F_5&=t_1^5+10t_3 t_1^2 + 5t_1 t_2^2 + 8t_5\\
F_6&=t_1^6+20 t_1^3 t_3 + 15 t_1^2 t_2^2 + 48 t_1 t_5 + 24 t_2 t_4 + 12 t_3^2
\end{align*}
We know that there is a single $1$-cylinder diagram in the stratum
$\cH(2)$ which has symmetry of order $3$, see Figure~\ref{fig:diag}
in section~\ref{ss:separatrix:diagrams}. For this stratum we have
$\nu=[3^1]$ so we can read the weighted number of $1$-cylinder
diagrams from the coefficient in front of $t_3$ in $F_3$
normalizing it as in~\eqref{eq:nonrooted:from:rooted}. This gives
$1/3$ as expected.

Consider now the stratum $\cH(3,1)$.
The number of associated rooted diagrams is given
by  the  coefficient  of  the monomial $24 t_2 t_4$ in the polynomial
$F_6$. Combining~\eqref{eq:nonrooted:from:rooted}
and~\eqref{eq:contribution:numbered} we get the following impact
of all $1$-cylinder square-tiled surfaces to the volume of this stratum:
$$
\frac{24}{6}\cdot \frac{2\cdot 1!\cdot 1!}{5!}\cdot\zeta(7)=\frac{1}{15}\cdot \zeta(7)\,.
$$
By~\cite{Eskin:Masur:Zorich} we have
$$
\Vol\cH_1(3,1)=\frac{16}{42525}\pi^6=\frac{16}{45}\zeta(6)\,.
$$
Thus, the relative impact $\prop_1(\cH(3,1))$
of $1$-cylinder diagrams is equal to
$$
\left(\frac{1}{15}\zeta(7)\right):\left(\frac{16}{45}\zeta(6)\right)=
\frac{3\zeta(7)}{16\zeta(6)}\,.
$$
which matches the value given in Example~\ref{ex:H31}.

\end{Example}
\medskip

\noindent\textbf{Strata of quadratic differentials.}
Now we proceed with with the case of nonorientable separatrix diagrams;
they represent strata of meromorphic quadratic
differentials with at most simple poles. Take a cylinder bounded by
two polygons, one with $l+2m$ sides and the other with $l+2n$ sides
and consider its orientable gluings that identify $m$ pairs of sides
of the first polygon, $n$ pairs of sides of the second polygon, and
$l$ sides of the first one with $l$ sides of the second one.

We warn the reader that we have two polygons
with a priori different number of sides, and that from now on
the symbol $n$ does not denote the total number of sides anymore.
Contrary to the previous section we
do not mark any side on either of the two polygons anymore.

We get a closed orientable surface, and the image of the boundary
polygons is a graph $\G$ (not necessarily connected) embedded into
it. Suppose that $\G$ has the vertex degree set $v_1,v_2,\ldots$,
where $\nu=[1^{v_1}\,2^{v_2}\,\ldots]$ is a partition of $2(l+m+n)$
(this means that $\G$ has $v_1$ vertices of degree 1, $v_2$ vertices
of degree 2, etc.). The associated $1$-cylinder separatrix diagram
corresponds to the stratum $\cQ(-1^{v_1},0^{v_2},1^{v_3},\dots)$, and
the complex dimension of this stratum is $l+m+n$.

Denote by $N_{l,m,n}(v_1,v_2,\ldots)$ the weighted
count of such gluings. It coincides with the number
$\cN_{l,m,n}(-1^{v_1},0^{v_2},1^{v_3},\dots)$
giving the \textit{weighted} count of $1$-cylinder diagrams
of type $(l,m,m)$ in $\cQ(-1^{v_1},0^{v_2},1^{v_3},\dots)$
with weights $1/|{\rm Aut}(\G)|$
up to a correction in the symmetric
case when $m=n$:
\begin{equation}
\label{eq:number:of:l:m:n:diagrams}
\cN_{l,m,n}(-1^{v_1},0^{v_2},1^{v_3},\dots)
=
\begin{cases}
N_{l,m,n}(v_1,v_2,\ldots)&\text{when $m\neq n$}\\
\frac{1}{2}\cdot N_{l,m,n}(v_1,v_2,\ldots)&\text{when $m= n$\,.}
\end{cases}
\end{equation}

Consider the generating series
\begin{align}
F_{l,m,n}=\sum_{\nu\vdash 2(l+m+n)}N_{l,m,n}(v_1,v_2,\ldots)p_1^{v_1}p_2^{v_2}\ldots\;.
\end{align}

To explicitly compute $F_{l,m,n}$ with $l=0,1,2$ we introduce an
auxiliary generating series $G(s,p_1,p_2,\ldots)$. The coefficient
of $G$ at the monomial $s^{2\nofint}p_1^{v_1}p_2^{v_2}\ldots$ is the
number of orientable gluings of a $2\nofint$-gon with fixed vertex degree
set given by the partition $[1^{v_1}\,2^{v_2}\,\ldots]$ of $2\nofint$. In
other words, each gluing produces a closed orientable surface of
genus $g=\frac{1}{2}\left(1+\nofint-\sum_i v_i\right)$ together with a
graph embedded into it with $v_1$ vertices of degree 1, $v_2$
vertices of degree 2, etc. As usual, the gluings are counted with
weights reciprocal to the orders of the automorphism groups.

The generating series $G(s,p_1,p_2,\ldots)$ was extensively studied
in ~\cite{Kazarian:Zograf}. In particular, as it follows from Theorem
3 (ii) in~\cite{Kazarian:Zograf}, the series $G$ is uniquely
determined by the equation
\begin{align}
\frac{1}{s}\frac{\partial G}{\partial s}=M_2G+p_1^2
\end{align}
modulo the initial condition $G|_{s=0}=0$,
where
\begin{align}
M_2=\sum_{i=2}^\infty \sum_{j=1}^{i-1}(i-2)p_j p_{i-j}\,\frac{\partial}{\partial p_{i-2}} + j(i-j) p_{i+2}\,\frac{\partial^2}{\partial p_j \partial p_{i-j}}\;.
\end{align}

It will be convenient to write $G$ as a power series in $s$:
\begin{align}
G(s,p_1,p_2,\ldots)=\sum_{\nofint=1}^\infty s^{2\nofint} G_\nofint(p_1,p_2,\ldots)\;.
\end{align}
Then we have

\begin{thm}\label{quad}
The following formulas hold:
\begin{align}
F_{0,m,n}&=G_m G_n\;,\label{0}\\
F_{1,m,n}&=\sum_{i=1}^\infty \sum_{j=1}^{\infty}ij\, p_{i+j+2}\,\frac{\partial G_m}{\partial p_i}\frac{\partial G_n}{\partial p_{j}}\;,\label{1}\\
F_{2,m,n}&=\frac{1}{2}\sum_{i=1}^\infty\sum_{j=1}^\infty\sum_{k=1}^\infty\sum_{l=1}^\infty ijk\ell\,p_{i+k+2}\,p_{j+\ell+2}
\frac{\partial^2 G_m}{\partial p_i \partial p_j}\frac{\partial^2 G_n}{\partial p_k \partial p_{\ell}}\label{2}\\
&+\sum_{i=1}^\infty\sum_{j=1}^\infty\sum_{k=1}^\infty ijk(k+1)\,p_{i+j+k+4}\left(\frac{\partial^2 G_m}{\partial p_i \partial p_j}\frac{\partial G_n}{\partial p_k}+\frac{\partial G_m}{\partial p_k}\frac{\partial^2 G_n}{\partial p_j \partial p_k}\right)\nonumber\\
&+\sum_{i=1}^\infty\sum_{j=1}^\infty ij \left(\sum_{k=0}^i\sum_{\ell=0}^j p_{k+\ell+2}\,p_{i+j+2-k-\ell}\right)\frac{\partial G_m}{\partial p_i}\frac{\partial G_n}{\partial p_j}\;.\nonumber
\end{align}
\end{thm}

\begin{proof}
Instead of the graph $\G$ (the image of cylinder's boundary) it is
handier to consider its dual graph $\G^*$. The graph $\G^*$ has two
vertices, $m$ loops incident to the first vertex, $n$ loops incident
to the second vertex and $l$ edges connecting the first vertex with
the second one. We also assume that the vertices are labeled.

Formula (\ref{0}) of Theorem \ref{quad} is obvious.

To prove (\ref{1}), let us take two ribbon graphs with one vertex
each, the first one with $m$ loops and the second one with $n$ loops.
Let us count the number of ways to connect the two vertices with a
single edge. For any boundary component of length $i$ of the first
graph and any boundary component of length $j$ of the second graph
there are $ij$ possibilities to connect them with an edge. Instead of
two disjoint boundary components of lengths $i$ and $j$ we get a
single boundary component of length $i+j+2$. This simple observation
is precisely described by Formula (\ref{1}).

The proof of Formula (\ref{2}) is similar to that of (\ref{1}).
Again, we start with two ribbon graphs with one vertex each, the
first one with $m$ loops and the second one with $n$ loops. Now we
count the number of different ways to connect the two vertices with a
double edge. Four possibilities can occur:
\begin{enumerate}[label=\roman*)]
\item Two different boundary components of the first graph of lengths
$i$ and $j$ are connected by two edges with two boundary components
of the second graph of lengths $k$ and $\ell$ respectively. There are
$ijk\ell$ ways to do that. The boundary components of lengths $i$ and
$k$ are replaced by a single boundary component of length $i+k+2$,
and the components of lengths $j$ and $\ell$ are replaced by a single
component of length $j+\ell+2$. This possibility is described by the
first line in the right hand side of (\ref{2}).
\item Two different boundary components of the first graph of
lengths $i$ and $j$ are connected by two edges with
one boundary components of the second graph of lengths $k$. This can
be done in $ijk(k+1)$ ways. The three boundary components of
lengths $i,\;j$ and $k$ are replaced by a single
boundary component of length $i+j+k+4$.
\item A boundary component of the first graph of length $k$ is
connected by two edges with two boundary components of the second
graph of lengths $i$ and $j$. Similar to the previous case, this can
be done in $ijk(k+1)$ ways. The three boundary components of
lengths $i,\;j$ and $k$ are replaced by a single
boundary component of length $i+j+k+4$. The cases (ii) and (iii) can
be united to produce the second line in the right hand side of
(\ref{2}).
\item A boundary component of the first graph of length $i$ is
connected by two edges with a boundary component of the second graph
of length $j$. There are $ij$ ways to connect the two boundary
components with one edge. If the endpoints of the second edge at the
distances $k$ and $\ell$ from the endpoints of the first one, the
components of lengths $i$ and $j$ get replaced by the boundary
components of lengths $k+\ell+2$ and
$i+j+2-k-\ell$. This last possibility is described by the third line
in the right hand side of (\ref{2}).
\end{enumerate}
\end{proof}

\begin{Example}
\label{ex:Qi313}
To find the contribution of $1$-cylinder separatrix diagrams to the
volume of the stratum $\cQ(1^3,-1^3)$ we have to find the weighted
number of ribbon graphs as above with $3$ vertices of valence $1$
(corresponding to $3$ simple poles) and with $3$ vertices of valence
$3$ (corresponding to $3$ simple zeroes). So the \textit{type of the
cylinder gluing} representing the stratum $\cQ(1^3,-1^3)$ is $[1^3,
3^3]$ and we are interested in monomials corresponding to $p_1^3
p_3^3$ in polynomials $F_{l,m,n}$ with $l+m+n=6$. We present some of
them to compare the result with the diagram-by-diagram calculation
presented in the next section.
\begin{align}
\label{eq:F015}
F_{0,1,5}&=4p_1^3 p_4 p_2 p_3 + p_1^5 p_2^2 p_3 + 3p_1^3 p_5 p_2^2 + \frac{1}{2} p_1^6 p_2 p_4 + 5p_1^4 p_6 p_2 + \frac{7}{2} p_1^4 p_5 p_3 + \frac{1}{10} p_1^7 p_5\\
\notag
&+\frac{5}{2} p_1^5 p_7 + \frac{21}{2} p_9 p_1^3 + \frac{21}{4} p_8 p_1^2 p_2 + \frac{7}{2} p_1^2 p_7 p_3 + \frac{13}{4} p_1^2 p_4 p_6 + \frac{33}{20} p_1^2 p_5^2 + \frac{1}{4} p_1^4 p_2^4\\
\notag
&+\frac{1}{4} p_1^6 p_3^2 + \frac{1}{2} p_1^3 p_3^3 + \frac{1}{2} p_1^2 p_4 p_2^3 + \frac{1}{2} p_1^2 p_2^2 p_3^2 + \frac{3}{2} p_1^4 p_4^2\,.\\
\label{eq:F033}
F_{0,3,3}&=\frac{1}{3}p_4 p_1^3 p_2 p_3 + p_4 p_1 p_2 p_5 + \frac{1}{36} p_3^4 + \frac{1}{3} p_1^5 p_2^2 p_3 + \frac{1}{6} p_1^2 p_2^2 p_3^2\\
\notag
&+\frac{1}{6} p_3^2 p_4 p_2 + \frac{1}{3}  p_1 p_5 p_3^2 + \frac{1}{4} p_1^4 p_2^4 + \frac{1}{9} p_1^6 p_3^2 + \frac{1}{9} p_1^3 p_3^3\\
\notag
&+\frac{1}{4} p_2^2 p_4^2 + p_5^2 p_1^2 + p_5 p_1^3 p_2^2 + \frac{1}{2} p_4 p_1^2 p_2^3 + \frac{2}{3} p_5 p_1^4 p_3\,.\\
\label{eq:F213}
F_{2,1,3}&=10p_4p_5p_2p_1 + 16 p_8 p_1^2 p_2 + 4p_7 p_2^2 p_1 + 13p_4 p_6 p_1^2 + 7 p_5^2 p_1^2 + 12 p_9 p_1^3 + 5 p_{10} p_2\\
\notag
&+36 p_{11} p_1 + \frac{1}{2} p_4^2 p_2^2 + 5 p_1^3 p_2 p_3 p_4 + 5 p_1 p_2 p_3 p_6 + p_3^3 p_1^3 + 3 p_1 p_5 p_3^2 + 4 p_1 p_3 p_4^2\\
\notag
&+2 p_3 p_9 + \frac{13}{2} p_4 p_8 + 5 p_5 p_7 + \frac{3}{2} p_6^2 + p_1^4 p_4^2 + p_1^2 p_2^3 p_4 + p_1^2 p_2^2 p_3^2 + 2 p_1^3 p_2^2 p_5\\
\notag
& + p_1^4 p_2 p_6 + 2 p_1^4 p_3 p_5 + 13 p_1^2 p_3 p_7\,.
\end{align}
By~\eqref{eq:F015} the term $p_1^3 p_3^3$ in $F_{0,1,5}$ has
coefficient $\frac{1}{2}$, so the weighted number $\sum_\cD
\frac{1}{\Gamma(\cD)}$ of $1$-cylinder diagrams representing the
stratum $\cQ(1^3,-1^3)$ with $l=0,m=1,n=5$ is equal to $\frac{1}{2}$.
Table~\ref{tab:Qi3:13}  in section~\ref{ss:Rauzy:classes} shows that such diagram
is, actually, unique, and that its symmetry group $\Gamma(\cD)$
indeed has order $2$.

By~\eqref{eq:F033} the term $p_1^3 p_3^3$ in $F_{0,3,3}$ has
coefficient $\frac{1}{9}$, so the weighted number $\sum_\cD
\frac{1}{\Gamma(\cD)}$ of $1$-cylinder diagrams representing the
stratum $\cQ(1^3,-1^3)$ with $l=0,m=3,n=3$ is equal to $\frac{1}{18}$
(recall that when $m=n$ we have to divide the corresponding coefficient
by $2$ to get the weighted number of diagrams; see~\eqref{eq:number:of:l:m:n:diagrams}).
Table~\ref{tab:Qi3:13} in section~\ref{ss:Rauzy:classes} shows that there is a
unique such diagram, and that its symmetry group $\Gamma(\cD)$ has
order $18$.

By~\eqref{eq:F213} the term $p_1^3 p_3^3$ in $F_{2,1,3}$ has
coefficient $1$.
Table~\ref{tab:Qi3:13} in section~\ref{ss:Rauzy:classes} shows that there is a
unique $1$-cylinder diagram with $l=2,m=1,n=3$ in the stratum $\cQ(1^3,-1^3)$,
and that this diagram does not have any symmetries.

\end{Example}

%--------------------------------------------------------------------
\subsection{Approach based on Rauzy classes}
\label{ss:Rauzy:classes}

In  this  section we consider a
complete  list  of  $1$-cylinder  separatrix  diagrams representing
one particular     stratum     of    meromorphic    quadratic    differentials
$\cQ(d_1,\dots,d_k)$ with at most simple poles and we consider their
contributions to the Masur--Veech volume.

Denote   by   $\mult_{-1},\mult_1,\mult_2,\dots$  the  multiplicities
$\mult_{j}$   of  entries  $j\in\{-1,1,2,\dots\}$  in  the  set
$\{d_1,\dots,d_k\}$,  where  $\sum  d_i  =4g-4$,  and $g\in\Z_+$.
In the notation of section~\ref{ss:recursive:relations} we have
$\mult_{i}=v_{i+2}$.
In  combinatorial  terms,  we  want  to  construct  all
possible oriented (which is stronger than \textit{orientable})
ribbon  graphs  with exactly $\mu_{-1}$ vertices of valence
$1$;  with  exactly  $\mu_1$  vertices  of  valence  $3$,  ...,  with
exactly  $\mu_j$  vertices  of  valence $j+2$, etc. We are interested
only   in  those  ribbon  graphs  which  have  exactly  two  boundary
components,  and  which  satisfy  the  following extra condition: for
each of the two boundary components $\partial\cD$ of the ribbon graph
$\cD$ there exists an edge of $\cD$ such that it has $\partial\cD$ on
both sides of it.

To  give  an  idea  of  an  approximate  calculation  of  the  volume
based  on our method we compute $\Vol\cQ_1(1^3,-1^3)$ (the stratum is
chosen  by  random).  We  present  a  list of all ribbon graphs $\cD$
satisfying the above conditions, which are realizable in
$\cQ(1^3,-1^3)$.   For   each   such   ribbon   graph   we  give  the
order   $|\Gamma|=|\Gamma(\cD)|$  of  its  symmetry  group,  we
present            $l,m,n$           and           we           apply
formula~\eqref{eq:general:contribution}  to  compute its contribution
to  the  volume  of  the  stratum.  Recall  the  convention used
in~\eqref{eq:general:contribution}:   defining   the  symmetry  group
$\Gamma(\cD)$  we  assume  that  none  of  the  vertices,  edges,  or
boundary  components  of the ribbon graph $\cD$ is labeled; however,
we assume  that  the orientation of the ribbons is fixed.

The  stratum  $\cQ(1^3,-1^3)$  corresponds  to  genus  $g=1$.  It  is
connected      and      $d=\dim_\C\cQ(1^3,-1^3)=6$.      We      have
$\mult_{-1}=3$,   $\mult_1=3$,   and   there  are  no  other  entries
$\mult_k$.  This  means that every such ribbon graph has $3$ vertices
of valence one, and $3$ vertices of valence $3$.

\begin{table}[htb]
$$
\begin{array}{|c|c|c|c|}
\hline &&& \\ [-\halfbls]
%\hspace*{5pt}
\text{Ribbon graph }\cD\hspace*{5pt} &|\Gamma|& l,m,n &\text{Contribution to }\Vol\cQ_1(1^3,-1^3)\\
&&& \\ [-\halfbls]
\hline&&&\\
[-\halfbls]
&& l=0 &\\
\includegraphics{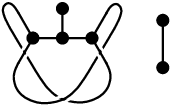}
& 2 & m=5 &
\cfrac{2^{0+2}}{2}\cdot\cfrac{(5+1-2)!}{(5-1)!(1-1)!}\cdot\cfrac{3!\cdot 3!}{(6-2)!}\,\zeta(6)=3\zeta(6)\\
&& n=1 &\\
   %--------------------------------------------------------------
\hline&&&\\
[-\halfbls]
&& l=0 &\\
\includegraphics{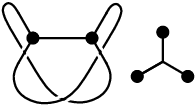}
& 18 & m=3 &
\cfrac{2^{0+2}}{18}\cdot\cfrac{(3+3-2)!}{(3-1)!(3-1)!}\cdot\cfrac{3!\cdot 3!}{(6-2)!}\,\zeta(6)=2\zeta(6)\\
&& n=3 &\\
   %--------------------------------------------------------------
\hline&&&\\
[-\halfbls]
&& l=2 &\\
\includegraphics{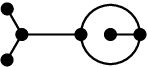}
& 1 & m=3 &
\cfrac{2^{2+2}}{1}\cdot\cfrac{(3+1-2)!}{(3-1)!\cdot(1-1)!}\cdot\cfrac{3!\cdot 3!}{(6-2)!}\,\zeta(6)=24\zeta(6)\\
&& n=1 &\\
   %--------------------------------------------------------------
\hline&&&\\
[-\halfbls]
&& l=3 &\\
\includegraphics{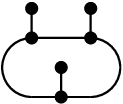}
& 1 & m=2 &
\cfrac{2^{3+2}}{1}\cdot\cfrac{(2+1-2)!}{(2-1)!(1-1)!}\cdot\cfrac{3!\cdot 3!}{(6-2)!}\,\zeta(6)=48\zeta(6)\\
&& n=1 &\\
\hline
\end{array}
$$
\caption{
\label{tab:Qi3:13}
Contribution of $1$-cylinder pillowcase covers to the Masur--Vech volume
$\Vol\cQ_1(1^3,-1^3)$
}
\end{table}

Table~\ref{tab:Qi3:13}  above  shows  that the total contribution of $1$-cylinder
separatrix   diagrams   to   the   volume   $\Vol\cQ_1(1^3,-1^3)$  is
$77\zeta(6)$.   The   statistics  of  frequencies of  $1:2:3$-cylinder
pillowcase  covers  in $\Vol\cQ_1(1^3,-1^3)$ collected experimentally
(see    the    table    for    strata    of    dimension    $6$    in
Appendix~\ref{s:tables:volumes})           gives           proportions
$0.4366:0.4000:0.1634$ which results in
$$
\Vol\cQ_1(1^3,-1^3)\approx \frac{77\zeta(6)}{0.4366}\approx 0.1866 \pi^{6}\,.
$$
as  an approximate value of the volume. The exact value of the volume
found by E.~Goujard in~\cite{Goujard:volumes} gives
$$
\Vol\cQ_1(1^3,-1^3)=\frac{11}{60}\cdot\pi^6\approx 0.1837 \pi^{6}\,,
$$
see  the  line  corresponding to $\cQ(1^3,-1^3)$ in the table for the
strata of dimension $6$ in Appendix~\ref{s:tables:volumes}.
The types of separatrix diagrams and orders
of their symmetry groups presented in the table above
matches the calculation by means of recursive relation
considered in Example~\ref{ex:Qi313}.

Direct  calculations  of  this  kind  were  performed  for  a limited
number  of  strata,  mostly  to  debug the more efficient alternative
approaches    discussed    in   section~\ref{ss:recursive:relations}.
It is, basically, impossible not to forget some ribbon graphs,  to
identify  all  isomorphic ones, and to compute correctly the
cardinality $|\Gamma(\cD)|$ of the symmetry group of each graph
performing ad hoc calculations for strata represented by more then
ten diagrams. Thus, the lists of the diagrams and the cardinalities
of their symmetry groups were, actually, found by computer and
verified in some simple cases by hands.

We used \textit{Rauzy diagrams} to generate these data. Rauzy
diagrams are strongly connected oriented graphs whose vertices are
\textit{generalized permutations}. A generalized permutation is an
ordered pair of ordered sets (traditionally
represented by two lines) composed of entries $0, 1, \dots, n$, where
each entry is presented in the above data exactly twice, and the
unordered union of elements of none of the two lines is a strict
subset of the unordered union of elements of the complementary line.
An usual permutation of the set $\{0, 1, \dots, n\}$ is a particular
case of a generalized one.

There is a bijection between Rauzy diagrams of generalized
permutations and connected component of strata,
see~\cite{Boissy:Lanneau} and~\cite{Veech:Gauss:measures}. Moreover,
any $1$-cylinder diagram in the corresponding component is
represented by a certain subcollection of generalized permutations
whose top first and bottom last symbols are identical; such
(generalized) permutations are called \textit{standard} permutations
in the context of Rauzy diagrams.

Figure~\ref{fig:Jenkins:Strebel} at the beginning of
section~\ref{ss:contribution:of:one:1:cylinder:diagram:computation} illustrates
how the standard generalized permutation
$$
\begin{pmatrix}&0&1&1&\\ &2&3&2&3&0\end{pmatrix}
$$
represents a nonorientable $1$-cylinder separatrix diagram.

The bottom picture in Figure~\ref{fig:merging:zeroes} from
section~\ref{ss:1:cylinder:diagrams:Abelian} illustrates how the
standard permutation
$$
\begin{pmatrix}
0& 1& 2& 3& 4& 5& 6& 7& 8\\
4& 3& 2& 5& 8& 7& 6& 1& 0
\end{pmatrix}
$$
represents the orientable $1$-cylinder diagram on top of
Figure~\ref{fig:merging:zeroes}.

The advantage of this approach is that it is very easy to generate
all permutations in a Rauzy diagram associated to any low-dimensional
stratum. Given a stratum of meromorphic quadratic differentials with
at most simple poles, say, $\cQ(1^3, -1^3)$, we first use the
method~\cite{Zorich:representatives} of one of the authors to
construct some generalized permutation representing the desired
(connected component of) the stratum. Next, one just has to apply two
simple transformation rules to generate the whole Rauzy diagram from
any element. Using the \texttt{surface\_dynamics} package of the
software SageMath it is a five line program:
\begin{verbatim}
sage: from surface_dynamics.all import *
sage: Q = QuadraticStratum({1:3, -1:3})
sage: p = Q.permutation_representative()
sage: R = p.rauzy_diagram(right_induction=True, left_induction=True)
sage: R
Rauzy diagram with 2010 permutations
sage: std_perms = [q for q in R if q[0][0] == q[1][-1]]
sage: len(std_perms)
158
\end{verbatim}

Note that the same $1$-cylinder separatrix diagram might be (and
usually is) represented by several standard generalized permutations.
For example, the following four standard generalized
permutations represent the same $1$-cylinder separatrix diagram:
\begin{equation}
\label{eq:all:standard:permutations}
\begin{pmatrix}
0\,1\,2\,3\,1\,2\,3\\
4\,4\,5\,5\,6\,6\,0
\end{pmatrix}
\quad
\begin{pmatrix}
0\,1\,2\,3\,1\,2\,3\\
4\,5\,5\,6\,6\,4\,0
\end{pmatrix}
\quad
\begin{pmatrix}
0\,1\,1\,2\,2\,3\,3\\
4\,5\,6\,4\,5\,6\,0
\end{pmatrix}
\quad
\begin{pmatrix}
0\,1\,2\,2\,3\,3\,1\\
4\,5\,6\,4\,5\,6\,0
\end{pmatrix}
\end{equation}
from the second line of Table~\ref{tab:Qi3:13}; the one for which we
have $|\Gamma| = 18$ and $l = 0$, $m = 3$, $n = 3$. We explain now
how we group the resulting standard permutations into subcollections
associated to separatrix diagrams, and how we compute the order
$|\Gamma(\cD)|$ of the symmetry group $\Gamma(\cD)$ of a separatrix
diagram $\cD$.

We can put standard permutations into the one-to-one correspondence
with $1$-cylinder separatrix diagrams endowing the latter with the
following extra structure. Choose one of the two possible choices of
a top and a bottom boundary component of the cylinder, and choose a
\textit{marking} on each of the components. (The marking corresponds
to a choice of distinguished saddle connection on each of the
boundary components.)

All standard generalized permutations representing any given
separatrix diagram $\cD$, can be obtained from any standard
generalized permutations representing $\cD$ by the following two
operations.

Remove distinguished symbols (denoted by ``$0$'' in the examples
above); rotate cyclically the top line by any rotation; rotate
cyclically the bottom line by any rotation; insert the distinguished
element on the left of the upper line and on the right of the bottom
one; renumber the entries. We get a collection $D_1$ of standard
generalized permutations.

Apply to every standard generalized permutations in $D_1$ the
following operation. Remove distinguished symbols (denoted by ``$0$''
in the examples above); interchange the top and the bottom line;
insert the distinguished element on the left of the upper line and on
the right of the bottom one and renumber the entries. We get one more
collection $D_2$ of standard generalized permutations.

Take the union of $D_1$ and $D_2$. It is easy to see that we have
constructed all standard generalized permutations representing the
initial separatrix diagram $\cD$. We suggest to the reader to check
that the collection~\eqref{eq:all:standard:permutations} can be
constructed by the two operations as above from any of its elements.

Since the top boundary component is composed from $l+2m$ separatrices
and the bottom component --- from $l+2n$ ones, the cardinality of the
set of nontrivial operations as above is $2 \times (l+2m) \times
(l+2n)$. Thus, the order $|\Gamma(\cD)|$ of the symmetry group
$\Gamma(\cD)$ of the associated separatrix diagram $\cD$ is
$$
\card\Gamma(\cD):=\Big(2 \times (l+2m) \times (l+2n)\Big)/
\card(D_1\cup D_2)\,.
$$
In example~\eqref{eq:all:standard:permutations} we get
$$
\card\Gamma(\cD)=\Big(2 \times (0+2\cdot 3) \times (0+2\cdot 3)\Big)/4=18
$$
as indicated in the second line in Table~\ref{tab:Qi3:13}.

\appendix

%######################################################################
%######################################################################
%######################################################################
\section{An overview of the Masur--Veech volumes of strata}
\label{s:overview:of:MV:volumes}

To   make  the  presentation  self-contained  we  reproduce  in  this
section   the   necessary   background  material  from  the  original
papers~\cite{Masur:82},                  \cite{Veech:Gauss:measures},
\cite{Eskin:Okounkov:Inventiones},        \cite{Zorich:square:tiled}.

%--------------------------------------------------------------------
\subsection{Masur--Veech volume element  in  the  strata}

A stratum $\cH(m_1,\dots,m_\noz)$ of Abelian differentials is locally
modelled  on  the  cohomology $H^1(S,\{P_1,\dots,P_\noz\};\C)$ of the
underlying   topological   surface   $S$   relative   to   collection
$\{P_1,\dots,P_\noz\}$  of  zeroes. The structure of the vector space
in  the  cohomology  gives  rise  to a one-parameter family of linear
measures  defined  up to a scalar factor. The canonical choice of the
scalar factor is imposed by the condition that the fundamental domain
of   the   lattice  $H^1(S,\{P_1,\dots,P_\noz\};\Z\oplus  i\Z)\subset
H^1(S,\{P_1,\dots,P_\noz\};\C)$   has   unit   volume.   Denote   the
corresponding  volume  element  (density  of  the measure) by $d\nu$.

Consider  an  Abelian  differential  $\omega$  on  a  Riemann surface
$S$;  let  $A_i,B_i$  be  its  periods.  The  area $S(\omega)$ of the
underlying    surface    $S$   measured   in   the   flat   structure
determined   by $\omega$ equals
$$
\text{Flat area of }S(\omega) =
            \cfrac{i}{2}\int_S\omega\wedge\bar{\omega} =
\cfrac{i}{2}\sum_{i=1}^g (A_i\bar{B}_i - \bar{A}_i B_i)\,.
$$
Thus,   the   area  of  the  translation  surface  is  a  homogeneous
real-valued function on the  moduli space of Abelian differentials:
$$
\Area: \mathcal{H}(m_1,\dots,m_\noz) \to \R \qquad
\Area S(\lambda\cdot\omega)=|\lambda|^2 \Area S(\omega), \quad \lambda\in\C{}\,.
$$
Consider  a  ``unit  sphere'', or, better say, a ``unit hyperboloid''
$$
\mathcal{H}_1(m_1,\dots,m_\noz)                         \subset
\mathcal{H}(m_1,\dots,m_\noz)
\quad\text{defined as a level hypersurface}\quad\Area S(\omega)=1\,.
$$
The volume element      $d\nu$  on the stratum induces  the
volume        element        $d\nu_1:=\cfrac{d\nu}{d\Area}$        on
$\mathcal{H}_1(m_1,\dots,m_\noz)$.

\begin{NoNumberTheorem}[H.~Masur~\cite{Masur:82};
W.~Veech~\cite{Veech:Gauss:measures}]
The    volume    of    any    stratum    of   Abelian   differentials
$\mathcal{H}_1(m_1,\dots,m_\noz)$  with  respect  to  the volume element
$d\nu_1$ is finite.
\end{NoNumberTheorem}

The  situation  with  the moduli spaces of quadratic differentials is
similar.  Consider  a  Riemann surface $S$ endowed with a meromorphic
quadratic   differential   $q$   with   at  most  simple  poles;  let
$\cQ(d_1,\dots,d_k)$  be  the  ambient  stratum for $(S,q)$. Any such
pair  $(S,q)$ defines a canonical orienting double cover $p:\hat S\to
S$ such that $p^\ast q=\hat\omega^2$ is already a global square of an
Abelian  differential $\hat\omega$ on the double cover $\hat S$. This
double  cover  $\hat  S$  is  endowed  with  the canonical involution
$\iota$ interchanging the two preimages of every regular point of the
cover.  The  stratum $\cQ(d_1,\dots,d_k)$ is modelled on the subspace
of  the  relative  cohomology of the orienting double cover $\hat S$,
antiinvariant   with   respect   to   the  involution  $\iota$.  This
antiinvariant   subspace   is   denoted   by   $H^1_-(\hat   S,\{\hat
P_1,\dots,\hat    P_\noz\};\C)$,    where    $\{\hat   P_1,\dots,\hat
P_\noz\}$   are   zeroes   of   the   induced   Abelian  differential
$\hat\omega$.

Recall our Convention~\ref{conv:lattice}:
\begin{NNConvention}
%\label{conv:lattice}
We     define     a    lattice    in    $H^1_-(\hat
S,\{\hat{P}_1,\ldots,\hat{P}_\noz\};\C{})$  as  the  subset  of those
linear  forms  which  take  values  in  $\Z\oplus i\Z$ on $H^-_1(\hat
S,\{\hat   P_1,\dots,\hat   P_\noz\};\Z)$.   We   define  the  volume
element   $d\nu$   on   $\cQ(d_1,\dots,d_k)$  as  the  linear  volume
element            in            the           vector           space
$H^1_-(S,\{\hat{P}_1,\ldots,\hat{P}_\noz\};\C{})$  normalized in such a
way  that  the  fundamental  domain  of  the  above  lattice has unit
volume.
\end{NNConvention}

We  warn  the  reader  that  for  $\noz>1$  this  lattice is a proper
sublattice of index $4^{2g+\noz-1}$ of the lattice
$$
H^1_-(\hat S,\{\hat P_1, \dots, \hat P_\noz\};\C)\ \cap\
H^1(\hat S,\{\hat P_1, \dots, \hat P_\noz\};\Z\oplus i\Z)\,.
$$

The  choice  of one or another lattice is a matter of convention. Our
choice   is  coherent  with~\cite{AEZ:Dedicata},  \cite{AEZ:genus:0},
and~\cite{Goujard:volumes}.

Another    convention   concerns   the normalization  of  the  area
of  the  flat  surface  $S$. Similarly to the  case  of Abelian
differentials we choose a real hypersurface $\cQ_1(d_1,\dots,d_k)$ of
flat surfaces of fixed area in the  stratum  $\cQ(d_1,\dots,d_k)$.

\begin{Convention}
\label{con:area:1:2}
We abuse notation by denoting by
$\cQ_1(d_1,\dots,d_k)$  the  space of flat surfaces of area $1/2$ (so
that the canonical double cover has area $1$).
\end{Convention}

We address the reader to \S~4.1 in~\cite{AEZ:genus:0}
for  the  arguments in favour of these conventions.

The volume element $d\nu$ in the embodying space $\cQ(d_1,\dots,d_k)$
induces  naturally  a  volume  element  $d\nu_1$  on the hypersurface
$\cQ_1(d_1,\dots,d_k)$  in  the  following  way.  In complete analogy
with   the   case  of  Abelian  differentials,  there  is  a  natural
$\C^\ast$-action        on        $\cQ(d_1,\dots,d_k)$:        having
$\lambda\in\C^\ast$  we associate to the flat surface
$S=(C,q)$ (where $C$ is a complex curve and $q$ is a meromorphic quadratic
differential) the
flat  surface
\begin{equation}
\label{eq:Cstar:action}
\lambda\cdot S:=(C,\lambda^2\cdot q)\,.
\end{equation}
In  particular, we can represent any $S\in\cQ(d_1,\dots,d_k)$ as $S =
R\cdot S_{(1)}$,  where  $R\in\R_+$,  and  where $S_{(1)}$ belongs to the
``hyperboloid'': $S_{(1)}\in\cQ_1(d_1,\dots,d_k)$. Geometrically this
means that the metric on $S$ is obtained from the metric on $S_{(1)}$
by  rescaling  with  linear  coefficient  $R$. In particular, vectors
associated  to  saddle connections on $S_{(1)}$ are multiplied by $R$
to  give  vectors  associated  to corresponding saddle connections on
$S$.  It  means  also that $\Area(S) = R^2\cdot\Area(S_{(1)})=R^2/2$,
since  $\Area(S_{(1)})  =  1/2$.  We  define the \textit{Masur--Veech
volume      element}      $d\nu_1$     on     the     ``hyperboloid''
$\cQ_1(d_1,\dots,d_k)$   by  disintegration  of  the  volume  element
$d\nu$ on $\cQ(d_1,\dots,d_k)$:
\begin{equation}
\label{eq:disintegration}
d\nu = R^{2d-1} \, dR\, d\nu_1\, ,
\end{equation}
where
$$
d=\dim_{\C{}}\cQ(d_1,\dots,d_k)=2g+k-2\,.
$$
Using   this   volume  element  we  define  the  \textit{Masur--Veech
volume} of the stratum $\cQ_1(d_1,\dots,d_k)$:
\begin{equation}
\label{eq:int:cF:nu1}
\Vol\cQ_1(d_1,\dots,d_k):= \int_{\cQ_1(d_1,\dots,d_k)}d\nu_1\,.
\end{equation}

For       a       subset       $E\subset\cQ_1(d_1,\dots,d_k)$      we
denote by $C(E)\subset\cQ_1(d_1,\dots,d_k)$ the cone based on $E$:
\begin{equation}
\label{eq:cone}
 C(E):=\{S=R\cdot S_{(1)}\,|\, S_{(1)}\in E,\ 0<R\le 1\}\,.
\end{equation}
Our  definition  of  the  volume element on $\cQ_1(d_1,\dots,d_k)$ is
consistent with the following normalization:
\begin{equation}
\label{eq:normalization}
\Vol(\cQ_1(d_1,\dots,d_k)) =
\dim_{\R{}} \cQ(d_1,\dots,d_k)\cdot\nu(C(\cQ_1(d_1,\dots,d_k))\,,
\end{equation}
where  $\nu(C(\cQ_1(d_1,\dots,d_k))$  is  the  total  volume  of  the
``cone''  $C(\cQ_1(d_1,\dots,d_k))\subset\cQ(d_1,\dots,d_k)$ measured
by means of the volume element $d\nu$ on $\cQ(d_1,\dots,d_k)$ defined
above.

%--------------------------------------------------------------------
\subsection{Counting volume by counting integer points}

One   of   the  ways  to  evaluate  the  ``hyperarea''  of  a  smooth
hypersurface  in  the  Euclidean  space $\mathbb{E}^d$ is to make a
homothety  with  a  huge  coefficient $R$ and count the number of
integer   points   inside the spatial body bounded by the  rescaled
hypersurface.  This  number asymptotically  behaves as the volume
$\operatorname{Vol}(R) = R^d\cdot \operatorname{Vol}(1)$ of  the
body bounded by the rescaled hypersurface. The desired surface area
equals
$$
\frac{d \operatorname{Vol}(R)}{dR}\Big|_{R=1}\Big. =
d\cdot \operatorname{Vol}(1)\,.
$$
In  other  words,  to  compute  the area of the surface bounding some
spatial  body  it  is  sufficient  to  know  the  coefficient  in the
leading  term  of  the  asymptotics  of  the number of integer points
which  get inside the stretched body.

The same approach can be applied to calculation of the volumes of the
strata  of  Abelian  and  of  quadratic differentials. Now we have to
count       the       \textit{integer       points}      $\omega_0\in
\mathcal{H}(m_1,\dots,m_\noz)$, (respectively
$q_0\in\cQ(d_1,\dots,d_k)$)  such  that $\Area S(\omega_0)$
(respectively $\Area  S(q_0)$) is bounded by some huge number $N$
(respectively $N/2$),  which  plays  the  role  of the radius $R$.
The     only     difference     with    the    previous    case    is
that     $Area(R\cdot S)$     is a
homogeneous   function  of  degree  $2$  in  $R$,  so  when evaluating  the
hypersurface  area  by  derivation  of  the volume one has to multiply
the result by the extra factor $2$.

Let  us  describe the geometry of translation surfaces represented by
{\it    integer    points}.    Having    an    Abelian   differential
$[\omega_0]\in  H^1(M^2_g,\{P_1,\dots,P_\noz\};\Z\oplus  i\,\Z)$  we can
define  a  map  $f_{\omega_0} : M^2_g \to \T^2=\C{}/(\Z\oplus i\,\Z)$
by
$$
f_{\omega_0} : P\mapsto \Big(\int_{P_1}^P \omega_0\Big)\
\text{mod } \Z\oplus i\,\Z\,.
$$
It   is  easy  to  see  that  $f_{\omega_0}$  is  a  ramified  cover;
moreover,  the  ramification  points  are  exactly  the  zeroes  $P_1,
\dots,  P_\noz$  of $\omega_0$. Consider the flat torus $\T^2$ as a unit
square    with    the    identified   opposite   sides.   The   cover
$f_{\omega_0}:S\to  \T^2$  endows  the  Riemann  surface  $S$  with a
tiling  by  unit  squares.  The  tiling  represents a standard square
lattice  except  for  the  vertices  $P_1,  \dots, P_\noz$ where we have
respectively  $4(m_1+1), \dots, 4(m_\noz+1)$  squares adjacent to a
vertex.  Note  that  all  the  unit  squares  are  provided  with the
following  additional  structure:  we know exactly which edge is top,
bottom,  right,  and  left;  adjacency  of  the squares respects this
structure  in  a  natural  way.  We  shall  call  a flat surface with
such tiling a \textit{square-tiled surface}.

In  the  case  of  quadratic  differentials  the  integer  points are
represented  by \textit{pillowcase covers} over $\CP$ branched at the
four      corners      of     the     square     pillow     as     in
Figure~\ref{fig:square:pillow}.

\begin{figure}[htb]
   %
   %  PILLOWCASE
   %
\includegraphics{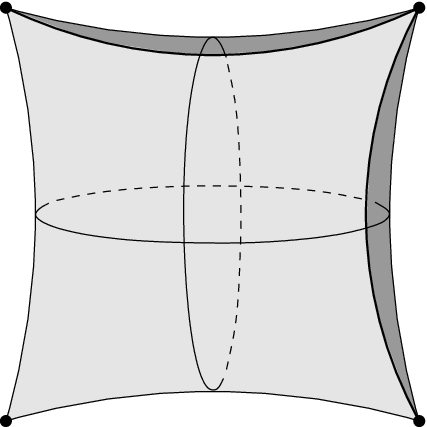}
\vspace{55pt}
\caption{
\label{fig:square:pillow}
Flat $\CP$ glued from two squares with the side $\frac{1}{2}$.}
\end{figure}

To                          summarize,                          under
conventions~\eqref{eq:disintegration}--\eqref{eq:normalization}    we
get  the  following  formulas (see~\cite{Eskin:Okounkov:Inventiones},
\cite{Zorich:square:tiled} for details):

\begin{NNLemma}
Let  $\frac{c}{2d}$  be  the  coefficient in the asymptotics
$\frac{c}{2d}\cdot N^d$ of the number     of     square-tiled
surfaces     in     the stratum $\cH(m_1,\dots,m_\noz)$  tiled  with
at  most  $N$ unit squares when $N\to+\infty$. Then
\begin{equation}
\label{eq:volume:from:square:tiled}
\Vol\cH_1(m_1,\dots,m_\noz)=c\,.
\end{equation}
Similarly, let $\frac{c}{2d}$ be the coefficient in the asymptotics
$\frac{c}{2d}\cdot N^d$ of    the    number    of
pillowcase   covers   in   the   stratum $\cQ(d_1,\dots,d_k)$ of
degree at most $N$. Then
\begin{equation}
\label{eq:volume:from:pillowcase}
\Vol\cQ_1(d_1,\dots,d_k)=c\,.
\end{equation}
\end{NNLemma}

Note  that  some  strata  are  not  connected.  In  this  case  it is
important  to  compute  the  volume  of  every  individual  connected
component.   The  connected  components  of  the  strata  of  Abelian
differentials   are   classified   in~\cite{Kontsevich:Zorich};   the
connected  components  of  the  strata of quadratic differentials are
classified  in~\cite{Lanneau}; see also a recent rectification in the
arXiv version of the latter paper.

%--------------------------------------------------------------------
\subsection{Jenkins--Strebel differentials. Separatrix diagrams}
\label{ss:separatrix:diagrams}

Assume   that   all   leaves   of  the  horizontal  foliation  of  an
Abelian  or  quadratic differential are either closed or connect
critical  points  (a  leaf  joining  two  critical points is called a
\textit{saddle   connection}  or  a  \textit{separatrix}).  Later  we
will  be  saying  simply  that  the  horizontal  foliation  has  only
closed   leaves.   The  square  of  an  Abelian  differential,  or  a
quadratic   differential   having   this   property   is   called   a
\textit{Jenkins--Strebel}           quadratic           differential,
see~\cite{Strebel}.  For example, square-tiled surfaces or pillowcase
covers   provide  particular  cases  of  Jenkins--Strebel  differentials.

Following~\cite{Kontsevich:Zorich}   we   will  associate  with  each
Abelian or quadratic differential whose horizontal foliation has only
closed   leaves   a   combinatorial  data  called  \textit{separatrix
diagram}.

We  start  with  an  informal  explanation. Consider the union of all
saddle   connections  for  the  horizontal  foliation,  and  add  all
critical points. We obtain a finite graph $\Gamma$. In the case of an
Abelian  differential  it  is  oriented, where the orientation on the
edges   comes  from  the  canonical  orientation  of  the  horizontal
foliation.   In   both   cases  of  an  Abelian  or  quadratic
differential,  the  graph  $\Gamma$  is drawn on an oriented surface,
therefore  it  carries  a \textit{ribbon structure}, i.e. on the star
of   each   vertex   $v$   a   cyclic  order  is  given,  namely  the
counterclockwise  order  in  which  half-edges are attached to $v$. In the
case  of  an Abelian differential, the direction of edges attached to
$v$  alternates  (between  directions  toward $v$ and from $v$) as we
follow  the cyclic order.

It  is  well  known  that  any  finite  ribbon graph $\Gamma$ defines
canonically  (up  to  an  isotopy)  an  oriented  surface $S(\Gamma)$
with  boundary.  To  obtain  this  surface  we  replace  each edge of
$\Gamma$  by  a  thin  oriented  strip  (rectangle)  and  glue  these
strips   together   using   the   cyclic  order  in  each  vertex  of
$\Gamma$.  In  our  case  surface  $S(\Gamma)$  can  be realized as a
tubular  $\varepsilon$-neighborhood  (in the sense of the transversal
measure)  of  the  union  of  all saddle connections for sufficiently
small $\varepsilon>0$.

In  the  case  of  an  Abelian differential, the orientation of edges
of  $\Gamma$  gives  rise  to  the  orientation  of  the  boundary of
$S(\Gamma)$.  Notice  that  this orientation is {\it not} the same as
the  canonical  orientation  of  the boundary of an oriented surface.
Thus,  connected  components  of  the  boundary  of  $S(\Gamma)$  are
decomposed  into  two  classes:  positively  and  negatively oriented
(positively   when   two  orientations  of  the  boundary  components
coincide  and  negatively,  when  they  are opposite). We shall also
refer  to  them as the \textit{top} and \textit{bottom} components of
the  corresponding cylinder, with respect to the positive orientation
of   the   vertical   foliation.   The   complement  to  the  tubular
$\varepsilon$-neighborhood  of $\Gamma$ is a finite disjoint union of
open    flat    cylinders   foliated   by   circles.   It   gives   a
decomposition      of     the     set     of     boundary     circles
$\pi_0(\partial   S(\Gamma))$   into   pairs   of  components  having
opposite orientation.

Now   we   are   ready   to   give  a  formal  definition  (see~\S  4
in~\cite{Kontsevich:Zorich} for more details on separatrix diagrams):

\begin{Definition}
A  \textit{separatrix  diagram}  is  a  finite  oriented ribbon graph
$\Gamma$,  and  a  decomposition of the set of boundary components of
$S(\Gamma)$  into  pairs.  An  \textit{orientable} separatrix diagram
satisfies the following additional properties:
\begin{enumerate}
\item  the  orientation  of the half-edges  at  any vertex
alternates with
respect to the cyclic order of edges at this vertex;
\item   there   is   one   positively  oriented  and  one  negatively
oriented boundary component in each pair.
\end{enumerate}
\end{Definition}

Any   separatrix   diagram   represents  a  measured  foliation  with
only  closed  leaves  on a compact oriented surface without boundary.
We  say  that  a  diagram  is  \textit{realizable} if, moreover, this
measured  foliation  can  be  chosen  as  the horizontal foliation of
some    Abelian    or    quadratic    differential    (depending   on
orientability of the foliation).

Assign  to  each  saddle  connection a real variable standing for its
``length''.  Now  any  boundary  component  is also naturally endowed
with  a  ``length''.  If  we  want  to  glue  flat  cylinders  to the
boundary  components,  the  lengths  of  the components in every pair
should  match  each  other.  Thus,  for  every two boundary components
paired  together  we  get  a linear
relation  on the lengths of saddle connections. Clearly, a diagram is
realizable  if  and  only  if  the  corresponding  system  of  linear
equations    on    lengths    of    saddle   connections   admits   a
strictly positive solution.

As  an example, consider all possible separatrix diagrams which might
appear      in      the      stratum      $\cH(2)$      (see~\S     5
in~\cite{Zorich:square:tiled}  for  more details). The single conical
singularity  of  a flat surface in $\cH(2)$ has cone angle $6\pi$, so
every  separatrix  diagram has a single vertex with six prongs. Since
it  corresponds to the stratum of Abelian differentials, it should be
oriented.  All  such diagrams are presented in Figure~\ref{fig:diag}.
We  see,  that the left diagram $\cD_1$ defines a translation surface
with  a  single  pair  of  boundary  components  (i.e.  with a single
cylinder  filled with closed horizontal leaves); it is realizable for
all  positive values $\ell_1,\ell_2,\ell_3$ of length parameters. The
middle   diagram  defines  a  surface  with  two  pairs  of  boundary
components  (i.e.  with  two  cylinders filled with closed horizontal
leaves);  it  is  realizable  when $\ell_1=\ell_3$. The right diagram
would  correspond  to  a  surface  with  a  single  ``top''  boundary
component,  and with three ``bottom'' boundary components. Since each
``top''  boundary component must be attached to a ``bottom'' boundary
component  by  a  cylinder,  this  diagram  is  not  realizable  by a
translation surface.

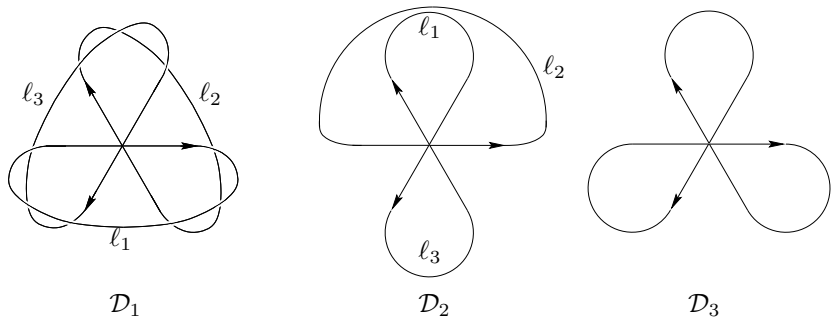
\begin{figure}[htb]
\includegraphics{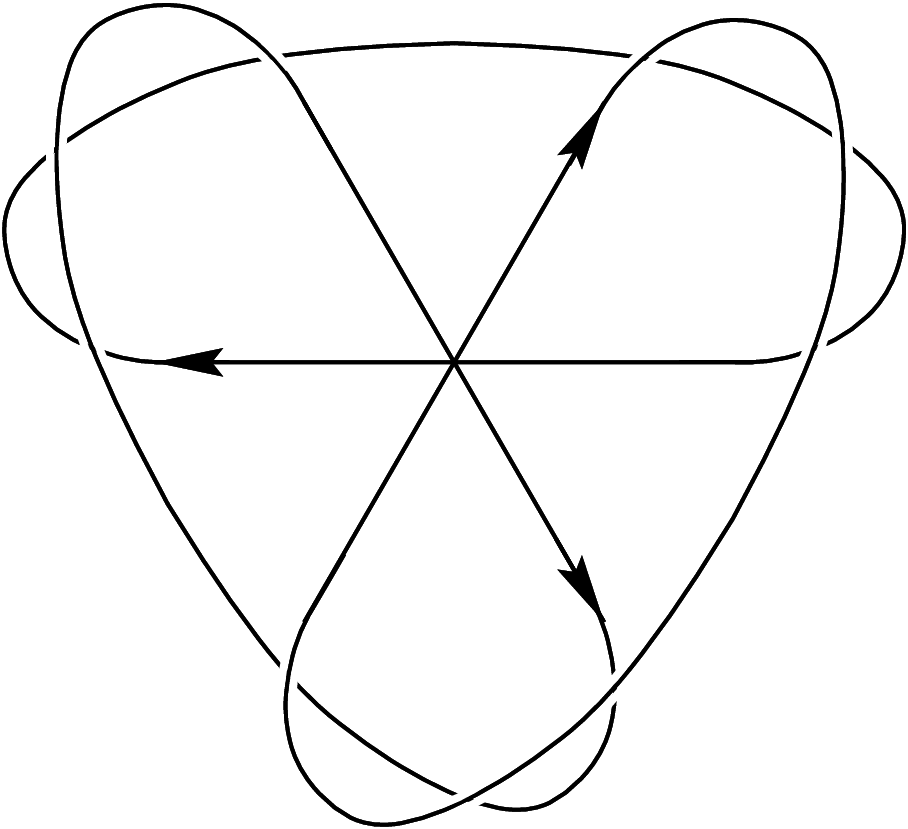}
\includegraphics{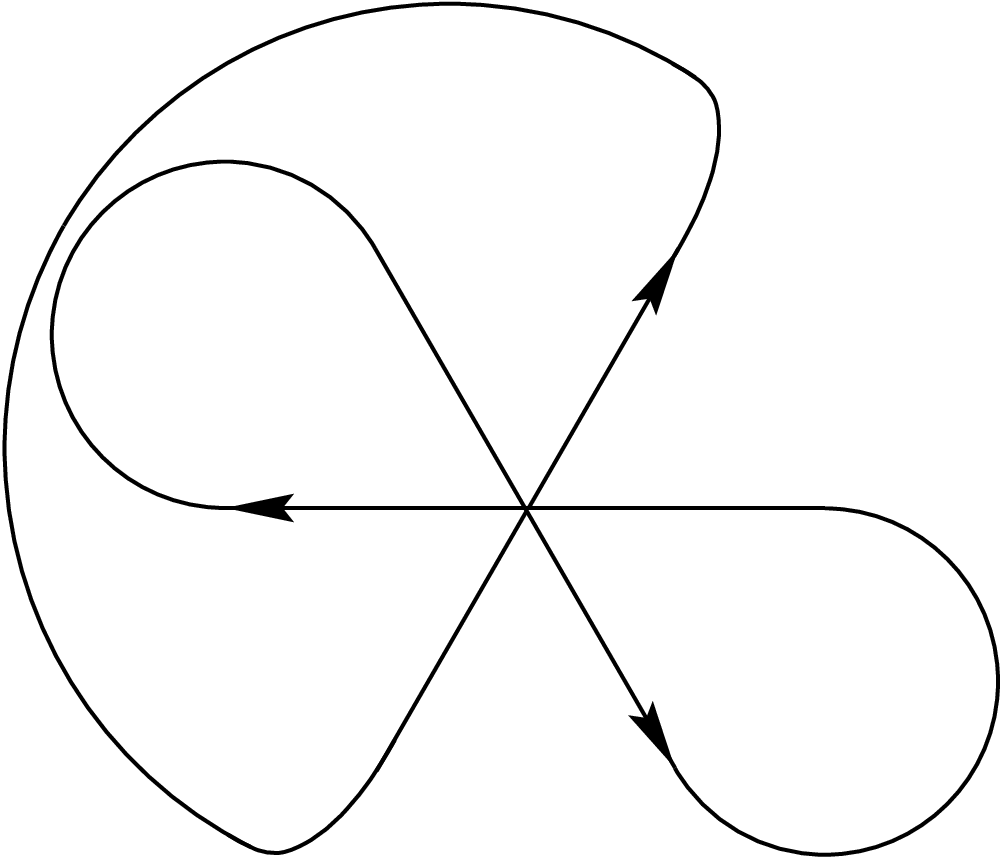}
\includegraphics{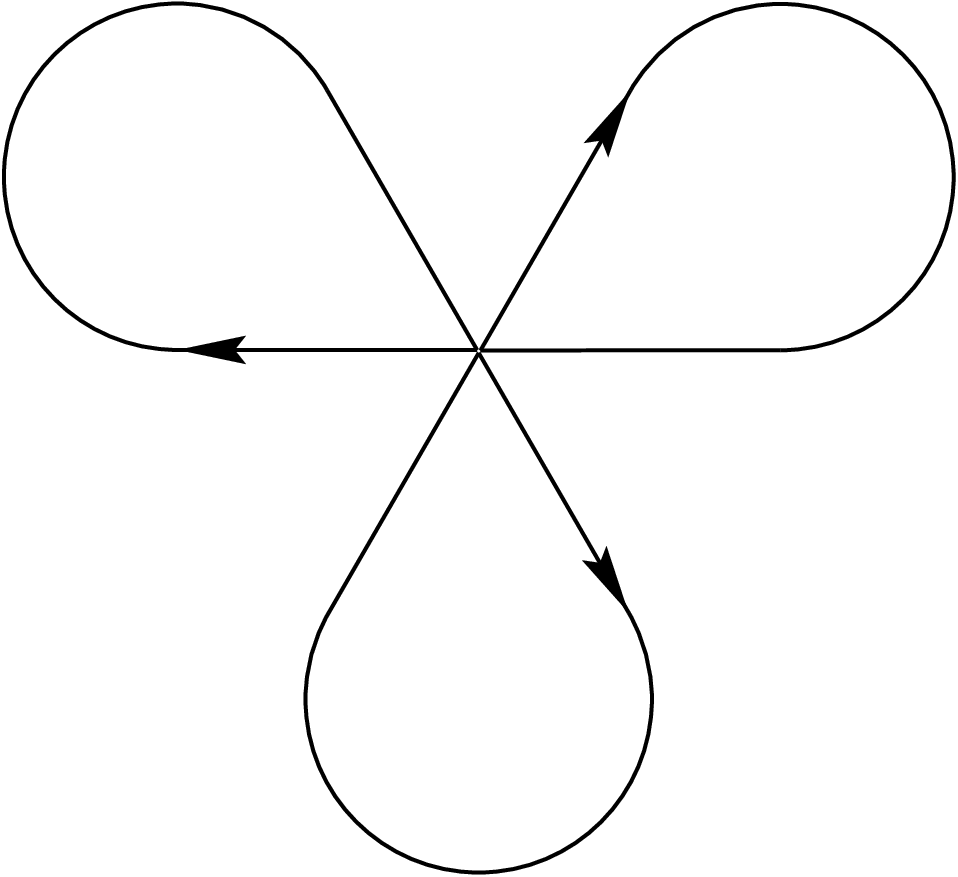}
\vspace{125bp} % 125 !!! Last Change
\begin{picture}(0,0)(170,-10)
\put(180,0){\begin{picture}(0,0)(0,0)
\put(-129,20){$\ell_1$}
\put(-95,75){$\ell_2$}
\put(-162,75){$\ell_3$}
\put(-12,100){$\ell_1$}
\put(35,85){$\ell_2$}
\put(-12,15){$\ell_3$}
\put(-129,-5){$\cD_1$}
\put(-12,-5){$\cD_2$}
\put(90,-5){$\cD_3$}
\end{picture}}
\end{picture}
\caption{
\label{fig:diag}
The   separatrix   diagrams   represent   from   left   to   right  a
square-tiled   surface   glued   from:   $\cD_1$  ---  one  cylinder;
$\cD_2$   ---   two  cylinders;  $\cD_3$  ---  not  realizable  by  a
square-tiled surface.}
\end{figure}

% %--------------------------------------------------------------------
% \subsection{Examples of straightforward computation of the volumes
% in genus 2}

We  reproduce now the original computation of M.~Kontsevich of the volume
of  the stratum $\cH(2)$; see~\cite{Zorich:square:tiled} for details.
\medskip

\noindent\textbf{Example of calculation of the Masur--Veech volume.}
Any  square-tiled  surface  in $\cH(2)$ corresponds to one of the two
separatrix  diagrams $\cD_1,\cD_2$ in Figure~\ref{fig:diag}. Consider
those    square-tiled   surfaces   from   $\mathcal{H}(2)$   which
correspond  to  the  leftmost  diagram,  that  is to $\cD_1$. In this
case  our surface is glued from a single cylinder. The waist curve of
the  cylinder  has  length  $w=\ell_1+\ell_2+\ell_3$,  where $\ell_1,
\ell_2,  \ell_3$  are  the  integer  lengths of the horizontal saddle
connections  (also called separatrix loops). Denote the height of the
cylinder  by  $h_1$.  Note  that  there is one more integer parameter
determining  our  square-tiled  surface:  the  integer  twist $\phi$
which  we  apply  to glue together the  two boundary components of the cylinder. It
has   an   integer   value   in   the   interval  $[1,w]$.  Thus  the
number  of  square-tiled  surfaces  of  this  type with area bounded by
$N$  is  asymptotically  equivalent to the sum
$$
\cfrac{1}{3}\sum_{\substack{\ell_1,\ell_2,\ell_3,h\in\N\\
                   (\ell_1+\ell_2+\ell_3)h\le N}}
(\ell_1+\ell_2+\ell_3)\,.
$$
The   coefficient   $1/3$   compensates   the  arbitrariness  of  the
choice   of  numbering  of  $\ell_1,\ell_2,\ell_3$  preserving  the
cyclic  ordering.  In  other words, the order $|\Gamma(\cD_1)|$
of  the  symmetry  group  of  the separatrix diagram $\cD_1$ is equal
to  three (the vertices of any separatrix diagram are numbered, while
the  edges are not). We can group the entries in the sum above having
the  same  length  $w$  of  the  waist  curve  of  the  cylinder. The
number   of  ordered  partitions  of  a  large  integer  $w$  into  the
sum   of  three  positive  integers  $w=\ell_1+\ell_2+\ell_3$  equals
approximately  $w^2/2$. Thus we can rewrite the sum above as follows:
\begin{multline}
\label{eq:D1}
\frac{1}{3}\sum_{\substack{\ell_1,\ell_2,\ell_3,h\\(\ell_1+\ell_2+\ell_3)h\le N}} (\ell_1+\ell_2+\ell_3)
\sim
\frac{1}{3}\sum_{\substack{w,h\\w\cdot h\le N}} w\cdot\frac{w^2}{2}
=
\frac{1}{6}\ \sum_{\substack{w,h\\w\le \frac{N}{h} }} w^3
\sim\\ \sim
\frac{1}{6}\ \sum_{h\in\N} \frac{1}{4}\cdot \left(\frac{N}{h}\right)^4
=
\frac{N^4}{24}\cdot\sum_{h\in\N} \frac{1}{h^4}
=
\frac{N^4}{24}\cdot \zeta(4)
=
\frac{N^4}{24}\cdot \frac{\pi^4}{90}\,.
\end{multline}

Consider  now  a  square-tiled  surface  corresponding  to the middle
diagram  $\cD_2$  on Figure~\ref{fig:diag}. The only admissible way
to  paste  horizontal cylinders into this diagram is drawn on
Figure~\ref{fig:colored:separatrix:diagram}  indicating
how the boundary components are coupled.

\begin{figure}[htb]
% This is an example of separatrix loops for a zero of order 2.
%
\includegraphics{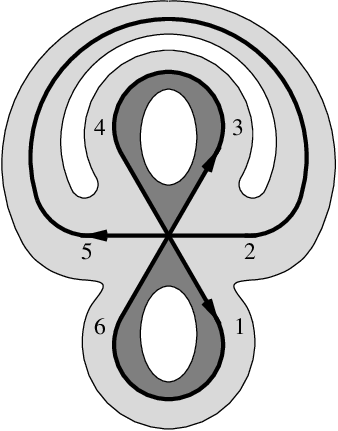}
\vspace{160bp}% 180
\caption{
\label{fig:colored:separatrix:diagram}
The colors on the ribbon graph associated to the separatrix diagram
indicate how to paste in the horizontal cylinders.
}
\end{figure}

Denote  the  integer  lengths  of  the  separatrix  loops by $\ell_1,
\ell_2,   \ell_3$   (see  Figure~\ref{fig:diag}).  The  realizability
condition  imposes  the  linear  relation  $\ell_1=\ell_3$.  The flat
surface  $S$  is  glued  from two cylinders: one having the waist
curve of length $\ell_1$,  and  the  other  one  having the waist
curve of length $\ell_1+\ell_2$. Denote the heights and twists of the
corresponding cylinders by $h_1, h_2$ and by $\phi_1,\phi_2$
respectively. The integer twist of the first cylinder takes value in
the interval $[1,\ell_1]$; the integer twist of the second cylinder
takes value  in  the  interval  $[1,\ell_1+\ell_2]$.  Thus, the
number  of surfaces  of 2-cylinder type with area bounded by $N$ is
asymptotically equivalent to the value of the sum
\begin{multline}
\label{eq:D2}
\sum_{\substack{\ell_1,\ell_2,h_1,h_2\\\ell_1 h_1+(\ell_1+\ell_2)h_2\le N}} \ell_1(\ell_1+\ell_2) =
\sum_{\substack{\ell_1,\ell_2,h_1,h_2\\\ell_1(h_1+h_2)+\ell_2 h_2\le N}} \ell_1^2+\ell_1 \ell_2\,.
=\\=
\cfrac{N^4}{24} \big[ 2\cdot\zeta(1,3)+\zeta(2,2)\big] =
\cfrac{N^4}{24} \left[ 2\cdot\cfrac{\zeta(4)}{4}+\cfrac{3\zeta(4)}{4}\right]
= \cfrac{N^4}{24} \cdot \cfrac{5}{4} \cdot \cfrac{\pi^4}{90}
\end{multline}
(see~\S5 in~\cite{Zorich:square:tiled} for details of the
calculation).

Adding the contributions of the two diagrams and applying
$\left. 2\frac{d}{dN}\right|_{N=1}$ we finally get
$$
\operatorname{Vol}\mathcal{H}_1(2)= \cfrac{\pi^4}{120}\,.
$$

%####################################################################
%####################################################################
%####################################################################
\section{Impact of the choice of the integer lattice on
diagram-by-diagram counting of Masur--Veech volumes}
\label{s:contibution:of:diag:for:two:lattices}

Recall the following two natural choices of the integer lattice in
period coordinates of a stratum of quadratic differentials.

\begin{enumerate}
\item the subset of $H^1_-(\hat S, \{\hat{P}_1,\dots, \hat{P}_\noz\};\C)$
consisting of those linear forms which take values in $\Z\oplus i\Z$
on $H_1^-(\hat S, \{\hat{P}_1,\dots,\hat{P}_\noz\};\Z)$
\item $H^1_-(\hat S, \{\hat{P}_1,\dots, \hat{P}_\noz\};\C)\cap
H^1(\hat S,\{\hat{P}_1,\dots, \hat{P}_\noz\};\Z\oplus i\Z)$
\end{enumerate}

Here we do not mark the preimages of simple poles, i.e.
$\hat{P}_1,\dots, \hat{P}_\noz$ are preimages of zeroes of
the quadratic differential under the double cover. The difference between the two choices affects the linear
holonomy along saddle connections joining two distinct zeroes. Under
the first convention the linear holonomy along such saddle
connections belongs to the half integer lattice $\frac{1}{2}\Z\oplus
\frac{i}{2}\Z$ while under the second convention it belongs to the
integer lattice $\Z\oplus i\Z$. This implies that the first lattice
in the period coordinates is a proper sublattice of index
$4^{s-1}$ of
the second one, where $s$ is the number of zeroes of the
quadratic differential.

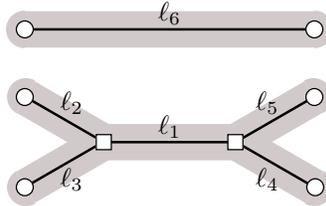
\begin{figure}[htb]
\includegraphics{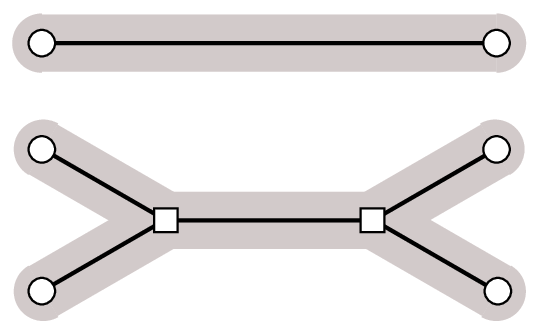}
   %
% \special{
% psfile=graph11i6.eps
% hscale=80
% vscale=80
% hoffset=110
% voffset=-120
% }
   %
\begin{picture}(0,0)(0,0)
\put(-9,-50){$\ell_1$}
\put(-46,-41){$\ell_2$}
\put(-46,-70){$\ell_3$}
\put(28,-70){$\ell_4$}
\put(28,-41){$\ell_5$}
\put(-9,-7){$\ell_6$}
\end{picture}

\vspace{80pt} % 120pt
\caption{
\label{fig:Q11i6}
A separatrix diagram for $\cQ(1^2, -1^6)$
}
\end{figure}

Thus, in the case of the stratum $\cQ(1^2, -1^6)$, it is a sublattice
of index $4$. Note, however, that the contributions of individual
separatrix diagrams change by the factors, which are, in general,
different from the index of one lattice in the other. Consider, for
example the separatrix diagram as in Figure~\ref{fig:Q11i6}
representing the stratum $\cQ(1^2, -1^6)$. The absolute contribution
of this separatrix diagram is twice bigger under the first choice of
the lattice than under the second one. Indeed, under the first choice
of the lattice in period coordinates, the parameter $\ell_1$ is
half-integer, as well as all the other parameters $\ell_2,\dots,
\ell_6,h,\phi$, (where $h,\phi$ are the height and the twist of the
single cylinder) whereas $\ell_1$ is integer under the second choice
of the lattice,  and the other parameters are half-integers. Hence,
the number of partitions of a given natural number $w$ (representing
the length of the waist curve of the single cylinder) into the sum
$$
w=2(\ell_1+\ell_2+\ell_3+\ell_4+\ell_5)
$$
is asymptotically twice bigger under the first choice of the lattice.

Now let us perform the computation for this diagram under the first
convention of the choice of the lattice. When the zeroes and poles
are \textit{not labeled}, the diagram has symmetry of order $4$. Since
the twist $\phi$ is half-integer, there are $2w$ choices of $\phi$.
Recall also, that the standard pillow as in
Figure~\ref{fig:square:pillow} has area $1/2$. Thus, under the first
choice of the lattice in period coordinates, the number of pillowcase
covers of order at most $N$ corresponding to this separatrix diagram
has the following asymptotics as $N\to+\infty$ (compare to
computation~\eqref{eq:D1}):
\begin{multline*}
\frac{1}{4}\sum_{\substack{\ell_1,\ell_2,\ell_3,\ell_4,\ell_5,h\in\N/2\\
(2(\ell_1+\ell_2+\ell_3+\ell_4+\ell_5))\cdot h\le N/2}} 2(2(\ell_1+\ell_2+\ell_3+\ell_4+\ell_5))
\sim
\frac{1}{4}\sum_{\substack{w,H\in\N\\w\cdot H\le N}} 2w\cdot\frac{w^4}{4!}
=\\=
\frac{1}{2\cdot 4!}\ \sum_{\substack{w,H\in\N\\w\le \frac{N}{H} }} w^5
\sim
\frac{1}{2\cdot 4!}\ \sum_{H\in\N} \frac{1}{6}\cdot \left(\frac{N}{H}\right)^6
=
\frac{N^6}{12\cdot 4!}\cdot\sum_{H\in\N} \frac{1}{H^6}
=
\frac{N^6}{12\cdot 4!}\cdot \zeta(6)\,.
\end{multline*}
Here in the first equivalence we passed from the half-integer
parameter $h$ to the integer parameter $H=2h$ replacing the condition
$wh\le N/2$ by the equivalent condition $wH\le N$. Applying $\left.
2\frac{d}{dN}\right|_{N=1}$ and multiplying by the factor $6!\cdot
2!$ responsible for numbering of zeroes and poles, we get
the total contribution $60\zeta(6)$ to the volume
$\Vol^{(1)}\cQ^{numbered}_1(1^2,-1^6)$ defined under the first
convention on the choice of the lattice.

Similar computations for each separatrix
diagram in this stratum are cumbersome, so,
following~\cite{AEZ:Dedicata}, we
distribute the diagrams into groups organized in the following way.

Each connected component of the separatrix diagram is encoded by a
vertex of a graph decorated with an ordered pair of natural numbers
indicating the number of zeroes and poles living at the corresponding
component. A flat cylinder joining two connected components of a
separatrix diagram is encoded by an edge of the graph. For example,
the separatrix diagram from Figure~\ref{fig:Q11i6} contains two
connected components joined by a single cylinder. The corresponding
graph contains two vertices joined by a single edge; one vertex is
decorated with the pair $(2,4)$ (standing for $2$ zeroes and $4$
poles) and the other vertex is decorated with the pair $(0,2)$
(standing for $0$ zeroes and $2$ poles). This graph is the top entry
of the left column in Table~\ref{table:normalization:Q11i6}.

\begin{table}[htb]
$$
\begin{array}{|c|c|c|}
\hline
\text{Tree} &\text{Contribution to }\Vol^{(1)}&\cfrac{\text{Contribution to }\Vol^{(1)}}{\text{Contribution to }\Vol^{(2)}}\\
[-\halfbls]
&&\\
\hline&&\\
\begin{picture}(35,10)(0,0)
\put(0,0){\circle{4}}
\put(-5,5){\tiny 2,4}
\put(2,0){\line(1,0){26}}
\put(30,0){\circle{4}}
\put(25,5){\tiny 0,2}
\end{picture}
& 60\zeta(6)&2
\\
&&\\
\hline&&\\
\begin{picture}(35,10)(0,0)
\put(0,0){\circle{4}}
\put(-5,5){\tiny 1,3}
\put(2,0){\line(1,0){26}}
\put(30,0){\circle{4}}
\put(25,5){\tiny 1,3}
\end{picture}
&80\zeta(6)&2^7
\\
&&\\
\hline&&\\
\begin{picture}(70,10)(-5,0)
\put(0,0){\circle{4}}
\put(-5,5){\tiny 0,2}
\put(2,0){\line(1,0){26}}
\put(30,0){\circle{4}}
\put(25,5){\tiny 2,2}
\put(32,0){\line(1,0){26}}
\put(60,0){\circle{4}}
\put(55,5){\tiny 0,2}
\end{picture}
&72\zeta(2)\zeta(4)&2
\\
&&\\
\hline
&&\\
\begin{picture}(70,10)(-5,0)
\put(0,0){\circle{4}}
\put(-5,5){\tiny 1,3}
\put(2,0){\line(1,0){26}}
\put(30,0){\circle{4}}
\put(25,5){\tiny 1,1}
\put(32,0){\line(1,0){26}}
\put(60,0){\circle{4}}
\put(55,5){\tiny 0,2}
\end{picture}
&48\zeta(2)\zeta(4)&2^5
\\
&&\\
\hline&&\\
\begin{picture}(100,10)(-5,0)
\put(0,0){\circle{4}}
\put(-5,5){\tiny 0,2}
\put(2,0){\line(1,0){26}}
\put(30,0){\circle{4}}
\put(25,5){\tiny 1,1}
\put(32,0){\line(1,0){26}}
\put(60,0){\circle{4}}
\put(55,5){\tiny 1,1}
\put(62,0){\line(1,0){26}}
\put(90,0){\circle{4}}
\put(85,5){\tiny 0,2}
\end{picture}
&24\zeta^3(2)&2^3
\\
&&\\
\hline
&&\\
\begin{picture}(70,10)(-5,5)
\put(0,0){\circle{4}}
\put(-5,5){\tiny 0,2}
\put(2,0){\line(1,0){26}}
\put(30,0){\circle{4}}
\put(25,5){\tiny 2,0}
\put(32,0){\line(2,1){26}}
\put(60,14){\circle{4}}
\put(65,14){\tiny 0,2}
\put(32,0){\line(2,-1){26}}
\put(60,-14){\circle{4}}
\put(65,-14){\tiny 0,2}
\end{picture}
&4\zeta^3(2)&2
\\
&&\\
&&\\
\hline
\end{array}
$$
\caption{
\label{table:normalization:Q11i6}
Table of diagram contributions to the Masur--Veech volume
$\Vol\cQ_1(1^2, -1^6)$ in normalizations $(1)$ and $(2)$}
\end{table}

Note that the stratum $\cQ(1^2,-1^6)$ corresponds to genus zero, so
the underlying topological surface is a sphere. This implies that the
graph defined by a separatrix diagram representing the stratum
$\cQ(1^2,-1^6)$ is, actually, a tree. The first column of
Table~\ref{table:normalization:Q11i6} provides the list of all
possible decorated trees which appear for the stratum
$\cQ(1^2,-1^6)$. It is easy to verify that the ratio of contributions
of a given separatrix diagram to the volume of the stratum
$\cQ(1^\noz,-1^{\noz+4})$ computed under the two conventions on the
choice of the integer lattice depends only on the corresponding
decorated tree. We group together all the diagrams corresponding to
each decorated tree and indicated in the second column the
corresponding contribution to the volume under the first choice of
the lattice (using~\cite[\S 3.8]{AEZ:Dedicata} as the source). In the
third column we give the ratio of the contributions represented by
the corresponding tree. For example, the tree in the first line
represents the unique diagram shown in Figure~\ref{fig:Q11i6}; as it
was computed above its contribution to the volume under the first
choice of the lattice is $60\zeta(6)$ and the contribution to the
volume under the second choice of the lattice is twice smaller. These
data constitute the first line of
Table~\ref{table:normalization:Q11i6}.

Recall that the normalization factor between the two lattices in the
period coordinates of the stratum $\cQ(1^2,-1^6)$ is $4$. However,
observing Table~\ref{table:normalization:Q11i6} the reader can see
that the individual contributions of diagrams differ by factors $2$,
$2^3$, $2^5$, $2^7$.

Note that the trees with the same number of edges provide
contributions of the same ``arithmetic'' nature, namely the total
contribution of $1, 2, 3$-cylinder diagrams are
$$
140\zeta(6)+120\zeta(2)\zeta(4)+28\zeta^3(2)=\cfrac{\pi^6}{2}=\Vol^{(1)}\cQ(1^2,-1^6)
$$
respectively under the first choice of the lattice and
$$
\frac{245}{8}\zeta(6)+\frac{75}{2}\zeta(2)\zeta(4)+5\zeta^3(2)
=\cfrac{\pi^6}{8}=\Vol^{(2)}\cQ(1^2,-1^6)
$$
respectively under the second choice. The volumes $\Vol^{(1)}$ and
$\Vol^{(2)}$ differ by the factor $4$ as expected.

We get a polynomial identity
$$
140\zeta(6)+120\zeta(2)\zeta(4)+28\zeta^3(2)=\cfrac{\pi^6}{2}=
4\Big(\frac{245}{8}\zeta(6)+\frac{75}{2}\zeta(2)\zeta(4)+5\zeta^3(2)\Big)
$$
in zeta values at even integers. Considering other strata
$\cQ(1^\noz,-1^{\noz+4})$ we get an infinite series of analogous
identities in zeta values at even integers.

We did not study the identities resulting from different choices of
the lattice in period coordinates for more general strata of
meromorphic quadratic differentials with at most simple poles in
genus zero. Considering zeroes of even order might produce identities
of much more elaborate arithmetic nature.

If our guess that the contribution of $k$-cylinder square-tiled
surfaces to a given stratum of Abelian differentials is a polynomial
in multiple zeta values with rational (or even integer) coefficients
is true, then playing with different choices of an integer lattice we
will get infinite series of mysterious polynomial identities in
multiple zeta values.

Another challenge is to see whether one can obtain some information
about volume asymptotics for large genera playing with the choice of
an integer lattice. We leave both questions as a problem, which might
be interesting to study.

\begin{Problem}
Describe and study polynomial identities on multiple zeta values
arising from $k$-cylinder contributions to the Masur--Veech volumes
under different choices of integer lattices in period coordinates.
Study these identities in asymptotic regimes when the genus of the
surface or the number of simple poles tends to infinity.
\end{Problem}

%####################################################################
%####################################################################
%####################################################################
\section{Tables of the Masur--Veech volumes of low-dimensional strata
in the moduli spaces of meromorphic quadratic differentials with at
most simple poles}
\label{s:tables:volumes}

In the tables below we present the volumes of all low-dimensional
strata of quadratic differentials up to dimension $d=6$. The
approximate values of Masur--Veech volumes were actively used in the
papers~\cite{Goujard:SV} and~\cite{Goujard:volumes} to debug the
rigorous theoretically found values of the Masur--Veech volumes of
the strata of quadratic differentials and related Siegel--Veech
constants.

Each table is organized as follows.

The left column indicates the stratum $\cQ(d_1,\dots,d_k)$.
The second column provides the rigorous rational
number $r$ in the absolute contribution,
see~\eqref{eq:general:contribution},
$$
c_1(d_1,\dots,d_k)=r\cdot \zeta(d)
$$
of $1$-cylinder pillowcase covers to the Masur--Veech volume
$\Vol\cQ(d_1,\dots,d_k)$.

The third column provides experimental statistical data, namely,
the approximate proportions of the number of square-tiled
surfaces with $1$ cylinder, $2$-cylinders, and so on tiled
with tiny squares.

Combining data from the second and the third column, we
provide in the fourth column the resulting
approximate value of the volume
obtained experimentally.

In the right two columns we present the value of the Masur--Veech
volume of the corresponding stratum obtain by rigorous methods,
namely, the approximate numerical value in the second column from the
right and the exact value in the rightmost column.

The rightmost column contains a symbol indicating the rigorous
methods of computation of the exact value of the Masur--Veech volumes
of the stratum under consideration (see \cite{Goujard:volumes} for
more details):

\begin{itemize}
\item[EO.]
By the results of Eskin--Okounkov \cite{Eskin:Okounkov:pillowcase},
the generating functions for the number of pillowcase covers are
quasimodular. The volumes are derived from the asymptotics of these
functions, which can be easily computed using the quasimodularity
property. Note that this method do not give the volumes of the
connected components of the strata.
\item[g0.]
For the case of genus $0$ surfaces, there exists a closed formula for
volumes, that was conjectured by Kontsevich and then proved by
Athreya--Eskin--Zorich in \cite{AEZ:Dedicata} and
\cite{AEZ:genus:0}.
\item[hyp.]
For hyperelliptic components, volumes are easily deduced from the
volumes of strata of genus $0$ surfaces.
\item[nv.]
Some strata have a special property: they are non-varying, meaning
that the Lyapunov spectrum (and so the Siegel--Veech constant) is the
same for any Teichm\"uller curve, see \cite{Chen:Moeller}. The volume
is then deduced from the (common value of the) Siegel--Veech constant
using the technics of Masur--Zorich (see \cite{Goujard:SV});
\item[diag.]
Finally, as for $1$-cylinder surfaces, the volumes of low-dimensional
strata can be computed diagram by diagram.
\end{itemize}

We always use normalization as in~\cite{AEZ:genus:0}; see
also Convention~\ref{con:area:1:2}.
\vspace*{1cm}

\newpage

\input{cyl_diags_dim4_27_04_modified.tex}

For $\cQ(2^2)$ the exact proportions
are more bulky, so we give them separately:
$$
\text{Frequences for }\cQ(2,2):\quad
  % 1:2:3-\text{cylinder surfaces in }\cQ(2,2):\qquad
\frac{17\zeta(4)}{4}\ :\ 16\zeta(3)-\frac{33}{2}\zeta(4)\ :\ 4\zeta(2)-16\zeta(3)+\frac{49\zeta(4)}{4}\,.
$$

\newpage
\input{cyl_diags_dim5_27_04_modified.tex}

\newpage
\input{cyl_diags_dim6_split_27_04_modified.tex}

\end{document}

%% file: cyl_diags_dim4_27_04_modified.tex
\centerline{Strata of dimension $4$}
\begin{table}[hbt]
\footnotesize % \small
$$
\begin{array}{|c|c|c|c|c|c|}
\hline &&&&\multicolumn{2}{|c|}{}\\ [-\halfbls]
\text{Component}& \text{Abs. contr.}             & \text{Statistics of}    & \text{Experim.} & \multicolumn{2}{|c|}{\text{Theoretical value}}\\
\text{of the}   & r\cdot\zeta(4) \text{ of} & \text{frequency of}  & \text{value}        & \multicolumn{2}{|c|}{\text{of the volume}}\\
\cline{5-6}
\text{stratum}  & \text{1-cylinder}         & 1:2:\cdots-\text{cylinder}& \text{of the}       &&\\
           & \text{surfaces}           & \text{surfaces}      & \text{volume}       & \text{Approx.}&\text{Exact}
\\
[-\halfbls] &&&&&\\ \hline
\multicolumn{6}{c}{} \\ [-\halfbls]
\multicolumn{6}{c}{} \\ [-\halfbls]
\multicolumn{6}{c}{\text{genus 0}}\\
\multicolumn{6}{c}{} \\ [-\halfbls]
\hline &&&& &\\ [-\halfbls]
\cQ(1, -1^5)&40&0.4382:0.5618&1.014\cdot \pi^{4}&1.000\cdot\pi^4& {\pi^4}\\
&&0.4444:0.5556&&&\\
&&\frac{4}{9}:\frac{5}{9}&&&\gzero, \EO\\
  %----------------------------------------------------------------------
[-\halfbls] &&&&&\\ \hline
\multicolumn{6}{c}{} \\ [-\halfbls]
\multicolumn{6}{c}{} \\ [-\halfbls]
\multicolumn{6}{c}{\text{genus 1}}\\
\multicolumn{6}{c}{} \\ [-\halfbls]
\hline &&&&& \\ [-\halfbls]
\cQ(1^2, -1^2)&50/3&0.5724:0.4276&0.324 \cdot\pi^{4}&0.3333\cdot\pi^4&{\frac{1}{3}\cdot\pi^4}\\
&&0.5556:0.4444&&&\\
&&\frac{4}{9}:\frac{5}{9}&&&\hyp, \diag, \EO\\
[-\halfbls] &&&&&\\ \hline &&&&& \\ [-\halfbls]
\cQ(3, -1^3)&30&0.6016:0.3984&0.554 \cdot\pi^{4}&0.5556 \cdot\pi^{4}&{\frac{5}{9}\cdot\pi^4}\\
&&0.6000:0.4000&&&\\
&&\frac{3}{5}:\frac{2}{5}&&&\diag,\nonvar, \EO\\
  %----------------------------------------------------------------------
[-\halfbls] &&&&&\\ \hline
\multicolumn{6}{c}{} \\ [-\halfbls]
\multicolumn{6}{c}{} \\ [-\halfbls]
\multicolumn{6}{c}{\text{genus 2}}\\
\multicolumn{6}{c}{} \\ [-\halfbls]
\hline &&&&& \\ [-\halfbls]
\cQ(2^2)&17/4&0.7065:0.2130:0.0805&{0.666} \cdot\pi^{2}&0.6666\cdot\pi^2&{\frac{2}{3}\cdot\pi^2}\\
&&0.6991: 0.2089 :0.0920&&&\\
&&
     %\textcolor{dblue}{\scriptstyle{\frac{17\zeta(4)}{4}}:\scriptstyle{ 16\zeta(3)-\frac{33}{2}\zeta(4)} :\scriptstyle{ 4\zeta(2)-16\zeta(3)+\frac{49\zeta(4)}{4}} }
&&&\hyp, \diag, \EO\\
[-\halfbls] &&&&&\\ \hline &&&&& \\ [-\halfbls]
\cQ(5, -1)&12&0.6488:0.3512&0.206 \cdot\pi^{4}&0.2074\cdot\pi^4&{\frac{28}{135}\cdot\pi^4}\\
&&0.6429:0.3571&&&\\
&&\frac{9}{14}:\frac{5}{14}&&&\diag, \nonvar, \EO\\
  %----------------------------------------------------------------------
[-\halfbls] &&&&&\\ \hline
\end{array}
$$
\end{table}

% The first line in the column ``Statistics of frequency of
% $1:2:\cdots$-cylinder surfaces'' represents the experimental data;
% the second line --- approximate theoretical one (when accessible);
% the third one --- exact proportions.
% For $\cQ(2^2)$ the exact proportions
% are more bulky, so we give them separately:
%    %
% $$\text{Frequences for }\cQ(2,2):\quad
%   % 1:2:3-\text{cylinder surfaces in }\cQ(2,2):\qquad
% \frac{17\zeta(4)}{4}\ :\ 16\zeta(3)-\frac{33}{2}\zeta(4)\ :\ 4\zeta(2)-16\zeta(3)+\frac{49\zeta(4)}{4}\,.
% $$
%
% We always use normalization as in~\cite{AEZ:genus:0}; see
% also Convention~\ref{con:area:1:2}.
% The abbreviations in the last column indicate the method (or independent methods)
% of evaluation of exact value of the volume, namely
%
% $$
% \begin{array}{l|l}
% \EO &
% \text{Method of Eskin--Okounkov}
% \\
% \hline
% \nonvar &
% \text{Stratum with non-varying }c_{\mathit{area}}
% \\
% \hline
% \hyp & \text{Hyperelliptic component}
% \\
% \gzero & \text{Stratum in genus } 0
% \\
% \hline
% \diag & \text{Diagram-by-diagram computation}
% \end{array}
% $$

%% file: cyl_diags_dim5_27_04_modified.tex
\centerline{Strata of dimension $5$}
\begin{table}[hbt]
\footnotesize % \small
$$
\begin{array}{|c|c|c|c|c|c|}
\hline &&&&\multicolumn{2}{|c|}{}\\ [-\halfbls]
\text{Component}& \text{Abs. contr.}             & \text{Statistics of}    & \text{Experim.} & \multicolumn{2}{|c|}{\text{Theoretical value}}\\
\text{of the}   & r\cdot\zeta(5) \text{ of} & \text{frequency of}  & \text{value}        & \multicolumn{2}{|c|}{\text{of the volume}}\\
\cline{5-6}
\text{stratum}  & \text{1-cylinder}         & 1:2:\cdots-\text{cylinder}& \text{of the}       &&\\
           & \text{surfaces}           & \text{surfaces}      & \text{volume}       & \text{Approx.}&\text{Exact}
\\
[-\halfbls] &&&&&\\ \hline
\multicolumn{6}{c}{} \\ [-\halfbls]
\multicolumn{6}{c}{} \\ [-\halfbls]
\multicolumn{6}{c}{\text{genus 0}}\\
\multicolumn{6}{c}{} \\ [-\halfbls]
\hline &&&&& \\ [-\halfbls]
\cQ(2, -1^6)&60&0.2472:0.6740:0.0789&2.584 \cdot\pi^{4}&2.666\cdot\pi^4&{\frac{8}{3}\cdot\pi^4}\\
&&0.2395 :  \mydummy   :  \mydummy  &&&\gzero, \EO\\
  %----------------------------------------------------------------------
[-\halfbls] &&&&&\\ \hline
\multicolumn{6}{c}{} \\ [-\halfbls]
\multicolumn{6}{c}{} \\ [-\halfbls]
\multicolumn{6}{c}{\text{genus 1}}\\
\multicolumn{6}{c}{} \\ [-\halfbls]
\hline &&&&& \\ [-\halfbls]
\cQ(2, 1, -1^3)&45&0.4919:0.4472:0.0610&0.974 \cdot\pi^{2}&\pi^4&{\pi^4}\\
&&0.4790 :  \mydummy   :  \mydummy  &&&\nonvar, \EO\\
[-\halfbls] &&&&&\\ \hline &&&&&\\ [-\halfbls]
\cQ(4, -1^4)&84&0.4309:0.5163:0.0528&2.075 \cdot\pi^{2}&2\cdot\pi^4&{2\cdot\pi^4} \\
&&0.4471 :  \mydummy   :  \mydummy  &&&\nonvar, \EO\\
  %----------------------------------------------------------------------
[-\halfbls] &&&&&\\ \hline
\multicolumn{6}{c}{} \\ [-\halfbls]
\multicolumn{6}{c}{} \\ [-\halfbls]
\multicolumn{6}{c}{\text{genus 2}}\\
\multicolumn{6}{c}{} \\ [-\halfbls]
\hline &&&&& \\ [-\halfbls]
\cQ(2, 1^2)&11/2&0.4398:0.4667:0.0935&0.133 \cdot\pi^{4}&0.1333\cdot\pi^4&{\frac{2}{15}\cdot\pi^4}\\
&&0.4391 : 0.4621  : 0.0988 &&&\hyp, \diag, \EO\\
  % &&
  % \textcolor{dblue}{\scriptstyle{\frac{11\zeta(5)}{2}}:\scriptstyle{ 3\zeta(2)\zeta(3)+\frac{16\zeta(4)}{3}-\frac{11\zeta(5)}{2}} :\scriptstyle{\zeta(2)(\frac{8\zeta(2)}{3}-3\zeta(3)) } }
  % &&&\\
  %\hyp, \diag, \EO\\
[-\halfbls] &&&&&\\ \hline &&&&& \\ [-\halfbls]
\cQ(4, 1, -1)&68/3&0.4772:0.4854:0.0374&0.506 \cdot\pi^{4}&0.5333\cdot\pi^4&{\frac{8}{15}\cdot\pi^4}\\
&&0.4524 :  \mydummy   :  \mydummy  &&&\nonvar, \EO\\
[-\halfbls] &&&&&\\ \hline &&&&& \\ [-\halfbls]
\cQ(3, 2, -1)&115/6&0.5528:0.3813:0.0659&0.369 \cdot\pi^{4}&0.3704\cdot\pi^4&{\frac{10}{27}\cdot\pi^4}\\
&&0.5509 :  \mydummy   :  \mydummy  &&&\nonvar, \EO\\
[-\halfbls] &&&&&\\ \hline &&&&& \\ [-\halfbls]
\cQ^{hyp}(6, -1^2)&65/12&0.3057:0.6374:0.0569&0.189 \cdot\pi^{4}&0.1777\cdot\pi^4&{\frac{8}{45}\cdot \pi^4}\\
&&0.3243 :  \mydummy   :  \mydummy  &&&\hyp\\
[-\halfbls] &&&&&\\ \cline{1-6} &&&&& \\ [-\halfbls]
\cQ^{non}(6, -1^2)&181/3&0.5366:0.4008:0.0626&1.197 \cdot\pi^{4}&1.1852\cdot\pi^4&{\frac{32}{27}\cdot\pi^4}\\
&&0.5419 :  \mydummy   :  \mydummy  &&&\nonvar, \EO\\
  %----------------------------------------------------------------------
[-\halfbls] &&&&&\\ \hline
\multicolumn{6}{c}{} \\ [-\halfbls]
\multicolumn{6}{c}{} \\ [-\halfbls]
\multicolumn{6}{c}{\text{genus 3}}\\
\multicolumn{6}{c}{} \\ [-\halfbls]
\hline &&&&& \\ [-\halfbls]
\cQ(8)&56/3&0.5211:0.4024:0.0764&0.381 \cdot\pi^{4}&0.3704\cdot\pi^4&\frac{10}{27}\cdot\pi^4\\
&&&&&\EO\\
  %----------------------------------------------------------------------
[-\halfbls] &&&&&\\ \hline
\end{array}
$$
\end{table}

\noindent
Exact values of frequences of $1:2:3$-cylinder
square-tiled surfaces for $\cQ(2,1^2)$
are:
$$
%\text{Frequences for }\cQ(2,1^2):\quad
\frac{11\zeta(5)}{2}\ :\
3\zeta(2)\zeta(3)+\frac{16\zeta(4)}{3}-\frac{11\zeta(5)}{2}\ :\
\zeta(2)(\frac{8\zeta(2)}{3}-3\zeta(3))
$$

% The abbreviations in the last column indicate the method (or independent methods)
% of evaluation of exact value of the volume, namely:
%
% $$
% \begin{array}{l|l}
% \EO &
% \text{Method of Eskin--Okounkov}
% \\
% \hline
% \nonvar &
% \text{Stratum with non-varying }c_{\mathit{area}}
% \\
% \hline
% \hyp & \text{Hyperelliptic component}
% \\
% \gzero & \text{Stratum in genus } 0
% \\
% \hline
% \diag & \text{Diagram-by-diagram computation}
% \end{array}
% $$

%% file: cyl_diags_dim6_split_27_04_modified.tex
\centerline{Strata of dimension $6$}
\begin{table}[hbt]
\footnotesize % \small
$$
\begin{array}{|c|c|c|c|c|c|}
\hline &&&&\multicolumn{2}{|c|}{}\\ [-\halfbls]
\text{Component}& \text{Abs.\! cont.}             & \text{Statistics of}    & \text{Experim.} & \multicolumn{2}{|c|}{\text{Theoretical value}}\\
\text{of the}   & r\cdot\zeta(6) \text{ of} & \text{frequency of}  & \text{value}        & \multicolumn{2}{|c|}{\text{of the volume}}\\
\cline{5-6}
\text{stratum}  & \text{1-cylinder}         & 1:2:\cdots-\text{cylinder}& \text{of the}       &&\\
           & \text{surfaces}           & \text{surfaces}      & \text{volume}       & \text{Approx.}&\text{Exact}
\\
[-\halfbls] &&&&&\\ \hline
\multicolumn{6}{c}{} \\ [-\halfbls]
\multicolumn{6}{c}{} \\ [-\halfbls]
\multicolumn{6}{c}{\text{genus 0}}\\
\multicolumn{6}{c}{} \\ [-\halfbls]
\hline &&&&& \\ [-\halfbls]
\cQ(1^2, -1^6)&140&\text{0.2943:0.4236:0.2821}&0.5034 \cdot\pi^{6}&0.5000 \cdot\pi^6&{\frac{1}{2}\cdot\pi^6} \\%\;(\gzero, \EO)\\
[-\halfbls] &&&&&\\ \hline &&&&& \\ [-\halfbls]
\cQ(3, -1^7)&84&\text{0.1106:0.6187:0.2707}&0.8037 \cdot\pi^{6}&0.8000\cdot\pi^6&{\frac{3}{4}\cdot\pi^6} \\%\;(\gzero, \EO)\\
  %----------------------------------------------------------------------
[-\halfbls] &&&&&\\ \hline
\multicolumn{6}{c}{} \\ [-\halfbls]
\multicolumn{6}{c}{} \\ [-\halfbls]
\multicolumn{6}{c}{\text{genus 1}}\\
\multicolumn{6}{c}{} \\ [-\halfbls]
\hline &&&&& \\ [-\halfbls]
\cQ(1^3, -1^3)&77&\text{0.4366:0.4000:0.1634}&0.1866 \cdot\pi^{6}&{0.1837\cdot\pi^6} &\frac{11}{60}\cdot{\pi^6} \\%\;(\EO)\\
[-\halfbls] &&&&&\\ \hline &&&&& \\ [-\halfbls]
\cQ(3, 1, -1^4)&126&\text{0.4000:0.4520:0.1480}&0.3333 \cdot\pi^{6}&0.3333\cdot\pi^6&{\frac{1}{3}\cdot\pi^6} \;(\nonvar)\\%, \EO)\\
[-\halfbls] &&&&&\\ \hline &&&&& \\ [-\halfbls]
\cQ(2^2, -1^4)&110&\text{0.364:0.511:0.108:0.016}&3.1544 \cdot\pi^{4}&{3.0222\cdot\pi^4}&\frac{136}{45}\cdot\pi^4 \\%\;(\EO)\\
[-\halfbls] &&&&&\\ \hline &&&&&\\ [-\halfbls]
\cQ(5, -1^5)&210&\text{0.3276:0.5301:0.1423}&0.6783 \cdot\pi^{6}&0.7000\cdot\pi^6&{\frac{7}{10}\cdot\pi^6} \;(\nonvar) \\%, \EO)\\
  %----------------------------------------------------------------------
[-\halfbls] &&&&&\\ \hline
\multicolumn{6}{c}{} \\ [-\halfbls]
\multicolumn{6}{c}{} \\ [-\halfbls]
\multicolumn{6}{c}{\text{genus 2}}\\
\multicolumn{6}{c}{} \\ [-\halfbls]
\hline &&&&& \\ [-\halfbls]
\cQ(1^4)&49/3&\text{0.2512:0.5577:0.1911}&0.0688 \cdot\pi^{6}&0.0666\cdot\pi^6&{\frac{\pi^6}{15}} \\%\;(\hyp, \EO)\\
[-\halfbls] &&&&&\\ \hline &&&&& \\ [-\halfbls]
\cQ(3, 1^2, -1)&119/3&\text{0.3593:0.5081:0.1325}&0.1168 \cdot\pi^{6}&0.1104\cdot\pi^6&{\frac{1}{9}\cdot\pi^6} \;(\nonvar) \\%, \EO)\\
[-\halfbls] &&&&&\\ \hline &&&&& \\ [-\halfbls]
\cQ(2^2, 1, -1)&94/3&\text{0.391:0.503:0.096:0.0098}&0.837 \cdot\pi^{4}&0.8000\cdot\pi^4&{\frac{4}{5}\cdot\pi^4} \;(\nonvar) \\%, \EO)\\
[-\halfbls] &&&&&\\ \hline &&&&& \\ [-\halfbls]
\cQ(5, 1, -1^2)&189/2&\text{0.4252:0.4569:0.1179}&0.2352 \cdot\pi^{6}&0.2333\cdot\pi^6&{\frac{7}{30}\cdot\pi^6} \;(\nonvar) \\%, \EO)\\
[-\halfbls] &&&&&\\ \hline &&&&& \\ [-\halfbls]
\cQ(4, 2, -1^2)&317/4&\text{0.450:0.449:0.089:0.012}&1.8409\cdot\pi^{4}&1.866\cdot \pi^4& \frac{28}{15}\cdot\pi^4 \;(\nonvar) \\%, \EO)\\
[-\halfbls] &&&&&\\ \hline &&&&& \\ [-\halfbls]
\cQ^{hyp}(3^2, -1^2)&161/30&\text{0.1724:0.6520:0.1756}&0.0329 \cdot\pi^{6}&0.0333\cdot\pi^6&{\frac{1}{30}\cdot\pi^6} \\%\;(\hyp)\\
[-\halfbls] &&&&&\\ \cline{1-6} &&&&& \\ [-\halfbls]
\cQ^{nh}(3^2, -1^2)&1106/15&\text{0.4894:0.3951:0.1154}&0.1594 \cdot\pi^{6}&0.1630\cdot\pi^6&{\frac{22}{135}\cdot\pi^6 } \;(\nonvar)\\%, \EO)\\
[-\halfbls] &&&&&\\ \hline &&&&& \\ [-\halfbls]
\cQ(7, -1^3)&441/2&\text{0.4179:0.4626:0.1195}&0.5583 \cdot\pi^{6}&0.54\cdot\pi^6&\frac{27}{50}\cdot{\pi^6} \\%\;(\EO)\\
[-\halfbls] &&&&&\\ \hline &&&&& \\ [-\halfbls]
  %----------------------------------------------------------------------
% [-\halfbls] &&&&&\\ \hline
% \end{array}
% $$
% \end{table}
%
% \newpage
% \centerline{Strata of dimension $6$ (continued)}
% \begin{table}[hbt]
% \footnotesize % \small
% $$
% \begin{array}{|c|c|c|c|c|c|}
% \hline &&&&\multicolumn{2}{|c|}{}\\ [-\halfbls]
% \text{Component}& \text{Impact}             & \text{Statistics of}    & \text{Experim.} & \multicolumn{2}{|c|}{\text{Theoretical value}}\\
% \text{of the}   & r\cdot\zeta(6) \text{ of} & \text{frequency of}  & \text{value}        & \multicolumn{2}{|c|}{\text{of the volume}}\\
% \cline{5-6}
% \text{stratum}  & \text{1-cylinder}         & 1:2:\cdots-\text{cylinder}& \text{of the}       &&\\
%            & \text{surfaces}           & \text{surfaces}      & \text{volume}       & \text{Approx.}&\text{Exact}
% \\
%   %----------------------------------------------------------------------
% [-\halfbls] &&&&&\\ \hline
% \multicolumn{6}{c}{} \\ [-\halfbls]
% \multicolumn{6}{c}{} \\ [-\halfbls]
% \multicolumn{6}{c}{\text{genus 3}}\\
% \multicolumn{6}{c}{} \\ [-\halfbls]
% \hline &&&&& \\ [-\halfbls]
%   %
\cQ(7, 1)&37&\text{0.3724:0.5171:0.1106}&0.1051 \cdot\pi^{6}&0.1028\cdot\pi^6&{\frac{18}{175}\cdot\pi^6} \;(\nonvar) \\%, \EO)\\
[-\halfbls] &&&&&\\ \hline &&&&& \\ [-\halfbls]
\cQ^{hyp}(6, 2)&65/48&\text{0.237:0.608:0.137:0.018}&0.0597 \cdot\pi^{4}&0.0593\cdot\pi^4&{\frac{8}{135}\cdot\pi^4} \\%\;(\hyp)\\
[-\halfbls] &&&&&\\ \cline{1-6} &&&&& \\ [-\halfbls]
\cQ^{nh}(6, 2)&389/12&\text{0.473:0.380:0.112:0.035}&0.7155 \cdot\pi^{4}&0.7111\cdot\pi^4&\frac{96}{135}\cdot\pi^4 \\%\;(\EO)\\
[-\halfbls] &&&&&\\ \hline &&&&& \\ [-\halfbls]
\cQ(5, 3)&77/3&\text{0.488:0.383:0.129}&0.056 \cdot\pi^{6}&0.0588\cdot\pi^6&\frac{14}{243}\cdot{\pi^6} \\% \;(\EO)\\
[-\halfbls] &&&&&\\ \hline &&&&& \\ [-\halfbls]
\cQ(4^2)&92/3&\text{0.388:0.481:0.109:0.022}&0.8259 \cdot\pi^{4}&0.8000\cdot\pi^4&{\frac{4}{5}\cdot\pi^4} \\% \;(\EO)\\
[-\halfbls] &&&&&\\ \hline &&&&& \\ [-\halfbls]
\cQ^{reg}(9, -1)&385/3&\text{0.4569:0.4195:0.1236}&0.2972 \cdot\pi^{6}&&\\
[-\halfbls] &&&&0.3580\cdot\pi^6&\frac{15224}{42525}\cdot{\pi^6} \\%\;(\EO)\\ \cline{1-4} &&&&&\\ [-\halfbls]
\cQ^{irr}(9, -1)&55/3&\text{0.3024:0.5740:0.1236}&0.0642 \cdot\pi^{6}&&\\
  %----------------------------------------------------------------------
[-\halfbls] &&&&&\\ \hline
\end{array}
$$
\end{table}
   %
%\vspace*{-0.3cm}
\noindent
As before, we applied the method of Eskin--Okounkov to all strata,
and methods specific for genus $0$ and for hyperelliptic components when
applicable. For the non-varying strata ($\nonvar$)
we also used evaluation through Siegel--Veech constants.